\documentclass[oneside,12pt]{amsart}
\usepackage{amsaddr}
\usepackage[margin=2cm]{geometry}  
\usepackage{mathrsfs,amssymb,mathabx}
\usepackage[mathlines]{lineno}
\usepackage[dvipsnames]{xcolor}
\usepackage{graphicx}
\usepackage{enumerate}
\usepackage[shortlabels]{enumitem}
\usepackage{nicefrac}
\usepackage{soul}

\newenvironment{lesn}{\begin{linenomath}\begin{esn}}{\end{esn}\end{linenomath}}

\theoremstyle{plain}
\newtheorem{theorem}{Theorem}
\newtheorem{lemma}[theorem]{Lemma}
\newtheorem{corollary}[theorem]{Corollary}
\newtheorem{proposition}[theorem]{Proposition}

\theoremstyle{definition}
\newtheorem*{definition}{Definition}
\newtheorem*{remark}{Remark}

\makeatletter
\def\thmhead@plain#1#2#3{%
  \thmname{#1}\thmnumber{\@ifnotempty{#1}{ }\@upn{#2}}%
  \thmnote{ {\the\thm@notefont{\bf (#3)}}}}
\let\thmhead\thmhead@plain
\makeatother

\newcommand{\cec}{}
\newcommand{\ger}[1]{#1}
\newcommand{\bl}{}

\newcommand{\defin}[1]{{\bf #1}}
\newcommand{\mc}[1]{\ensuremath{\mathscr{#1}}}

\newcommand{\cond}[2]{\left.\vphantom{#2}#1\ \right| #2}

\newcommand{\e}{\ensuremath{\mathbf{e}}}

\newcommand{\set}[1]{\ensuremath{\left\{ #1\right\} }}

\newcommand{\paren}[1]{\ensuremath{\left( #1\right) }}
\newcommand{\bra}[1]{\ensuremath{\left[ #1\right] }}

\newcommand{\sip}{\mathbb{P}}

\newcommand{\ssa}{\ensuremath{\mathscr{F}}}

\newcommand{\eps}{\ensuremath{ \varepsilon}}
\newcommand{\na}{\ensuremath{\mathbb{N}}}

\newcommand{\z}{\ensuremath{\mathbb{Z}}}
\newcommand{\re}{\ensuremath{\mathbb{R}}}

\newcommand{\proba}[1]{\ensuremath{\sip\! \left( #1 \right)}}

\newcommand{\probac}[2]{\ensuremath{\proba{\cond{#1}{#2}}}}

\newcommand{\abs}[1]{\hspace{.25mm}\left|#1\right|\hspace{.25mm}}

\newcommand{\F}{\ssa}

\newcommand{\p}{\ensuremath{ \sip  } }

\newcommand{\sag}[1]{\sigma\!\paren{#1}}

\newcommand{\imf}[2]{\ensuremath{#1\!\paren{#2}}}

\newenvironment{esn}{\begin{equation*}}{\end{equation*}}

\title
[Monkey random walks]
{Random walks with preferential relocations and fading memory: a study through random recursive trees}
\author{C\'ecile Mailler}
\address{Department of Mathematical Sciences, University of Bath, UK}
\email{c.mailler@bath.ac.uk}
\author{Ger\'onimo Uribe Bravo}
\address{Instituto de Matem\'aticas\\ Universidad Nacional Aut\'onoma de M\'exico, M\'exico}
\email{geronimo@matem.unam.mx}

\thanks{
CM is grateful to EPSRC for support through the fellowship EP/R022186/1.
GUB's research supported by
CoNaCyT
grant FC-2016-1946 and UNAM-DGAPA-PAPIIT grant IN115217.
CM and GUB thank the Bath-UNAM-CIMAT (\href{http://buc.cimat.mx}{BUC})  research platform for the workshop \href{http://buc.cimat.mx/node/15}{BUC3}, during which the authors started their collaboration.
}

\newcommand{\mrw}[1]{\ensuremath{\imf{\mc{M}}{#1}}}

\newcommand{\fr}{\ensuremath{\varphi_{\text{run}}}}

\newcommand{\bs}{\ensuremath{\boldsymbol}}
\newcommand{\sss}{\ensuremath{\scriptscriptstyle}}

\usepackage
[pagebackref,colorlinks,citecolor=Plum,urlcolor=Periwinkle,linkcolor=DarkOrchid]
{hyperref}

\begin{document}
%\linenumbers
\maketitle 
\begin{abstract}
Consider a stochastic process that behaves as a $d$-dimensional simple and symmetric random walk, 
except that, with a certain fixed probability, at each step, 
it chooses instead to jump to a given site with probability proportional to the time it has already spent there. 
This process has been analyzed in the physics literature under the name \emph{random walk with preferential relocations}, where it is argued that the position of the walker after $n$ steps, scaled by $\log n$, converges to a Gaussian random variable; 
because of the $\log$ spatial scaling, the process is said to undergo a \emph{slow diffusion}. 

In this paper, we generalize this model by allowing the underlying random walk to be any Markov process and the random run-lengths (time between two relocations) to be i.i.d.-distributed. 
We also allow the memory of the walker to fade with time, meaning that when a relocations occurs, the walker is more likely to go back to a place it has visited more recently. 

We prove rigorously the central limit theorem  described above (plus a local limit theorem and the convergence of the weighted occupation measure) by associating to the process a growing family of vertex-weighted random recursive trees and a Markov chain indexed by this tree. 
The spatial scaling of our relocated random walk is related to the height of a ``typical'' vertex in the random tree. 
This typical height can range from doubly-logarithmic to logarithmic or even a power  of the number of nodes of the tree, depending on the form of the memory. 
%The limit theorem then follows by a central limit theorem for the Markov chain along each branch and the fact that the height of a typical vertex in our weighted tree with $n$ elements is of order $\log n$.
%This new approach allows us to prove stronger limit theorems such as convergence of the occupation measure and a local limit theorem.
\end{abstract}

\section{Introduction and statement of the results}
A \defin{random walk with preferential relocation} is a process heuristically defined as follows: it behaves as a (discrete-time) random walk on $\mathbb Z^d$ except at some (relocation) times at which it jumps to a value it has already visited with probability proportional to its number of visits there. 

This model has been proposed in the physics literature by Boyer and Solis-Salas~\cite{PhysRevLett.112.240601} as a model for animal foraging behaviour, and they argue that ``the model exhibits good agreement with data of free-ranging capuchin monkeys''.
%this model fits measured data  in the particular case of capuchin monkeys in . 
%A central limit theorem was obtained for the position 
In the case when the underlying random walk is a simple random walk on $\mathbb Z$
and the times between relocations are i.i.d.\ geometric random variables,
they show that this ``monkey walk'' satisfies a central limit theorem by calculating moments.
This central limit theorem is similar to the central limit theorem of the simple random walk, except that the appropriate (standard-deviation) scaling is $\sqrt{\log n}$ instead of $\sqrt{n}$ as in the simple-random-walk case. 
Boyer and Solis-Solas argue that ``The very slow growth [of the variance of the position of the monkey with time] in our model agrees qualitatively with the fact that most animals have limited diffusion or home ranges". 
In a subsequent paper, Boyer and Pineda~\cite{PhysRevE.93.022103} generalize the results of ~\cite{PhysRevLett.112.240601}
to the case when the underlying random walk has power-law-tailed increments; a limit theorem with stable limit is exhibited in that case. 
Finally, the model can also be extended by adding the effect of memory, so that the walker prefers to visit more recent places: see, e.g.,~\cite{doi:10.1088/1742-5468/aa58b6} and~\cite{FalconEtAl17} where a trapping site is added to the model.

\medskip
{\bf The main contribution of this paper} is two-fold: firstly, we prove rigorously
the recent results of~\cite{PhysRevLett.112.240601,PhysRevE.93.022103,doi:10.1088/1742-5468/aa58b6}; 
secondly, the robustness of our proof techniques allow 
us to obtain results for a much greater range of models.
Indeed, we allow the time to be discrete or continuous, 
the underlying process to be any Markov process on $\mathbb R^d$, 
and the times between relocations (also called ``run-lengths'') 
to have other distributions than the geometric one as long as they are i.i.d.\ (and satisfy moment conditions). 
As in~\cite{doi:10.1088/1742-5468/aa58b6}, 
we are able to generalise the model further by allowing 
the memory of the walker to 
decay with time:
the probability to relocate to a visited site depends 
on how long ago this site was visited.
We are able to treat a large class of ``memory kernels'' (the function that models the memory of the walker),
comprising most of the cases introduced 
by~\cite{doi:10.1088/1742-5468/aa58b6} 
as well as some new interesting cases.

Under moment assumptions on the moments of the run-lengths, 
we prove that, if the original Markov process verifies a central limit theorem (we actually allow for more general distributional limit theorems that we call ``ergodic limit theorems''), then the version of this process with reinforced relocations does too. %We prove similar results for local limit theorems.
We are also able to prove that, if the original Markov process verifies an 
ergodic limit theorem, then the weighted occupation measure of the process with relocations converges in probability (for the weak topology) to a deterministic limit.

The crucial step in our proof is reducing the study of these random walks to the analysis of a new model of random trees that we call the weighted random recursive tree ({\sc wrrt}); the height of a typical vertex in these trees is responsible for the $\log$-scaling exhibited in \cite{PhysRevLett.112.240601}. This key step allows us to state our results in full generality, without having to specify the underlying Markov process we consider.

This new approach allows us to prove a local limit theorem for the process with random reinforced relocations.
%other asymptotic theorems; we state a local limit theorem and a large deviation principle for the process with random reinforced relocations.
Finally, Boyer and Solis-Salas \cite{PhysRevE.93.022103} ask in their paper whether the monkey walk is recurrent in all dimensions; we answer this question both when the underlying process (process without the relocations) is the simple symmetric random walk, and when the underlying process is the Brownian motion.

\medskip
{\bf Discussion of the context of related random walks.}
Two principle features of the monkey walk are that it is non-Markovian 
(the walker needs to remember all its past) and that it involves reinforcement, and
the non-Markovianity is the main difficulty in the rigorous analysis.
The Monkey walk is a rare example of a non-Markovian random walk where one can obtain precise asymptotic results, with a great level of generality (all dimensions, discrete- or continuous-time, no need to specify the underlying Markov process, etc). We now review other non-Markovian reinforced random walks studied in the literature.

The most famous, and most extensively-studied,
non-Markovian random walk with reinforcement is
the reinforced random walk introduced by Coppersmith and Diaconis in 1987.
The dynamics of this model are very different from the monkey walk; indeed, in this model, every edge of the lattice has originally weight~$\alpha$, and the walker starts at the origin; at every step, the walker crosses a neighbouring edge with probability proportional to the weight of the edge, and the weight of the crossed edge is increased by~$1$. Different variant of this edge-reinforced process (and its vertex-reinforced version) have been studied in the literature on the $d$-dimensional lattice 
%(see, e.g.~\cite{Davis90}) 
as well as on arbitrary finite and infinite graphs%
% (see, e.g.~\cite{Volkov01},~\cite{BenaimTarres11} and~\cite{BenaimRaimondShapira12})
, and in particular on trees; %(see, e.g.~\cite{Pemantle88} and~\cite{Benaim97}).
see~\cite{Pemantle07} for a survey.

Two other models of non-Markovian random walks with reinforcement are the so-called elephant random walk, introduced in the physics literature by \cite{PhysRevE.70.045101}, and the shark walk introduced very recently by~\cite{2017arXiv171005671B}. We now briefly describe both these random walks to emphasise how they are very different from the ``monkey walk'' studied in this paper.
%Sch\"utz and Trimper~\cite{ST04} 

The elephant walk on $\mathbb Z$ depends on a parameter $p\in[0,1]$ and has jumps of size $1$. At every time-step, 
the elephant chooses a uniform time in its past, and, with probability~$p$ 
takes the same step it took at this random time, or, with probability $(1-p)$, 
decides at random which direction to follow. 
%Baur and Bertoin
\cite{PhysRevE.94.052134} 
showed that the elephant random walk can be successfully studied by using 
the theory of generalized P\'olya processes (and in particular the results of \cite{MR2040966}):
for $p\leq \nicefrac34$, the elephant walk is diffusive, and for $p>\nicefrac34$, it is super-diffusive.
An alternative proof was developed independently by \cite{MR3652225}, and 
\cite{2017arXiv170907345B} generalized these results to higher dimensions.
The shark walk is a generalization of the elephant random walk except that the steps taken by the walker have heavy tails.

The elephant random walk was so named because of the popular belief that elephants have a long memory; this name thus refers to the fact that the walk is non-Markovian but, as far as we know, there is no link between this random walk and the actual behaviour of free-ranging elephants.

In contrast, the shark walk and the ``monkey walk'' 
studied in this paper have a closer link to actual animal behaviour.
The shark random walk was named because, according to the physics literature (see, e.g.~\cite{SimsEtAl08}), marine predators' foraging behaviour exhibits similarities with random walks whose standard deviation grows faster than the square root of the time.
Similarly, \cite{PhysRevLett.112.240601} argue (using measured-data of capuchin monkey) that monkeys' foraging behaviour, and in particular the number of distinct sites visited by the monkey as a function of time, exhibits similarities with the monkey walk.

%we call the model studied in this paper the ``monkey random walk'', as~\cite{PhysRevLett.112.240601} and the references therein argue that the foraging behaviour of monkeys exhibits similar features as slow diffusions.
%When studying animal behaviour, a natural quantity to measure is the diameter of their explored territory as a function of time. When this diameter grows like the square root of the time, the walk is described as {\it diffusive}. When it grows faster, the walk is said to undergo a {\it anomalous diffusion}; when slower, it is said to undergo a {\it slow diffusion}.
%When studying animal foraging a
%The shark random walk was named because, according to the physics literature, marine predators' foraging behaviour exhibits similarities with anomalous diffusions (see, e.g.~\cite{SimsEtAl08}). Similarly, we call the model studied in this paper the ``monkey random walk'', as~\cite{PhysRevLett.112.240601} and the references therein argue that the foraging behaviour of monkeys exhibits similar features as slow diffusions.

\smallskip
Aside from non-Markovianity and reinforcement, the third main feature of our monkey walk is the fading memory. Although it is natural and often studied in the physics literature (see, e.g.\ \cite{Alves, Moura} for the elephant random walk with decaying memory,~\cite{PhysRevLett.112.240601} for the monkey walk, \cite{doi:10.1088/1742-5468/aa58b6} for diffusions with random reinforced relocations), we are not aware of any probabilistic result about non-Markovian random walks with decaying memory. The model studied in~\cite{vdH01} is the (deterministic) mean-field version of a weakly self-interaction random walk where we allow the trajectory to self-intersect with some probability that increases with the time between the two visits at this intersection.

\smallskip
Finally, we mention random walks with stochastic restarts; although they are Markovian and non-reinforced, their definition is similar to the monkey walk except that at relocation times, the walker always goes back to where it started. The literature (see e.g.~\cite{EvansMajumdar11, EvansMajumdar14}) mainly focuses on their hitting times properties as these walks are used as search algorithms.

\subsection{Definition of the model}
We consider time to be either discrete or continuous, i.e.\ $\mc{T}$ equal  to $\na$ or  $[0,\infty)$. 
The ``monkey {\cec} process'' $X=(X(t))_{t\in\mc T}$ is a stochastic process that depends on three parameters: a semi-group $P=(P_t)_{t\in \mc{T}}$, 
a probability distribution $\fr$ on $\mc T$ such that $\fr(\{0\}) =0$,
and a function $\mu : \mc T\to \mathbb R^+$ called the memory kernel.

Let $Z$ be a Markov process of semi-group $P$ on $\re^d$;
for all $x\in\mathbb R^d$, we denote by $\p_x$ the law of $Z$ started at $x$, 
and by $\p_x^t$  the law of $\paren{Z(s),s<t}$ under $\p_x$.

Let $\bs L = \paren{L_i, i\geq 1}$ be a sequence of i.i.d.\ random variables {\cec with} distribution~$\fr$. We let $T_0 = 0$, and, for all $n\geq 1$, $T_n = \sum_{i=1}^n L_i$, and call these random elements of $\mc T$ the ``relocation times''.

\begin{figure}
\begin{center}
\includegraphics[width=9.5cm]{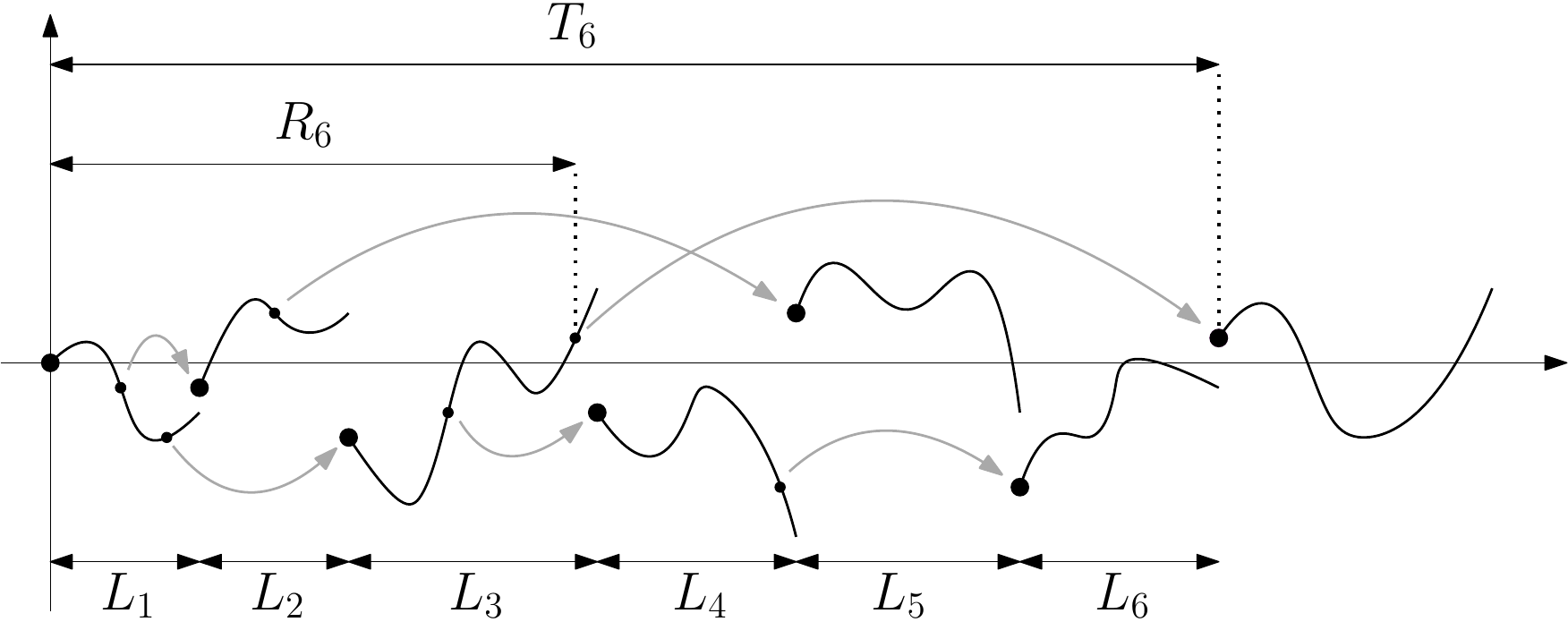}
\end{center}
\caption{The monkey {{\cec}} process: an illustration of the definition in dimension~1. Time is on the horizontal axis, position on the vertical axis. The grey arrows point from $X(R_n)$ to $X(T_n)$, which are equal, by definition, for all $n\geq 1$.
}
\label{fig:def}
\end{figure}

We define the monkey {\cec} process recursively as follows (the notations are illustrated in Figure~\ref{fig:def}):
Choose a starting point $X(0)\in\re^d$, and draw $(X(s)\colon s< T_1)$ at random according to~$\mathbb P_0^{T_1}$.
For all $n\geq 1$, given $(X(s)\colon s<T_n)$,
\begin{itemize}
\item draw a random variable $R_n$ in $[0, T_n)$ according to the following distribution\footnote{In the whole paper, if $\mc T = \{0,1,\ldots\}$, we use the integral notation for a sum: $\int_x^y \mu(u)\mathrm du := \sum_{k=x}^{y-1} \mu(k)$.}
\begin{equation}\label{eq:def_Rn}
\mathbb P(R_n\leq x \mid T_n)= \frac{\int_0^x \mu(u)\,\mathrm du}{\int_0^{T_n} \mu(u)\,\mathrm du};
\end{equation}
\item set $V_{n+1} = X(R_n)$ and draw $(X(s)\colon T_n\leq s<T_{n+1})$ according to $\mathbb P_{V_{n+1}}^{L_{n+1}}$.
\end{itemize}
We say that~$X$ is the {\cec} monkey process with preferential relocation and fading memory  -or equivalently, monkey {\cec} process- of semigroup $P$, run-length distribution $\fr$ (being the distribution of the $L_i$'s) and memory kernel $\mu$, and denote this by $X=\mrw{P,\fr,\mu}$.

\bigskip
\noindent{\bf Choice of the memory kernel:} 
The techniques of this paper are robust in the choice of the memory kernel, 
but they become more explicit when its integral  is available in closed form. 
In this paper, we consider the following  two parametric families of memory kernels to cover a large range of phenomena: 
for all $x\in\mc T$,
\begin{equation}\label{eq:memory_kernel}
\mu_1(x) = \frac{\alpha}{x}\log (x)^{\alpha-1}\mathrm e^{\beta\log (x)^\alpha},
\end{equation}
where $\alpha>0$ and $\beta\geq 0$ (set $\mu_1(x)=0$ if $x=0$) and
\begin{equation}\label{eq:memory_kernel2}
\mu_2(x) = \gamma\delta x^{\delta-1}e^{\gamma x^\delta},
\end{equation}where  $\delta\in (0,1/2]$ (for technical reasons to be discussed {\cec later}) and $\gamma>0$.

%The function $\mu_1$ is positive on its all domain of definition.
Note that, when $\alpha=\beta = 1$, the memory kernel $\mu_1$ is uniform, 
meaning that, at each relocation time, the walker chooses a time uniformly in its past, and relocates to where it was at that random time; this is the original process introduced by Boyer and Solis-Salas~\cite{PhysRevE.93.022103}.
When $\alpha>1$, the memory kernel increases, which indeed corresponds to the idea of a fading memory. When $\alpha<1$, the memory kernel decreases; although this case is less natural when thinking about the memory of a walker, we include it in the study since this is done at no additional cost. 
When $\beta=0$ and $\alpha=1$, we get $\mu_1(x)= \nicefrac1x$, which corresponds to Case~1 in \cite{doi:10.1088/1742-5468/aa58b6} and when $\alpha=1$ and $\beta\neq 0$ then $\mu_1(x)= x^{\beta-1}$, corresponding to Case~$2$ of \cite{doi:10.1088/1742-5468/aa58b6}. 
Finally, $\mu_2$ coresponds to Case~3 of of \cite{doi:10.1088/1742-5468/aa58b6}. 
%These particular memory kernels  cover cases 1 to 3 of  \cite{doi:10.1088/1742-5468/aa58b6}. 
The memory kernel $\mu_1$ with $\alpha\neq 1$ is not considered in~\cite{doi:10.1088/1742-5468/aa58b6}, and thus allows us to cover an even larger range of phenomena.

\subsection{Central limit theorem}
%The memory kernel impacts the spatial scaling of $X$ in the following manner. 
The first result of this paper states that, if $(Z(t)-b_t)/a_t$ converges weakly, then so does $(X(t)-b_{\kappa_2 s(t)})/a_{\kappa_2 s(t)}$ where, $\kappa_2 = \mathbb E L^2/(2\mathbb E L)$, and where, for all $x\in[0,\infty)$,
\begin{equation}\label{eq:def_s}
s(x)=\begin{cases}
\log(x)^\alpha& \text{with memory kernel $\mu_1$ when $\beta\neq 0$,}\\
\alpha \log\log(x)& \text{with memory kernel $\mu_1$ when $\beta= 0$,}\\
\gamma x^\delta& \text{with memory kernel $\mu_2$}.
\end{cases}
\end{equation}
\ger{Note that $b_t$ gives the deterministic drift (or bias) of $Z$, while $a_t$ gives the asymptotic size of the fluctuations of $Z$ around this deterministic bias; they would equal the asymptotic mean and standard deviation of $Z$ in the context of the central limit theorem.}
{\cec In particular, our result says that, if the standard deviation of $Z(t)$ is proportional to $a_t$, then the standard deviation of $X(t)$ is of asymptotically proportional to $a_{\kappa_2 s(t)}$.}
\ger{In other words, the monkey process $X$ behaves as $Z$ taken at time $\kappa_2s(t)$; the parameter $s(t)$ can be thought of as a time-change which slows down the order of magnitude of $X$ when compared to $Z$. This idea of time-change is formalized in our proof where we show that, for each fixed $t$,  $X(t)\stackrel{\sss d}{=}Z(S(t))$, where $S(t)$ is a random time-change.} {\cec The rest of the proof then relies on proving that $S(t)\approx \kappa_2 s(t)$ with sufficient accuracy.}
Note that~$\kappa_2$ only depends on the run-length distribution, while $s({\cec x})$ only depends on the memory kernel.
% or $\mathbb{E} L^3/3(\mathbb{E} L)^{\delta+1}$ depending on the choice of memory kernel. 
The precise statement of our first result uses the following notion that generalizes the central limit theorem:
\begin{definition} 
We say that {\cec a Markov process} $Z$ is $\paren{a_t,b_t}$-ergodic if $\paren{Z(t)-b_t}/a_t$ converges in distribution under $\p_x$ to a law $\gamma$ that does not depend on $x$. We call $\gamma$ the limiting distribution of~$Z$. 
\end{definition}

{\cec If $Z$ is $\paren{a_t,b_t}$-ergodic, the function $(t\mapsto b_t)$ gives the deterministic drift (or bias) of $Z$, while $(t\mapsto a_t)$ gives the size of the random fluctuations of~$Z$ around this deterministic bias at time $t$. Examples of $(a_t,b_t)$-ergodic Markov process are given in Section~\ref{sub:examples}.}
%\ger{Let us just mention that the time-change $S(t)$ is $(s(t), \sqrt{s(t)})$-ergodic with a Gaussian limiting distribution. }

\begin{theorem}
\label{monkeyMarkovLimitTheorem}
Let $X=\mrw{P,\fr,\mu}$ where $P$ is a semi-group on $\mathbb R^d$, $\fr$ a probability distribution on $\mc T$, and $\mu$  is any one of $\mu_1$ or $\mu_2$. Let $L$ be a random variable of distribution $\fr$, and $Z$ a Markov process of semi-group~$P$.
Assume that 
\begin{enumerate}[(i)]
\item $\mathbb E L^8<\infty$, 
%$\mathbb E L^4$ and $\mathtt{Var}(L^4)<\infty$. \comment{Should change to $\mathbb E L^8<\infty$?}
\item there exists two functions $(a_t, b_t)_{t\in\mc T}$ such that the Markov process~$Z$ is $\paren{a_t, b_t}$-ergodic 
with limiting distribution $\gamma$, and
\item for all $x\in\re$, the following limits exist and are finite:
\begin{description}
\item[(Scaling)] 
$\displaystyle f(x)=\lim_{t\to\infty} \frac{a_{t+x\sqrt{t}+\eps_t}}{a_t}$ and 
$\displaystyle g(x)=\frac{b_{t+x\sqrt{t}+\eps_t}-b_t}{a_t}$ whenever $\eps_t=\imf{o}{\sqrt{t}}$. 
\end{description}
\end{enumerate}
Then, in distribution\footnote{In the whole paper, we use the notation $\xrightarrow{\sss d}$ to denote convergence in distribution.} when $t\to\infty$, we have
\begin{equation}\label{eq:clt_monkey}
\frac{X(t)-b_{\kappa_2 s(t)}}{a_{\kappa_2 s(t)}}\xrightarrow{\sss d}
\imf{f}{\Omega}\Gamma
+\imf{g}{\Omega},
\end{equation}where 
$\kappa_i = \mathbb E L^i/(i\mathbb E L)$ for $i\in\{2,3\}$, 
$\Omega\sim \mc N(0,\nicefrac{\kappa_3}{\kappa_2})$ and $\Gamma\sim\gamma$ are independent.
\end{theorem}

%
%The proof of this result 
%%(and all other results of this paper) 
%relies on the fact that,
%in distribution, $X(t)$ is equal to $Z(S(t))$, where $S(t)$ is a random time, independent of $Z$, and such that $\big(S(t) - \kappa_2s(t)\big)/ \sqrt{\kappa_3s(t)}\xrightarrow{\sss d}\Lambda\sim\mathcal N(0,1)$. 
%By assumption, $(Z(S(t))-b_{S(t)})/a_{S(t)}\xrightarrow{\sss d} \Gamma$.
%The result follows by ``replacing $S(t)$ by $\kappa_2 s(t)$''
%in $a_{S(t)}$ and $b_{S(t)}$ in this limit theorem;
%the terms $\imf{f}{\Omega}$ and $\imf{g}{\Omega}$ (note that $\Omega = \Lambda\sqrt{\nicefrac{\kappa_3}{\kappa_2}}$) come from controlling the error made by doing so.

{\cec 
\begin{remark}[Discussion of Assumption (i)] 
Proving that $\big(S(t) - \kappa_2s(t)\big)/ \sqrt{\kappa_3s(t)}\xrightarrow{\sss d}\Lambda\sim\mathcal N(0,1)$ is done by applying the strong law of large numbers and central limit theorems to sum of independent (but not identically distributed random variables). To apply these limit theorems, one needs assumptions on the moments of the summands; this is where Assumption (i) is used. The forthcoming Lemmas \ref{lem:exp_Fi}, \ref{lem:sums} and \ref{lem:sumsF_iW_i} make this technical point very precise.
\end{remark}
}

\begin{remark}[Discussion on the value of $\delta$ when $\mu=\mu_2$]
We prove that Theorem~\ref{monkeyMarkovLimitTheorem} holds for $\delta=\nicefrac12$, {\cec but our proof cannot be easily generalized to} larger $\delta$ except if the underlying Markov process is such that $f\equiv 1$ and $g\equiv 0$. 
{\cec The reason is that, when $\delta<\nicefrac12$, the distribution of $\Omega$ in  Theorem~\ref{monkeyMarkovLimitTheorem} is a priori unknown.}
Boyer, Evans and Majumdar~\cite{doi:10.1088/1742-5468/aa58b6} argue that a central limit theorem holds for a (standard) Brownian motion with relocations and memory kernel $\mu_2$, for all $\delta\in(0,1]$; in this case, we indeed have $f\equiv 1$ and $g\equiv 0$. More complicated phenomena are expected when $\mu=\mu_2$ and $\delta>\nicefrac12$ and when the functions $f$ and $g$ are non-trivial; 
we leave this as an open problem.
\end{remark}

{\cec
\begin{remark}[A quenched version of this result]
We actually prove a stronger, {\it quenched}, version of this result: Denote by $\bs L = (L_i)_{i\geq 1}$ the run-lengths of the monkey process~$X$. Then, under the assumptions of Theorem~\ref{monkeyMarkovLimitTheorem}, Equation~\eqref{eq:clt_monkey} holds conditionally on $\bs L$, $\bs L$-almost surely.
\end{remark}
}

%\ger{
%\begin{remark}
%The relocation times $T_i$ naturally form a regenerative process, 
%whose deterministic asymptotic behavior $T_i\sim i\esp{L_1}$ follows from a finite first moment assumption. 
%An $8$-th moment assumption, might therefore seem at first sight unnatural. 
%However, note that we need to study the relocation times $R_i$, where the memory kernel intervenes and amplifies the sizes of the times between relocation. 
%The method of proof then leads to the need of replacing random series of the type $\sum_{i=1}^n L_i^b/i$ by their deterministic asymptotic form $\esp{W_1^b}\log n$ under the assumption of a finite $b$-th moment. 
%The choice of memory kernels implies the need to consider these sums for $b\leq 4$. 
%On the other hand, we also need to control the error incurred by the replacement, 
%which is done by computation of variances, and therefore explains the need of a finite $8$-th moment. 
%The forthcoming Lemmas \ref{lem:exp_Fi}, \ref{lem:sums} and \ref{lem:sumsF_iW_i} make this technical point very precise. 
%\end{remark}} 

\subsection{Examples} \label{sub:examples}
In this section, we briefly show how to apply Theorem~\ref{monkeyMarkovLimitTheorem} to different Markov semi-groups $P$.
In all examples, we assume that the run-lengths $(L_i)_{i\geq 1}$ satisfy the assumptions of Theorem~\ref{monkeyMarkovLimitTheorem}.

\medskip
{\bf The simple random walk case --}
Assume that $P$ is the semi-group of a simple random walk (say on $\mathbb R$) 
whose increments have finite mean $\mu$ and finite variance $\sigma^2$.
In that case, the central limit theorem gives that
\[\frac{Z(t) - \mu t}{\sigma\sqrt t} \xrightarrow{\sss d} \imf{\mc{N}}{0,1};\]
in other words, the simple random walk $Z$ is $(\sigma\sqrt t, \mu t)$-ergodic with limiting distribution $\gamma = \mc N(0,1)$.
Therefore, $f(x) = 1$, and $g(x) = \nicefrac{\mu x}{\sigma}$ for all $x\in\mathbb R$, 
and Theorem~\ref{monkeyMarkovLimitTheorem} implies that
\[\frac{X(t) - \mu\kappa_2 s(t)}{\sigma\sqrt{\kappa_2s(t)}} 
\xrightarrow{\sss d}  \imf{\mc{N}}{0,1+\frac{\mu^2 \kappa_3}{\sigma^2\kappa_2}}
\quad\Leftrightarrow\quad
\frac{X(t) - \mu\kappa_2 s(t)}{\sqrt{s(t)}} 
\xrightarrow{\sss d}  \imf{\mc{N}}{0,\sigma^2\kappa_2+\mu^2 \kappa_3}
,\]
or, equivalently,
\begin{equation}
\label{eq:MonkeyWalkCLT}
\frac{X(t)- \mu\kappa_2s(t)}{\sqrt{(\sigma^2\kappa_2 + \mu^2\kappa_3)s(t)}} \xrightarrow{\sss d}  \imf{\mc{N}}{0,1}.
\end{equation}
{\cec From the definition of $s(t)$ (see Equation~\eqref{eq:def_s}), we can see that this monkey process is diffusive only when $\mu = \mu_2$ and $\delta = \nicefrac12$; it is otherwise sub-diffusive.}

We can apply Equation~\eqref{eq:MonkeyWalkCLT} to the particular case studied by~\cite{PhysRevLett.112.240601}: In that case, the $L_i$'s are geometric of parameter $q\in(0,1)$, 
implying that $\kappa_2 = (2-q)/(2q)$. 
Moreover, the memory kernel is $\mu_1$, with $\alpha=\beta=1$, $\mu=0$ and $\sigma^2 = 1$, implying that
\[\frac{X(t)}{\sqrt{(2-q)\log t/(2q)}} \xrightarrow{\sss d} \mc N(0,1),\]
{\cec as claimed in~\cite{PhysRevLett.112.240601}.}

\medskip
{\bf The random walk case with heavy-tailed increments --}
Assume that $P$ is the semi-group of a simple random walk
whose real-valued increments $(\Delta_i)_{i\geq 1}$ are independent copies of~$\Delta$,
%have finite mean $\mu$, and verify $\mathbb P(\Delta\geq u)\sim u^{\beta}\ell(u)$, \comment{Finite mean is incompatible with slowly varying tails of index $\alpha\in (0,1)$. Adjust this case for example by requiring the jump distribution to be concentrated on nonnegative reals. }
and satisfy
\[
\mathbb P(\Delta\geq u)\sim c_+ u^{{\cec -}\omega}\ell(u)
\quad\text{and}\quad
\mathbb P(\Delta\leq -u)\sim c_- u^{{\cec -}\omega}\ell(u)
\]
as $u\to\infty$,
where $1<\omega <2$, and $\ell$ is slowly varying at infinity. 
Then, there exist a slowly varying function $\tilde \ell$ such that, if we set $a_t = t^{\nicefrac1\omega}\tilde \ell(t)$, $b_t = 0$ if $\omega<1$ and $b_t = \mu t$ otherwise, then we have
\[\frac{Z(t)- b_t}{a_t} \xrightarrow{\sss d} \Phi_{\omega},\]
where $\Phi_{\omega}$ is some $\omega$-stable random variable.  
In both cases ($0<\omega<1$ and $1<\omega<2$), 
$f \equiv 1$, and $g \equiv 0$; therefore, 
Theorem~\ref{monkeyMarkovLimitTheorem} gives
\begin{linenomath}
\begin{align*}
\frac{X(t)}{s(t)^{\nicefrac1\omega}} 
&\xrightarrow{\sss d} \Phi_\omega, \text{ if } 0<\omega<1,\\
\frac{X(t) - \mu s(t)}{s(t)^{\nicefrac1\omega}} 
&\xrightarrow{\sss d} \Phi_\omega,  \text{ if } 1<\omega<2.
\end{align*}
\end{linenomath}
{\cec From the definition of $s(t)$ (see Equation~\eqref{eq:def_s}), we see that this monkey process is sub-diffusive when $\mu = \mu_1$; it is also sub-diffusive when $\mu=\mu_2$ and $\delta < \nicefrac\omega2$. It is diffusive when $\mu = \mu_2$ and $\delta = \nicefrac\omega2$, and becomes super-diffusive when $\delta \in (\nicefrac\omega2, \nicefrac12]$.}

\medskip
{\bf The Brownian motion case --}
Assume that $P$ is the semi-group of the Brownian motion on~$\mathbb R$ with drift $c\in\mathbb R$. Therefore, $(Z(t) - ct)/{\sqrt t} \sim \imf{\mc{N}}{0,1}$, implying that $a_t = \sqrt t$ and $b_t = ct$ for all $t\in[0,\infty)$.
Therefore, $f(x) = 1$ and $g(x) = cx/\sqrt \kappa_2$ for all $x\in\mathbb R$, and Theorem~\ref{monkeyMarkovLimitTheorem} gives
\[\frac{X(t) - c\kappa_2s(t)}{\sqrt{\kappa_2s(t)}} 
\xrightarrow{\sss d}  \imf{\mc{N}}{0,1+\frac{c^2 \kappa_3}{\kappa_2}}
\quad\Leftrightarrow\quad
\frac{X(t) - c\kappa_2s(t)}{\sqrt{s(t)}} 
\xrightarrow{\sss d}  \imf{\mc{N}}{0,\kappa_2+c^2 \kappa_3}.\]
{\cec Like in the first example, this monkey process is diffusive only if $\mu=\mu_2$ and $\delta=\nicefrac12$; it is otherwise sub-diffusive.}

%{\color{magenta} Do we need more exotic examples?} 

\medskip
\begin{remark}[Discussion on the possible values for the function $f$] The definition of $(a_t,b_t)$-ergodicity is related to Lamperti's assumption that $((Z(xt)-b_t)/a_t)_{x\geq 0}$ converges in the sense of finite-dimensional distributions to a process $(Z^\infty(x))_{x\geq 0}$. 
\ger{Note that we only assume convergence of the one-dimensional distributions.}
Lamperti's assumption was introduced in  \cite{MR0138128} and implies that  $Z^\infty$ is a self-similar process, that $a$ and $b$ are regularly varying and that therefore the function $f$ defined in Theorem~\ref{monkeyMarkovLimitTheorem} equals~$1$. 
In particular, $f=1$ when $Z$ is a random walk in the domain of attraction of stable law (which includes the Gaussian case). In the particular case of non-negative Markov chains, invariance principles with self-similar limits can be found in \cite{MR2854770} and \cite{MR3543905}. 
More complex functions $f$ can arise when the process $Z$ is not self-similar. For example, recall that if $Y$ is the standard Yule process, then $Y(t)/e^t$ converges to a standard exponential random variable; this implies that $Z(t):=Y(\sqrt{t})$ is $(\e^{\sqrt t},0)$-ergodic. In that case, $f$ is equal to $x/(2\sqrt{\kappa_2})$, but since $Z$ is not homogeneous, its version with reinforced relocation does not fit in our framework.
Also note that the function $g$ can be arbitrary, 
which we illustrate again, for simplicity,  with a non-homogeneous Markov process by taking $Z(t)=B(t)+g(t)$ where $B$ is the standard Brownian motion.
\end{remark}

\subsection{Convergence of the occupation measure}
Under the assumptions of Theorem~\ref{monkeyMarkovLimitTheorem}, we can actually prove a stronger result; namely convergence of the {weighted} occupation measure {\cec of $X$} to a deterministic limit. 
{\bl The {weighted} occupation measure of the process $X$ on $[0,t]$ is defined as follows: let\footnote{\cec In the whole paper, we sometimes omit the variable and differential from the integrals: $\int \mu := \int \mu(u)\mathrm du$.} $\bar\mu(x) = \int_0^x \mu$, for all $x\geq 0$ and 
\begin{lesn}
\imf{\pi_t}{\mathcal B}=\frac1{\bar\mu(t)}\int_0^t {\mu(s)}\mathbf{1}_{X_s\in \mathcal B}\, ds
\end{lesn}for all Borel set $\mathcal B\subseteq {\re^d}$.
{When $\mu\equiv 1$,} for all Borel set $\mathcal B$, $\pi_t(\mathcal B)$ is the proportion of the time spent by the monkey process in $\mathcal B$ until time $t$.
 {When $\mu\not\equiv 1$ this time spent in $\mathcal B$ is weighted according to~$\mu$.} Note that $\pi_t$ is a random probability measure on $\re^d$; it is the conditional distribution (given the trajectory up to time $t$) of the position the walker would relocate to if a relocation event would happen at time~$t$. Therefore, $\pi_t$ can be thought of as the memory that the walker has of its trajectory up to time~$t$.
We prove that $\pi_t$ converges in distribution to a deterministic probability distribution $\pi_{\infty}$ on the space of all probability distributions on $\mathbb R^d$ equipped with the weak topology; in other words, for all continuous and bounded function $\varphi$ from $\mathbb R^d$ to $\mathbb R$,
\[\int_{\mathbb R^d}\varphi(x) \mathrm d\pi_t(x) 
\xrightarrow{\sss p} \int_{\mathbb R^d} \varphi(x) \mathrm d\pi_\infty(x),\]
when $t\to\infty$.\footnote{In the whole paper, we use $\xrightarrow{\sss p}$ to denote convergence in probability.}
}
\begin{theorem}\label{th:cv_occupation}
Let $\pi_\infty$ be the distribution of the random variable $f(\Omega)\Gamma + g(\Omega)$.
{\cec Assume that the assumptions of Theorem \ref{monkeyMarkovLimitTheorem} hold, and that, additionally, $\delta\in(0,\nicefrac12)$ when $\mu=\mu_2$.}
Then,  $\imf{\pi_t}{a_{\kappa_2s(t)}\cdot +b_{\kappa_2s(t)}}\xrightarrow{\sss p} \pi_\infty$ when $t\to\infty$, on the set of probability measures on $\re^d$ equipped with the weak topology.
\end{theorem}

{\cec 
\begin{remark}[Idea of the proof and discussion on the values of $\delta$ when $\mu=\mu_2$]
The proof of this theorem relies on the fact that $\pi_t$ is approximately equal to the distribution of the position of the monkey process at the last relocation before $t$. To prove that a random probability distribution converges in probability (for the weak topology), it is enough to prove that (1) a random variable sampled according to this probability distribution converges in distribution and (2) two random variables sampled independently according to this random distribution are asymptotically independent. When the memory is too steep ($\mu=\mu_2$ and $\delta\geq \nicefrac12$), our proof does not work, and (2) may not be true.
\end{remark}}

\subsection{A local limit theorem}
The main idea of the paper is to associate a labelled  tree to the monkey Markov process (its branching structure) to transport properties of the underlying Markov process~$Z$ to the process with random relocations~$X$. 
Heuristically, {\it if~$Z$ renormalized by some functions of $t$ satisfies an asymptotic property, then~$X$ satisfies the same asymptotic property once renormalized by the same functions but applied to $\kappa_2s(t)$ instead of $t$.}
In this section we illustrate this principle with the following local limit theorem, for which the state-space is taken to be~$\mathbb Z^d$. 
%\comment{Should we add that the statespace is now specialized to the integers?}
%by stating two additional results: the first is about a local limit theorem, the second is about a large deviation principle.

\begin{theorem}[Local limit theorem]\label{th:LLT}
We assume that the assumptions of Theorem~\ref{monkeyMarkovLimitTheorem} hold, that $Z^d$ takes values in~$\mathbb Z^d$ and $X(0)\in\mathbb Z^d$, and that, furthermore, $f\equiv1$. 
If $\gamma$ has a bounded and Lipschitz density function~$\phi$ and 
\[\sup_{m\in \mathbb Z^d}a_t\left|
\mathbb P(Z(t)=m) - \frac{1}{a_t}\phi\Big(\frac{m - b_t}{a_t}\Big)
\right| \to 0, \text{ when }t\to\infty,\]
then,
\[\sup_{m\in \mathbb Z^d}a_{\kappa_2s(t)}\left|
\mathbb P(X(t)=m) - \frac{1}{a_{\kappa_2s(t)}}\psi\Big(\frac{m - b_{\kappa_2s(t)}}{a_{\kappa_2s(t)}}\Big)
\right| \to 0,\]
when $t\to\infty$,
where $\psi$ is the density function of 
$\Gamma+g(\Omega)$, 
where $\Omega\sim\mc N(0,\nicefrac{\kappa_3}{\kappa_2})$ and $\Gamma\sim\gamma$ are independent.
\end{theorem}

Note that the function $\psi$ can be taken equal to $\mathbb E[\imf{\phi}{x-g(\Omega)}]$ and that it is also Lipschitz and bounded. 

\subsection{Recurrence}
When considering the simple random walk or the Brownian motion on $\mathbb Z^d$ (resp.\ $\mathbb R^d$), it is natural to ask whether or not the walker started from zero will come back to zero (resp.\ a neighbourhood of zero) almost surely in finite time; we then call the process ``recurrent'' (resp.\ ``neighborhood recurrent''). Boyer and Solis-Salas ask this question for the monkey walk; since it is slowly diffusive, one can imagine that the monkey walk is recurrent in all dimensions. 
We are able to prove that the answer depends on the choice of the memory kernel:
\begin{theorem}\label{th:rec}
Assume that $X(0)=0$ and $X=\mc M(P, \fr, \mu)$, where $\fr$ is such that
$\mathbb E L^8<\infty$ 
%\comment{Just expectation? Not $L_8$?} 
if $L$ is a random variable of distribution~$\fr$.
\begin{enumerate}[(a)]
\item If $P$ is the semi-group of
the lazy\footnote{The lazy simple symmetric random walk is defined for an arbitrary $p\in(0,1)$: at each time step, the walker does not move with probability $p$, and moves according to a simple symmetric random walk otherwise. It is a standard way to avoid parity conditions when calculating the probability that the random walker is at the origin at time $n$.} simple symmetric random walk on $\mathbb Z^d$, 
and if we let 
$\mathcal{Z}=\{t\geq 0: X(t)=0\}$, 
then, 
\begin{enumerate}[(i)]
\item 
%$\#\mathcal{Z}=\infty$ 
the cardinal of $\mathcal Z$ is almost surely infinite if $\mu=\mu_1$ or $\mu=\mu_2$ and $\delta d\leq 2$;
\item $\|X(t)\|\to\infty$ almost surely as $t\to\infty$ (in particular, $\mathcal Z$ is almost surely finite) if $\mu=\mu_2$ and $\delta d>2$.
\end{enumerate}
%$\#\mathcal{Z}<\infty$ almost surely. 
%for all $d\geq 1$, $\frak t_0 < +\infty$ almost surely.
\item If $P$ is the semi-group of the standard Brownian 
motion on $\mathbb R^d$, we let 
$\mathcal{Z}=\set{t\geq 0: \|X(t)\|\leq \eta}$; 
then, for all $\eta>0$, 
$\mathcal{Z}$ is unbounded (and thus its cardinal is infinite) almost surely if $\mu=\mu_1$ or $\mu=\mu_2$ and $\delta d\leq 2$. 
%for all $d\geq 1$, for all $\varepsilon>0$, $\frak t_\varepsilon <\infty$ almost surely.
\end{enumerate}
\end{theorem}

We believe that the ``monkey Brownian motion'' (case~$(b)$) is not neighborhood recurrent, and even transient, when $\mu=\mu_2$ and $\delta d>2$. The proof seems more involved than for the monkey random walk (case~$(a)$); and we leave this question as an open problem.

In the case when the underlying process $Z$ is the simple symmetric random walk on $\mathbb Z^d$, and, in particular, when $d\geq 2$, one could also ask how many sites have been visited at time~$t$ (asymptotically when $t\to\infty$), and what is the shape of this visited set; we also leave this question open.

\subsection{Plan of the paper}
Section~\ref{sec:wrrt} contains the main idea of the paper:  the definition of the {\it branching structure} of the monkey Markov process. This branching structure is a random tree whose nodes are labelled by the successive runs of the monkey Markov process (the trajectories between relocations). We show that the underlying random tree (once labels have been removed) is a weighted version of the random recursive tree, which we call the weighted random recursive tree, or {\sc wrrt}.
As a by-product of our proofs, we state the convergence 
in probability of the profile of this random tree (see Theorem~\ref{th:profile}).

Section~\ref{sec:CLT+} contains the proof of Theorem~\ref{monkeyMarkovLimitTheorem} (central limit theorem), Theorem~\ref{th:LLT} (local limit theorem) and 
%the remaining case ($\alpha>1$) of 
Theorem~\ref{th:rec} (recurrence); 
these three theorems are grouped into one section since their proofs rely on the same principles. 
The proof of Theorem \ref{th:cv_occupation} (convergence of the weighted occupation measure) is more involved and is presented in Section~\ref{sec:occupation_measure}. Finally, Section~\ref{sec:proof_prop} is devoted to the technical analysis of the {\sc wrrt}.

%\begin{lemma}
%Almost surely when $i$ tends to infinity, we have
%\[W_i = (\alpha+o(1))\, \frac{\mathtt e^{\log^\alpha (im)}}{i \log (im)},\]
%where $m:=\mathbb E L_1$ is the average run-length.
%\end{lemma}

\section{The weighted random recursive tree}\label{sec:wrrt}
\subsection{The Monkey walk branching structure}

The key idea of this paper is to note that each monkey Markov process 
$X = \mrw{P, \fr, \mu}$ can be coupled with a labelled random tree called 
its {\bf branching structure} (see Figure~\ref{fig:ex_branchingstructure}). 
The shape of this tree is a weighted version of the random recursive tree, 
and the labels are the runs of the monkey Markov process, i.e.\ the trajectories $(X_t)_{t\in[T_n, T_{n+1})}$ for all $n\geq 0$.

The {branching structure} $(\mathcal \tau_n, \ell_n)_{n\geq 1}$ of the monkey Markov process $X=\mrw{P,\fr,\mu}$ is built recursively as follows (see Figure~\ref{fig:ex_branchingstructure}):
We use the random variables $(L_i)_{i\geq 1}$ and $(R_i)_{i\geq 1}$ that were used when defining $X$. 

Recall that $T_0 = 0$ and $T_n = L_1 + \cdots + L_n$ for all $n\geq 1$.
\begin{itemize}
\item The tree $\tau_1$ is the one-node tree whose unique node $\nu_1$ is labelled by $\ell_1(\nu_1) = (X_t)_{t\in[0, T_1)}$.
\item For all integers $n\geq 1$, the tree $\tau_{n+1}$ is obtained from the tree $\tau_n$ by attaching a new node called $\nu_{n+1}$ in the following way:
\begin{itemize}
\item let $\xi(n+1)\in\{1, \ldots, n\}$ be the unique index such that $T_{\xi(n+1)-1} \leq R_{n+1} < T_{\xi(n+1)}$, and add $\nu_{n+1}$ as a new child of $\nu_{\xi(n+1)}$;
\item set $\ell_{n+1}(\nu) = \ell_n(\nu)$ for all $\nu$ already in $\tau_n$ and label $\nu_{n+1}$ by
\begin{esn}\ell_{n+1}(\nu_{n+1}) = (X_t)_{t\in[T_n, T_{n+1})}.\end{esn}
\end{itemize}
\end{itemize}

\begin{figure}
$\vcenter{\hbox{\includegraphics[width=8cm]{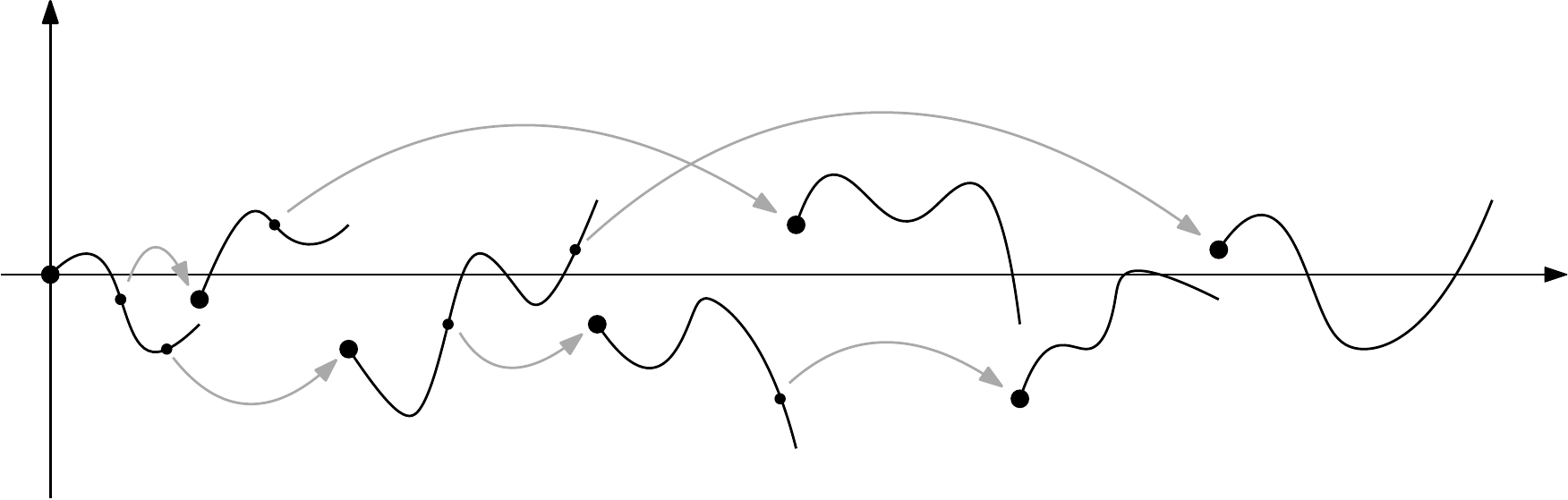}}}$
\hspace{1cm}$\vcenter{\hbox{\includegraphics[width=5cm,page=1]{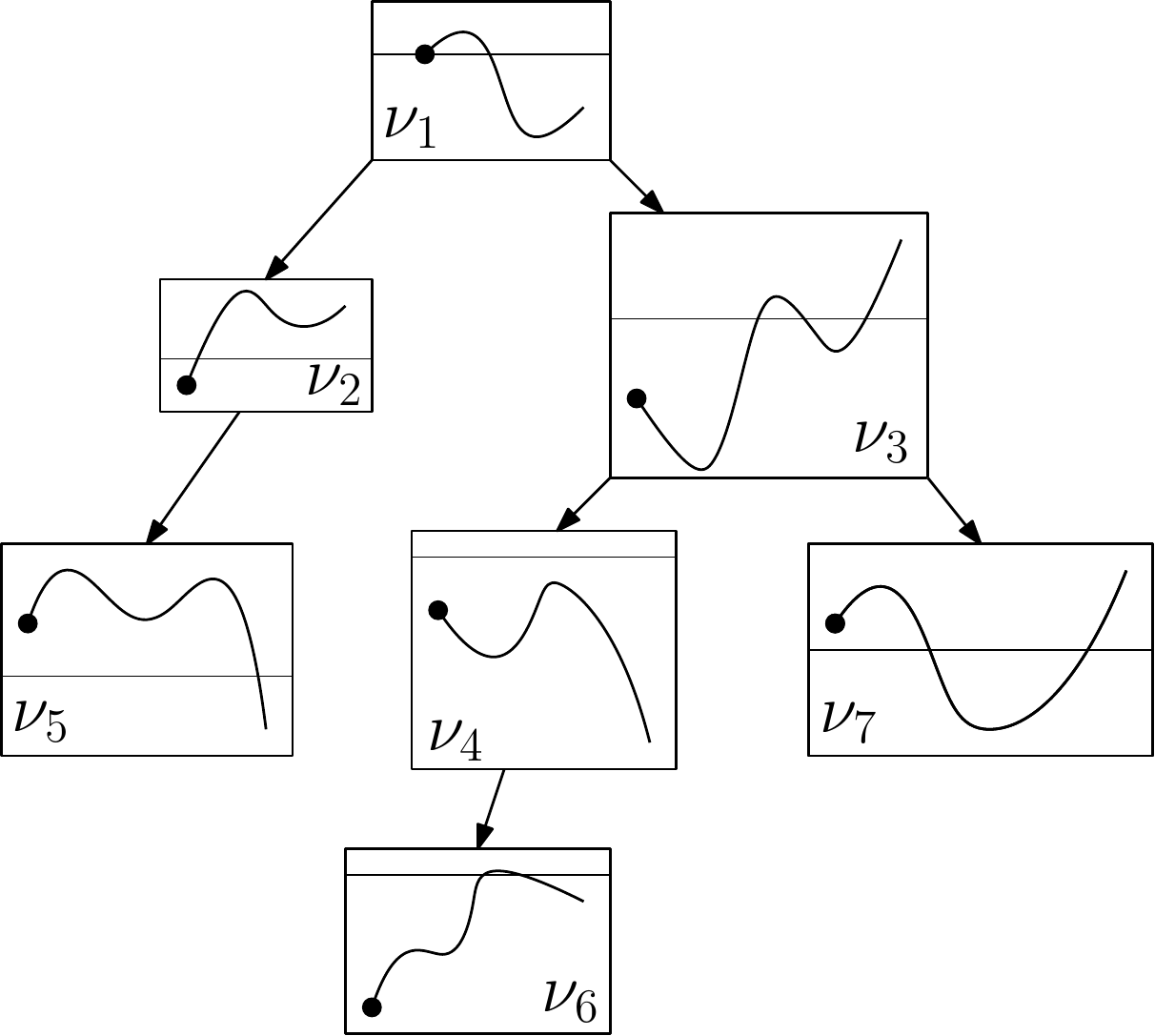}}}$
\caption{A realization of the monkey Markov process on $\mathbb R$ (left), and its associated branching structure (right). The grey arrows on the left are a graphic representation of the relocations; recall that at every relocation point $T_i$, the walker chooses a random time 
%$U_{i+1}T_i$ uniformly in its past, 
$R_{i}$ in its past according to the memory kernel, 
and jumps back to where it was at that time. 
Each grey arrow points from $R_{i}$ to $T_i$ for some~$i$. 
%\comment{Say explicitly that $\alpha=1$.} 
On the right, we show the tree $\mathcal T_7$: 
we have $(\nu_1, \ldots, \nu_7) = (\emptyset,1,2,21,11,211,22)$.
%It can be summarized by $\nu_1\xrightarrow{\sss d} \nu_2\xrightarrow{\sss d} \nu_5$, $\nu_1\xrightarrow{\sss d} \nu_3\xrightarrow{\sss d}\nu_4\xrightarrow{\sss d}\nu_6$ and $\nu_3\xrightarrow{\sss d}\nu_7$. 
%The planar tree that is associated is $\set{\emptyset,1,11,2,21,211,22}$. 
}
\label{fig:ex_branchingstructure}
\end{figure}

\begin{definition}
Given a sequence of weights $\bs w= (w_i)_{i\geq 1}$, 
we define the $\bs w$-weighted random recursive tree ($\bs w$-{\sc wrrt}) $\mathcal T=(\mathcal T_n)_{n\geq 1}$ as follows:
\begin{itemize}
\item $\mathcal T_1$ is the one node tree;
\item $\mathcal T_{n+1}$ is the tree $\mathcal T_n$ to which a new node $\nu_{n+1}$ has been attached at random as follows: choose an integer $\xi(n+1)$ at random in $\{1, \ldots, n\}$ with probability
\[\mathbb P(\xi(n+1) = i) = \frac{w_i}{w_1+\cdots+w_n},\]
and add $\nu_{n+1}$ to $\mathcal T_n$ as a new child of $\nu_{\xi(n+1)}$. 
\end{itemize}%The law of the $\bs w$-weighted random recursive tree will be denoted $\p_{\bs w}$. 
\end{definition}
In the above definition, each of the trees $\mathcal T_n$ is \emph{increasing} (cf. \cite{MR1251994}) in the sense that for all integers $1\leq i, j\leq n$, if node $\nu_j$ is a child of node $\nu_i$, then $i<j$. 
%it can be considered rooted at $\nu_1$, and that paths from the root $\nu_{i_1},\ldots, \nu_{i_l}$ are such that $i_1<i_2<\cdots< i_l$. 
We can compute the law of a $\bs w$-{\sc wrrt} tree $\mathcal T_n$ as follows: 
if $t$ is an increasing tree with $n$ vertices and $\pi_t(i)$ is the index of the parent of~$\nu_i$ in the tree~$t$, then:
\begin{equation}
\label{eq:wrrtLaw}
\proba{\mathcal T_n=t}=\prod_{i=2}^n\frac{w_{\pi_t(i)}}{w_1+\cdots+w_{i-1}}. 
\end{equation}

Note that, when $w_i = w_1\in (0,\infty)$ for all $i\geq 1$, then the $\bs w$-{\sc wrrt} is the so-called random recursive tree ({\sc rrt}). The {\sc wrrt} can thus be seen as a generalization of the {\sc rrt}. 

\begin{remark}[The {\sc wrrt} in the literature]
{\cec The {\sc wrrt} was already introduced by Borovkov and Vatutin~\cite{BV05,BV06}. They assume that the weight of node is given $w_i = \prod_{j=1}^i c_j$, where the sequence $(c_j)_{j\geq 1}$ is a sequence of i.i.d.\ random variables, and then show a central limit theorem for the height of a node taken at random in the tree with probability proportional to the weights. Hiesmayr and I\c slak~\cite{HI17} also introduced the {\sc wrrt} (they call it the {\sc wrt}); they prove asymptotic results for the number of leaves and the height of the random tree.}
We prove in this paper convergence in probability of the profile of this new random tree (see Theorem~\ref{th:profile}) when the $(w_i)_{i\geq 1}$ are randomly chosen independently and identically at random and have finite second moment.
{\cec Other questions about this random tree, such as determining its degree distribution or stronger convergence theorems for the profile, remain open.}
\end{remark}

{\cec In} our random walk setting, we let, for all $i\geq 1$,
\begin{equation}\label{eq:def_weights}
W_i = \int_{T_{i-1}}^{T_i} \mu,
\end{equation}
and $\bs W  = (W_i)_{i\geq 1}$. Note that, for all integers $i\geq 1$, $(W_1, \ldots, W_i)$ is $(L_1, \ldots, L_i)$-measurable.
Conditionally on $\bs L$, and thus $\bs W$, 
the (unlabelled) tree $(\tau_n)_{n\geq 1}$ is distributed as the $\bs W$-weighted random recursive tree; in other words, for all $n$-node increasing tree $t$,
\begin{linenomath}
\begin{equation*}
\mathbb P_{\bs L}(\tau_n=t)=\prod_{m=2}^n\frac{W_{\pi_t(m)}}{W_1 + \cdots +W_{m-1}}, 
\end{equation*}
where $\mathbb P_{\bs L}$ denotes the probability conditionally on the sequence~$\bs L$.
\end{linenomath}

\medskip
Since for all $n$, $\tau_{n}$ is a subtree of $\tau_{n+1}$, and $\ell_{n+1}$ is equal to $\ell_n$ on $\tau_n$, we can therefore define $(\tau,\ell)$ as the union of $\{(\tau_n,\ell_n)\}_{n\geq 1}$. Note that, conditionally on $\bs L$,  
the sequence of labels $\ell = \{\ell(\nu_n)\}_{n\geq 1}$ indexed by the nodes of $\tau$ is a branching Markov chain, meaning that 
\begin{itemize}
\item the sequence of labels along any branch from the root to infinity is distributed as a Markov chain of a given kernel $K$,
\item the sequences of labels along two distinct branches are independent after their last common ancestor.
\end{itemize}
Note that the kernel $K$ of the branching Markov chain $(\tau,\ell)$ 
is defined as follows:
for all $0\leq s\leq t$, for all $x : [s,t)\to \mathbb R^d$,
draw a random variable R according to the following distribution
\[\mathbb P(R\leq u) = \frac{\int_s^u \mu}{\int_s^t \mu},\]
start a Markov process of semi-group $P$ at position $x(R)$ and run it from time 0 to time $[0,W)$ (where $W$ is a random variable of distribution $\fr$, independent from $R$); the law of this process is our definition of $K(x,\cdot)$.

In this construction, we have used the following fact: 
When relocating at time $T_n$, 
the walker chooses a time $R_n$
%uniformly \comment{Not uniform unless $\alpha=1$} 
randomly in its past according to the memory kernel; 
this is equivalent to first choosing a run in its past, with probability proportional to the $W_i$'s, and then choosing a time at random inside this random run with probability given by
\[\mathbb P_{\bs L}(R_n\leq x \mid R_n\in [T_{i-1},T_i))
=\frac{\int_{T_{i-1}}^x \mu}{\int_{T_{i-1}}^{T_i} \mu}.\]

\subsection{General background on trees and notations}
So far, we have described trees as cycle-free graphs on a set of nodes $\{\nu_1, \nu_2, \ldots\}$, where $\nu_1$ is seen as the root. We call {\it parent} of a node $u$ the first node in the path from $u$ to the root, the {\it ancestors} of $u$ are all the nodes on the path from~$u$ to the root. The {\it children} of a node $u$ are all the nodes whose parent is~$u$, a \textit{leaf} of $\tau$ is a node with no children, while the \textit{internal nodes} are those with at least one child. 
The \textit{height} of a node $u$, denoted by $|u|$, is the graph distance from $u$ to the root. Finally, the \textit{last common ancestor}~$u\wedge v$ of two nodes~$u$ and~$v$ is the highest node (i.e.\ with the largest height) that is ancestor to both $u$ and $v$.

In the following, it is convenient to embed trees in the plane by ordering the children of all nodes. We decide that the children of a node are ordered from left to right in increasing order of their indexes. 
%In this paper, a tree can be viewed as an increasing tree, as above, or as a planar tree. 
%Both representations have their advantages.
We can then associate to each node a word on the alphabet $\mathcal A = \{1, 2, \ldots\}$ as follows: the root is associated to the empty word $\varnothing$, and each node $u$ is given the word of its parent to which a last ``letter'' is added; this last letter is the rank of $u$ among its siblings (from left to right). For example, node $13$ is the third child of the first child of the root.
Our trees can thus be seen as sets of words on $\mathcal A = \{1, 2, \ldots\}$;
%a planar tree $\tau$ is a set of words on a countable alphabet $\mathcal A$ such that,
%if a word $w$ is in~$\tau$, then all its prefixes are also in~$\tau$. In this paper, we only consider the so-called planar trees:  A tree is called a \textit{planar tree} if the alphabet is $\mathcal A = \{1, 2, \ldots\}$, 
%and if, for all nodes $u=u_1\ldots u_{r-1}u_r$ in $\tau$, the set $\{u_1 \ldots u_{r-1}i\colon i\leq u_r\}$ is also contained in $\tau$. To pass from the increasing tree to the planar tree, we only need to associate a word to every $\nu_i$. This can be done in only one way which preserves the order $\nu_1<\cdots<\nu_n$. Note however, that the planar tree contains less information than the increasing tree; in the latter, we now the order in which individuals were incorporated to the tree, while in the former we only know the order relationship between siblings. 
we denote by $\mathcal A^*$ the set of all (finite) words on $\mathcal A$.
In the following, we identify a node with its word, and allow ourselves to say, e.g.\ ``the prefix of node $\nu_i$'' instead of the more accurate ``the prefix of the word associated to node $\nu_i$''.
Note that the ancestors of a node are all its prefixes, the height of a node is its length (in terms of number of letters), and the last common ancestor of~$u$ and~$v$ is their longest common prefix.

\subsection{Key property of the {\sc wrrt}}
As already discussed, the key in studying the monkey Markov process
is to note that, for all $t\in\mc T$, $X(t)$ has the same distribution as $Z(S(t))$, where $S(t)$ is a random variable taking values in $\mc T$ that depends
on the sequence of run-lengths $\bs L = (L_i)_{i\geq 1}$ and on the memory kernel, but is independent of the Markov process $Z$. 
%\ger{$S(t)$ is the effective number of Markovian steps we need to take in order to reach the position at time $t$; in Figure \ref{fig:ex_branchingstructure} it can be obtained by starting at $t$ and following the pieces of paths backwards following the grey arrows. }
All our results follow from this equality in distribution, and from a central limit theorem verified by $S(t)$. 

{\cec For all fixed $t$, $S(t)$ is defined by looking at the trajectory of the walker backwards in time, starting from $t$ (see Figure~\ref{onlyUsefulPartOfTheTrajectoryPicture}); whenever we see a relocation point, say at time $T_i$, we jump directly to time $R_i$ (note that $X(R_i) = X(T_i)$ by definition). By doing so, we follow a continuous path; $S(t)$ is equal, by definition, to the length of this path. The exact definition of $S(t)$ is given in Section~\ref{sub:CLT}.}

The following theorem is key in proving a central limit theorem for $S(t)$; indeed, we show later that $S(t)$ is equal in distribution to the quantity $\Phi(N(t))$ (defined in the statement of the theorem) for some random $N(t)$ that behaves almost surely as $t/\mathbb EL$ when $t\to\infty$.

\medskip
Let $F^{\sss (n)}_i$ be a random variable distributed as $R_n-T_{i-1}$ conditionally on $R_n\in [T_{i-1}, T_i)$ for all $1\leq i\leq n$; 
note that, conditionally on $\bs L$, for all $x\in[T_{i-1},T_i)$, we have,
\[\mathbb P_{\bs L}(F^{\sss (n)}_i\leq x)
=\mathbb P_{\bs L}(R_n\leq T_{i-1}+x\mid R_n\in[T_{i-1},T_i)) 
= \frac{\int_{T_{i-1}}^{T_{i-1}+x} \mu}{\int_{T_{i-1}}^{T_i}\mu},\]
implying that the law of $F_i^{\sss (n)}$ does not depend on $n$. Therefore, for all $i\geq 1$, we denote by $(F_i)_{i\geq 1}$ a sequence of independent random variables such that
\begin{equation}\label{eq:distr_F}
\mathbb P_{\bs L}(F_i\leq x) = \frac{\int_{T_{i-1}}^{T_{i-1}+x} \mu}{\int_{T_{i-1}}^{T_i}\mu},
\end{equation}
for all $i\geq 1$.

Let $\mathcal T_n$ be the $n$-node $\bs W$-{\sc wrrt}, 
and $\bs F=(F_i)_{i\geq 1}$ a sequence of independent random variables whose distribution is given by Equation~\eqref{eq:distr_F}.
Define, for all integers $n$ and for all nodes $\nu\in\mathcal T_n$,
\[\Phi(\nu) = \sum_{i=1}^n F_i \bs 1_{\nu_i\preccurlyeq \nu},\]
where $u\preccurlyeq \nu$ if, and only if, $u$ is a (non-strict) ancestor (or prefix) of $\nu$. In the following theorem, we show a joint limit theorem for $\Phi(u_n)$ and $\Phi(v_n)$, where $u_n$ and $v_n$ are two nodes taken independently and weight-proportionally at random in the $n$-node $\bs W$-{\sc wrrt}.

\begin{theorem}\label{th:sums_along_branches}
Assume that $\mathbb E L^8<\infty$ if~$L$ is a random variable of distribution $\varphi_{run}$, 
and denote by 
\[\hat \kappa_i = 
\frac{\mathbb E L^i}{i\mathbb E L} \text{ if }\mu=\mu_1,
\quad\text{ and }\quad
\hat\kappa_i = \frac{\mathbb E L^i}{i(\mathbb E L)^{1-\delta}} \text{ if }\mu=\mu_2,\]
for $i\in\{2,3\}$. Finally, if $\mu=\mu_2$, assume that $\delta\in(0,\nicefrac12)$.

Let $u_n$ and $v_n$ two nodes taken independently and weight-proportionally at random in $\mathcal T_n$.
\begin{itemize}
\item If $\mu=\mu_1$ or $\mu=\mu_2$ and $\delta\in(0,\nicefrac12)$, then, conditionally on $\paren{\bs L,\bs F}$, $\paren{\bs L,\bs F}$-almost surely, we have 
\[
\left(\frac{\Phi(u_n) - \hat\kappa_2s(n)}{\sqrt{\hat\kappa_3s(n)}},
\frac{\Phi(v_n) - \hat\kappa_2s(n)}{\sqrt{\hat\kappa_3s(n)}}
\right)
\xrightarrow{\sss d} (\Lambda_1,\Lambda_2),\]
%in distribution when
as $n\to\infty$, 
where $\Lambda_1$ and $\Lambda_2$ are two independent standard Gaussian random variables. 
\item If $\mu=\mu_2$ and $\delta = \nicefrac12$, then, conditionally on $\paren{\bs L,\bs F}$, $\paren{\bs L,\bs F}$-almost surely, we have
\[\frac{\Phi(u_n) - \hat\kappa_2s(n)}{\sqrt{\hat\kappa_3s(n)}}
\xrightarrow{\sss d} \mathcal N(0,1),\]
as $n\to+\infty$.
\end{itemize}
\end{theorem}

%\begin{remark}
%\end{remark}

\begin{remark}
Such a conditional weak convergence statement is typical of the random environment literature where this would be called a ``quenched" result valid for almost all realizations of the environment. Here, the sequence of run lengths could be thought of as the random environment. 
It implies annealed weak convergence statements: in particular the limit theorem holds if we just condition on $\bs L$ or $\bs F$ or even unconditionally. 
\end{remark}

The proof of this theorem is relatively technical and long, and we therefore choose to postpone it to Section~\ref{sec:proof_prop}, and first show how it implies our main results. Note that the conditions $\delta\in(0,\nicefrac12]$ in Theorem~\ref{monkeyMarkovLimitTheorem} and $\delta\in(0,\nicefrac12)$ in Theorem~\ref{th:cv_occupation} come from the conditions needed here in Theorem~\ref{th:sums_along_branches}: the convergence of the first marginal is enough to prove Theorem~\ref{monkeyMarkovLimitTheorem}, while joint convergence is needed for Theorem~\ref{th:cv_occupation}. 

\bigskip
Following the same proof as for Theorem~\ref{th:sums_along_branches}, just replacing $\Phi(\nu)$ by the simpler 
\[\Psi(\nu) = |\nu| +1 = \sum_{i=1}^n \bs 1_{\nu_i \preccurlyeq \nu},\]
one can prove convergence in probability of the {weighted} profile of the ${\bs W}$-{\sc wrrt}.

{\cec
\begin{definition}
The profile of a tree $t$ is the probability distribution of the height of a node taken uniformly at random in $t$. Equivalently, it is given by the following sum of Dirac masses:
$\frac1{|t|}\sum_{\nu\in t} \bs \delta_{|\nu|}$, 
where $|t|$ denotes the number of nodes in $t$, $|\nu|$ the height of node $\nu$ and $\bs \delta_x$ is the Dirac mass at $x$, for all $x\in \mathbb N$.
\end{definition}

We denote by 
\[\pi_n = \frac{\sum_{i=1}^n W_i \bs \delta_{|\nu_i|}}{\sum_{i=1}^n W_i}\]
the {weighted} profile of the $n$-node ${\bs W}$-{\sc wrrt} $\mathcal T_n$.
Note that since $\mathcal T_n$ is a random tree, then $\pi_n$ is a random 
probability distribution on~$\mathbb N$; it is the probability distribution of the height of a node taken weight-proportionally at random in $\mathcal T_n$. The following result states convergence in probability of this random probability distribution on the space of all probability distributions on $\mathbb N$ equipped with the weak topology.}

\begin{theorem}\label{th:profile}
Assume that $\mathbb E L^2<\infty$, and,
if $\mu=\mu_2$, assume that $\delta\in(0,\nicefrac12)$.
Then, when $n$ goes to infinity,
we have
\[\pi_n\big(\sqrt{s(n)}~\cdot~ + s(n)\big) \xrightarrow{\sss p} \mc N(0,1),\]
with respect to the weak topology.
\end{theorem}
The profile of a random tree is the distribution of the height of a node taken uniformly at random in this tree. 
Profiles of random trees are widely studied in the literature: see, e.g., Drmota and Gittenberger~\cite{DrmotaGittenberger97} for the Catalan tree,
Chauvin, Drmota and Jabbour-Hattab~\cite{CDJH01} and Chauvin, Klein, Marckert and Rouault~\cite{CKMR05} for the binary search tree,
Schopp~\cite{Schopp10} for the $m$-ary increasing tree,
Katona~\cite{Katona05} and Sulzbach~\cite{Sulzbach08} for the preferential attachment tree, and the very recent universal result of Kabluchko, Marynych, and Sulzbach~\cite{KMS17}. In all these examples, the height of a typical node is of order $\sqrt{n}$ or $\log n$ where $n$ is the number of nodes in the tree. 
However, in our {\sc wrrt}s, we exhibit typical heights of order $\log\log n$, $(\log n)^\alpha$ and even~$n^{\delta}$ (for $\delta\leq 1/2$).

%When the memory kernel is uniform ($\mu=\mu_1$ and $\alpha=\beta=1$), and thus $L_i = W_i$ for all $i\geq 1$,
%then { the weighted profile is equal to the (non-weighted) profile}, and Theorem~\ref{th:profile} implies that
%\[\pi_n\big(\sqrt{\log n}~\cdot~ + \log n\big) \xrightarrow{\sss p} \mc N(0,1).\]

%It is interesting to note that the randomness of the weights has no effect on this asymptotic result.
%We believe that removing the second moment assumption on the weights $(L_i=W_i)_{i\geq 1}$ would  lead to substantially different asymptotic profiles (one would expect the tree to be even ``flatter''); this problem is left open.

%Note that the result of \cite{KMS17} applies to a large class of random trees and gives almost sure convergence of the profile, but this general result does not apply to the weighted random recursive tree because of the fact that each node is given a random i.i.d.\ weight. For the same reason, the preferential attachment tree of Bianconi and Barab\'asi~\cite{BB01} is also beyond the general framework of~\cite{KMS17}, and convergence results about its profile are open.

\subsection{A key coupling}

We now give two results which help in understanding why Theorem \ref{th:sums_along_branches} might be true. 
We link each one of the summands defining $\Phi$ (or more easily $\Psi$) to independent random variables. 
This is done by constructing, in a consistent manner, the $\bs w$-{\sc wrrt} together with vertices that are sampled weight-proportionally at random from the $n$-node tree. We refer the reader to Figure~\ref{fig:coupling1} for a visual aid.
\begin{figure}
\includegraphics[width=.8\textwidth]{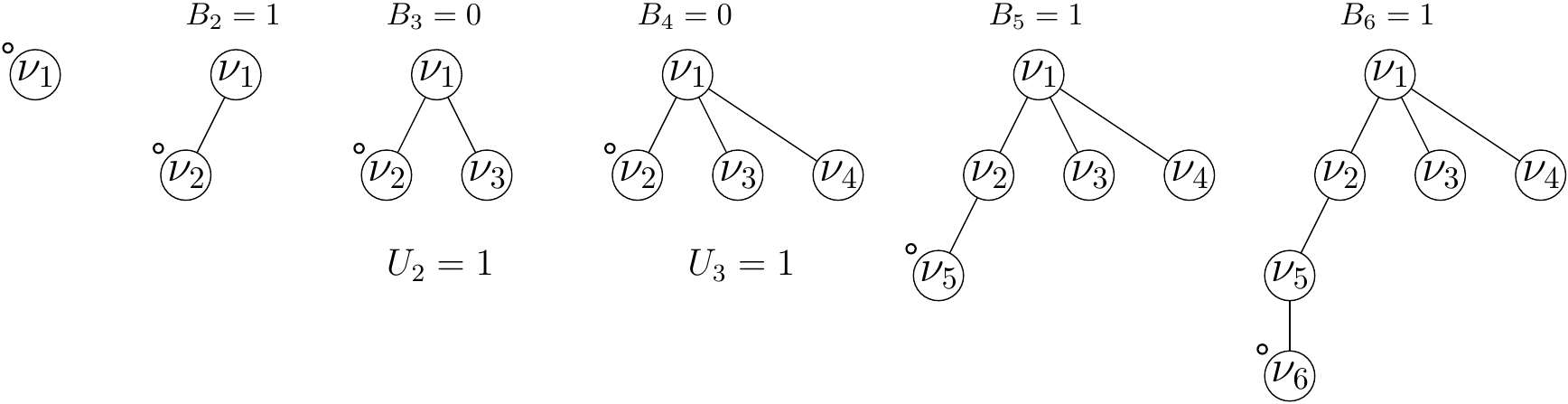}
\caption{A realization of the couples $(\tilde{\mathcal T}_n, \tilde u_n)$, for $1\leq n\leq 6$ defined in Proposition~\ref{prop:Dobrow}. For each $n$, the white dot marks the node that equals $\tilde u_n$ in $\tilde {\mathcal T}_n$.}
\label{fig:coupling1}
\end{figure}

Let $\bs w= (w_i)_{i\geq 1}$ be a sequence of positive real numbers 
and $\bs s=(s_i)_{i\geq 1}$ the sequence of its cumulative sums: for all $i\geq 1$, $s_i = \sum_{j=1}^i w_j$. 
Let $\bs U=(U_n)_{n\geq 1}$ be a sequence of independent random variables such that, for all $i, n\geq 1$, $\proba{U_n=i}=\nicefrac{w_i}{s_n}$. 
Let $\bs B = (B_n)_{n\geq 1}$ be a sequence of independent random variables 
(between themselves and also of~$\bs U$) 
such that, for all $n\geq 1$, $B_n$ is Bernoulli-distributed with parameter~$\nicefrac{w_n}{s_n}$. 

Let $\nu_1=\varnothing$ and $\tilde{\mathcal T}_1$ be the tree equal to $\{\nu_1\}$; 
also let $\tilde u_1=\nu_1$. 
For all $n\geq 1$, given $(\tilde{\mathcal T}_{n-1},\tilde u_{n-1})$, 
we define $(\tilde{\mathcal T}_n,\tilde u_n)$ as follows: 
\begin{itemize}[topsep=0pt]
\item if $B_n=1$, set $\tilde u_n=\nu_n$ and $\tilde \xi(n)=\tilde u_{n-1}$; 
\item if $B_n=0$, set $\tilde u_n=\tilde u_{n-1}$ and $\tilde \xi(n)=U_{n-1}$.
\end{itemize}
Finally, let $\tilde{\mathcal T}_n$ be the tree obtained by adding node $\nu_n$ to $\tilde{\mathcal T}_{n-1}$ with an edge between $\nu_{\tilde \xi(n)}$ and $\nu_n$.

\begin{proposition}
\label{prop:Dobrow}
For all $n\geq 1$, $\tilde{\mathcal T}_n$ is distributed as the $n$-node ${\bs w}$-{\sc wrrt}, and, given $\tilde{\mathcal T}_n$, the node $\tilde u_n$ is taken weight-proportionally at random in~$\tilde{\mathcal T}_n$.
\end{proposition}

Note that the sequence $(\tilde{\mathcal T}_n, \tilde u_n)_{n\geq 1}$ 
%defined in Proposition~\ref{prop:Dobrow} 
is such that
\begin{itemize}
\item the parent of $\tilde u_n$ is $\tilde u_{n-1}$ if, and only if, $B_n = 1$;
\item $\tilde u_n = \tilde u_{n-1}$ if, and only if, $B_n = 0$.
\end{itemize}
This implies in particular that $\{B_i=1\}=\{\nu_i\preccurlyeq \tilde u_n\}$, 
for all integers~$n, i\geq 1$; we thus get the following corollary:
\begin{corollary}\label{cor:bernoullis}
Let $\bs w=(w_i)_{i\geq 1}$ be a sequence of positive real numbers.
 %and $\bs s=(s_i)_{i\geq 1}$ the sequence of its cumulative sums.
For all integers $n\geq 1$, let $u_n$ be a node chosen at random in the $n$-node $\bs w$-{\sc wrrt} $\mathcal T_n$ with probability proportional to the weights. 
Then, %conditionally on $\bs W= (W_i)_{i\geq 1}$, 
the random variables
$(\bs 1_{\nu_i \preccurlyeq u_n})_{1\leq i\leq n}$ are independent Bernoulli random variables of 
respective parameters $\nicefrac{w_i}{s_i}$, 
where $s_i = \sum_{j=1}^i w_j$ ($\forall i\geq 1$).
% \footnote{\color{magenta} State for a deterministic sequence $\bs w$.}
\end{corollary}
%Dobrow~\cite{Dobrow}
\begin{remark}[Discussion of Corollary~\ref{cor:bernoullis} with respect to the literature]
{\cec Corollary~\ref{cor:bernoullis} is a classical result in the case when the weights are constant (and the {\sc wrrt} is thus equal to the {\sc rrt}); it is for example proved in~\cite{MR969872}. The proof of~\cite{MR969872} uses the link between the height of a uniform random node and the number of records in a uniform random permutation, which is well known to be a sum of independent Bernoulli random variables.}
Note that~\cite{MR1401472}
also states a version of this result for the (unweighted) random recursive tree case but
the details of the proof are omitted.
%; we believe that the omitted argument is the induction argument we give below for the {\sc wrrt}. 
% Kuba and Wagner~\cite{KubaWagner}
\cite{MR2729386}
give an alternative proof of Devroye's result;
their argument is a bijective one, and could be adapted to the {\sc wrrt} case. 
{\cec Note that our ``coupling'' approach was used in~\cite{CH14} and \cite{Haas} for aggregating trees.}
\end{remark}

%If we can control their parameters, we can therefore obtain a central limit theorem for the height.

This corollary implies that, for all $n\geq 1$, $\Psi(u_n) = |u_n| + 1 = \sum_{i=1}^n B_i$, 
which is a sum of independent Bernoulli random variables.
Similarly, conditionally on $\bs L$ and $\bs F$, the sum defining $\Phi(u_n)$ in Theorem~\ref{th:sums_along_branches} is an (inhomogeneous) random walk. Controlling the behavior of $\bs L$ and $\bs F$ then yields a central limit theorem for each marginal in Theorem \ref{th:sums_along_branches}. 
\begin{proof}[Proof of Proposition \ref{prop:Dobrow}]
Let $t$ be an increasing tree with nodes $\nu_1,\ldots, \nu_n$. 
Denote the index of the parent of $\nu_i$ in $t$ by $\pi_t(i)$. 
The assertion follows from the formula
\[
\proba{\tilde{\mathcal T}_n=t, \tilde u_n=\nu_j}
=\bra{\prod_{i=2}^n \frac{w_{\pi_t(i)}}{s_{i-1}}}\frac{w_j}{s_n}, 
\]which we now prove by induction, the case $n=1$ being trivial. 
If the above formula is valid for~$n$, let $t$ be an increasing tree with nodes $\nu_1,\ldots, \nu_{n+1}$ and let $t'$ be the tree obtained from $t$ by removing node $\nu_{n+1}$. 
For $j\leq n$, note that $\tilde u_{n+1}=\nu_j$ if, and only if, $B_{n+1}=0$. 
Therefore, we get: 
\begin{linenomath}
\begin{align*}
\proba{\tilde{\mathcal T}_{n+1}=t, \tilde u_{n+1}=\nu_j}
&=\proba{\tilde{\mathcal T}_n=t', \tilde u_{n}=\nu_j, U_{n+1}=\pi_{t}(\nu_{n+1}),B_{n+1}=0}
\\&=\bra{\prod_{i=2}^n \frac{w_{\pi_{t'}(i)}}{s_{i-1}}}\frac{w_j}{s_n}\frac{w_{\pi_t(\nu_{n+1})}}{s_n}\frac{s_n}{s_{n+1}}
\\&=\bra{\prod_{i=2}^{n+1} \frac{w_{\pi_{t}(i)}}{s_{i-1}}}\frac{w_j}{s_{n+1}},
\end{align*}
\end{linenomath}which finishes the proof. 
\end{proof}

\begin{figure}
\includegraphics[width=.8\textwidth]{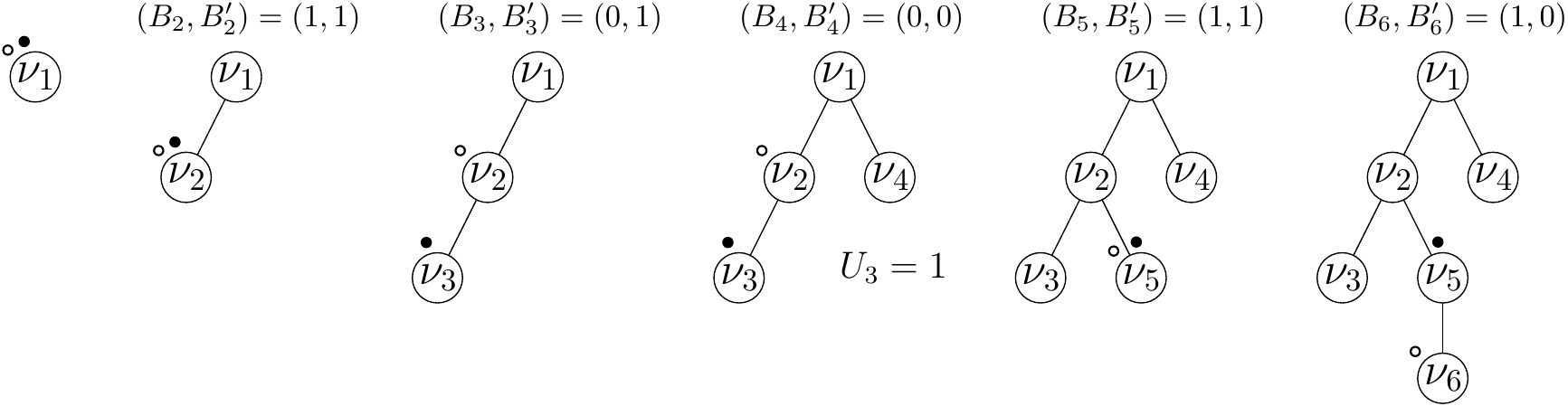}
\caption{A realization of the triples $(\tilde{\mathcal T}_n, \tilde u_n, \tilde v_n)$, for $1\leq n\leq 6$ defined in Proposition~\ref{prop:2Dobrow}. For each $n$, the white dot marks the node that equals $\tilde u_n$ in $\tilde {\mathcal T}_n$, the black dot marks the node that equals $\tilde v_n$.}
\label{fig:coupling}
\end{figure}

To get joint convergence of the two marginals in Theorem~\ref{th:sums_along_branches}, 
one needs to improve the coupling as follows (see Figure~\ref{fig:coupling} for a realization of this coupling).
Additionally to $\bs B$ and $\bs U$, we consider a sequence $\bs B'=(B'_i)_{i\geq 1}$ of independent (between themselves and of $\bs B$ and $\bs U$) Bernoulli random variables with respective parameters $\nicefrac{w_i}{s_i}$. 

Let $\nu_1 = \varnothing$, and $\tilde{\mathcal T}_1=\{\nu_1\}$; also set $\tilde u_1=\tilde v_1=\nu_1$. 
Recursively, for all $n\geq 2$, given $(\tilde{\mathcal T}_{n-1},\tilde u_{n-1}, \tilde v_{n-1})$, we define $(\tilde{\mathcal T}_n, \tilde u_n,\tilde v_n)$ as follows: 
\begin{itemize}
\item if $(B_n,B_n')=(0,0)$, we set $\tilde \xi(n)=U_{n-1}$, $\tilde u_n=\tilde u_{n-1}$ and $\tilde v_n=\tilde v_{n-1}$;
\item if $(B_n,B_n')=(1,0)$, we set $\tilde \xi(n)=\tilde u_{n-1}$, $\tilde u_n=\nu_n$ and $\tilde v_n=\tilde v_{n-1}$; 
\item if $(B_n,B_n')=(0,1)$, we set $\tilde \xi(n)=\tilde v_{n-1}$, $\tilde u_n=\tilde u_{n-1}$ and $\tilde v_n=\nu_n$; 
\item if $(B_n,B_n')=(1,1)$, we set $\tilde \xi(n)=\tilde u_{n-1}$, $\tilde u_n=\tilde v_n=\nu_n$. 
\end{itemize}
Finally, define $\tilde{\mathcal T}_{n}$ as the tree obtained by adding $\nu_{n}$ to $\tilde{\mathcal T}_{n-1}$ and a new edge from $\nu_{n}$ to $\nu_{\tilde \xi(n)}$.

\begin{proposition}
\label{prop:2Dobrow}
For every $n\geq 1$,  $\tilde {\mathcal T}_n$ is distributed as the $n$-node ${\bs w}$-{\sc wrrt}, and, given $\tilde {\mathcal T}_n$, the nodes $\tilde u_n$ and $\tilde v_n$ are independent, and taken weight-proportionally at random in $\tilde {\mathcal T}_n$. 
\end{proposition}

Note that, when $n$ evolves, $\tilde u_n$ stays on one branch of the tree, and occasionally increases its height by one (when $B_n = 1$). 
This is not the case for $\tilde v_n$: $\tilde v_n$ evolves along one branch of the tree, occasionally increasing its height by one (if $B'_n = 1$ and $B_n = 0$), 
except at some random times ($B_n = B'_n = 1$) 
when it ``jumps'' to take the same value as $\tilde u_n$ 
(see Figure~\ref{fig:coupling}, where such a jumps occurs when $n=5$).
%In contrast to Proposition \ref{prop:Dobrow}, 
%it is no longer true that the heights of $u_n$ and $v_n$ grow exactly when the corresponding Bernoulli random variables $B_n$ and $B'_n$ equal one; 
%this only happens when one of them equals one and the other one equals zero. 
%Actually, each time these variables are simultaneously equal to~$1$, the whole path from the root to $u_n$ and $v_n$ is obliterated and a new one is created. 
We will show later that, for the memory kernels of Theorems~\ref{th:cv_occupation} and~\ref{th:sums_along_branches},
this ``jumping'' only happens a finite number of times and, after that (random) time, 
$\tilde u_n \neq \tilde v_n$ for all $n$, and the last common ancestor of $\tilde u_n$ and $\tilde v_n$ stays constant in~$n$.
This observation is formalized in the forthcoming Proposition \ref{lem:LCA}; 
we later show that it implies the asymptotic independence in the joint convergence of Theorem~\ref{th:sums_along_branches}. 

\begin{proof} %\footnote{check indices!}
If $t$ is any increasing tree with nodes $\nu_1,\ldots\nu_n$ and $\pi_t(i)$ is the number of the parent of $\nu_i$ in the tree $t$, we assert that
\[
\proba{\tilde{\mathcal T}_n=t,
%\tilde \xi_{n}=\nu_i, 
\tilde u_n=\nu_j,\tilde v_n=\nu_k}=\bra{\prod_{i=2}^n \frac{w_{\pi_t(i)}}{s_{i-1}}}
%\frac{W_{i}}{S_n} 
\frac{w_j}{s_n}\frac{w_k}{s_n}.
\]Let us prove this by induction, the case $n=1$ being trivial. 
Note that, by construction,  $\big(\tilde {\mathcal T}_n,\tilde u_n,\tilde v_n\big)$, 
$U_n$, $B_{n+1}$ and $B'_{n+1}$ are independent. 
Suppose now that the above formula holds true and let $t$ be an increasing tree on $n+1$ nodes. Let $t'$ be the tree obtained from $t$ by deleting node $\nu_{n+1}$ and the edge from $\nu_{n+1}$ to its parent. 
For $j,k\leq n$, $\tilde u_{n+1}=\nu_j$ if, and only if, $B_{n+1}=0$ (likewise for $\tilde v_{n+1}$); therefore, we get: 
\ger{\begin{linenomath}
\begin{align*}
&\proba{\tilde{\mathcal T}_{n+1}=t,
\tilde u_{n+1}=\nu_j,\tilde v_{n+1}=\nu_k}
\\&=
\proba{B_{n+1}=0,B'_{n+1}=0, \tilde{\mathcal T}_{n}=t',U_{n}=\pi_t(n+1),\tilde u_n=\nu_j,\tilde v_n=\nu_k}
\\
&=\frac{s_n}{s_{n+1}}\frac{s_n}{s_{n+1}}
\proba{\tilde{\mathcal T}_{n}=t', 
%\xi_{n+1}=\pi_t(n+1),
\tilde u_n=\nu_j,\tilde v_n=\nu_k}\proba{U_{n+1}=\pi_t(n+1)}
\\&=\frac{s_n}{s_{n+1}}\frac{s_n}{s_{n+1}}\bra{\prod_{i=2}^n \frac{w_{\pi_t(i)}}{s_{i-1}}}\frac{w_{\pi_t(n+1)}}{s_{n}} \frac{w_j}{s_n}\frac{w_k}{s_n}%\frac{W_i}{S_{n+1}}
%\\&=\bra{\prod_{m=2}^n \frac{W_{\pi_t(m)}}{S_m}}\frac{W_{\pi_t(n+1)}}{S_{n+1}} \frac{W_j}{S_{n+1}}\frac{W_k}{S_{n+1}}\frac{W_i}{S_{n+1}}
\\&=\bra{\prod_{i=2}^{n+1} \frac{w_{\pi_t(i)}}{s_{i-1}}}
%\frac{W_i}{S_{n+1}}
\frac{w_j}{s_{n+1}}\frac{w_k}{s_{n+1}}.
\end{align*}
\end{linenomath}}If $\nu_j\leq n$ and $\nu_k=n+1$, then
\ger{
\begin{linenomath}
\begin{align*}
\proba{\tilde{\mathcal T}_{n+1}=t,\tilde u_{n+1}=\nu_j,\tilde v_{n+1}=\nu_{n+1}}
%&\hspace{2cm}=
%\frac{S_n}{S_{n+1}}\frac{W_{n+1}}{S_{n+1}}
&=\proba{B_{n+1}=0, B'_{n+1}=1, \tilde{\mathcal T}_{n}=t', \tilde v_n=\pi_t(n+1),\tilde u_{n}=\nu_j}
\\
&=\frac{S_n}{S_{n+1}}\frac{w_{n+1}}{s_{n+1}}\bra{\prod_{i=2}^n \frac{w_{\pi_t(i)}}{s_{i-1}} } 
\frac{w_{\pi_t(n+1)}}{s_{n}} 
\frac{w_j}{s_n}
\\&=\bra{\prod_{i=2}^{n+1} \frac{w_{\pi_t(i)}}{s_{i-1}} }\frac{w_j}{s_{n+1}}\frac{w_{n+1}}{s_{n+1}}.
\end{align*}
\end{linenomath}}Finally, if both $j, k=n+1$, we get: 
\ger{
\begin{linenomath}
\begin{align*}
\proba{\tilde{\mathcal T}_{n+1}=t,\xi_{n+2}=\nu_i,\tilde u_{n+1}=\nu_j,\tilde v_{n+1}=\nu_k}
&=
%\frac{W_{n+1}}{S_{n+1}}\frac{W_{n+1}}{S_{n+1}}
\proba{B_{n+1}=1, B'_{n+1}=1,\tilde{\mathcal T}_{n}=t',U_{n+1}=\pi_t(n+1)}
%\\&=\frac{W_{n+1}}{S_{n+1}}\frac{W_n}{S_{n+1}}
%\proba{\tilde{\mathcal T}_{n}=t',U_{n+1}=\nu_i}
\\
&=\frac{w_{n+1}}{s_{n+1}}\frac{w_{n+1}}{s_{n+1}}
\bra{\prod_{i=2}^n \frac{w_{\pi_t(i)}}{s_{i-1}}}\frac{w_{\pi_t(n+1)}}{s_{n}}
\\&=\bra{\prod_{i=2}^{n+1} \frac{w_{\pi_t(i)}}{s_{i-1}}} \frac{w_{n+1}}{s_{n+1}}\frac{w_{n+1}}{s_{n+1}}.\qedhere
\end{align*}
\end{linenomath}
}
\end{proof}

\section{Central limit theorem, local limit theorem and recurrence}\label{sec:CLT+}
% We change the tree with trajectories into a tree with initial positions of trajectories. 
% Key observation: labels of the nodes in the branch going from $n$ to the root can be obtained by sampling $Z$ at $W_1U_1+\cdots+W_nU_n$. This forms a new Markov process
% The new Markov process is (\tilde a_n,\tilde b_n)-ergodic we compute everything. 
% Then, relate labels nodes to $X_n=V_{|u_n|+O(1)}}$. Apply the ergodic property and CLT for |u_n|. 

\subsection{Proof of Theorem~\ref{monkeyMarkovLimitTheorem}}\label{sub:CLT}
\begin{figure}
$\vcenter{\hbox{
\includegraphics[width=9cm]{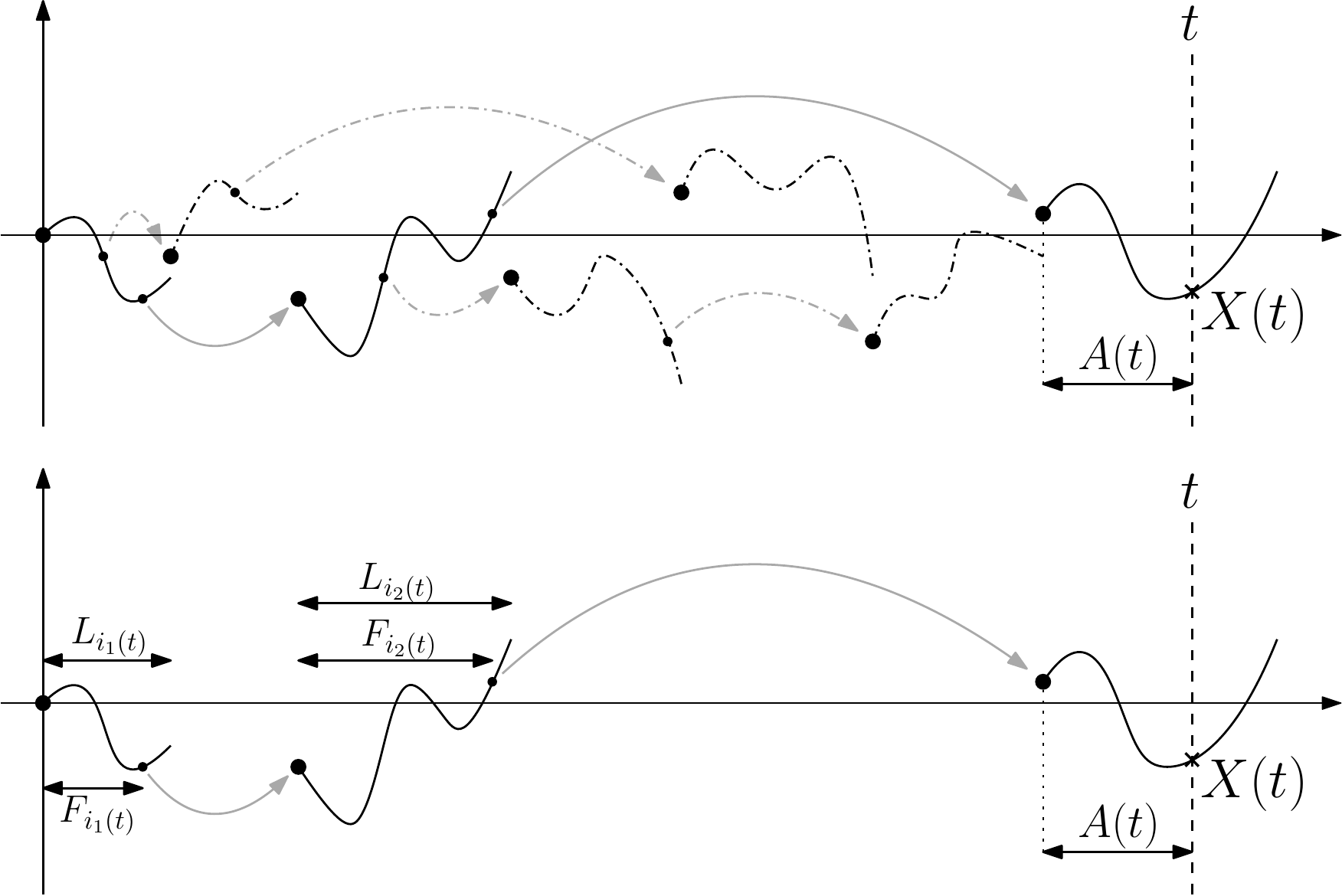}
}}$
\hspace{1cm}
$\vcenter{\hbox{\includegraphics[width=5.5cm,page=2]{tree_walk}}}$
\caption{A realization of the monkey Markov process and its branching structure (this picture illustrates notations and ideas used in the proof of Theorem~\ref{monkeyMarkovLimitTheorem}).
In this example, the number of runs before time $t-A(t)$ is $N(t) = 6$, 
the height of $\omega_{N(t)}$ is $K(t) = 2$.
The runs associated to a node on the branch from 
the root to $\omega_{N(t)}$ are in solid line on the top-left corner, 
the other runs are dotted.
}
\label{onlyUsefulPartOfTheTrajectoryPicture}
\end{figure}
We refer the reader to Figure~\ref{onlyUsefulPartOfTheTrajectoryPicture}
for {\cec the} notations and main ideas of the proof.
We denote by $A(t)$ the time to the last relocation \ger{before time}~$t$, 
\ger{by} $N(t)$ the number of runs before time $t-A(t)$, 
\ger{by }$\omega_{N(t)}$ the node of $\tau_{N(t)}$ defined as the parent of the node labelled by the run that straddles time $t$, 
and by $K(t)$ the height of $\omega_{N(t)}$.  
Finally, we denote by $L_{i_1(t)}, \ldots, L_{i_{K(t)}(t)}$ the run-lengths of the successive runs labelling the nodes of the branch from the root of $\tau_{N(t)}$ to $\omega_{N(t)}$ (in that order).
By construction of the model, and using the Markov property of~$Z$, we \ger{get the following equality in distribution}\footnote{\cec In the whole paper, we use $\overset{\sss d}{=}$ to denote equality in distribution.}:
\[X(t) \overset{\sss d}{=} Z\bigg(\sum_{\ell=1}^{K(t)} F_{i_\ell(t)} + A(t)\bigg),\]
where, for all integers $i$, $F_i$ is distributed as
\[\mathbb P_{\bs L}(F_i\leq x) 
= \frac{\int_{T_{i-1}}^{T_{i-1}+x} \mu }{\int_{T_{i-1}}^{T_i} \mu }.\]
It is important to note that the $F_{i_\ell(t)}$'s are no longer distributed according to this distribution since they have been ``size-biased'' by the fact that their definition involve conditioning on run number $i_\ell(t)$ to be on the branch leading to $X(t)$.
Let
\[S(t):=\sum_{\ell=1}^{K(t)} F_{i_\ell(t)} + A(t),\]
so that $X(t)\overset{\sss d}{=}  Z\big(S(t)\big)$.

The first step in the proof of Theorem~\ref{monkeyMarkovLimitTheorem} is to prove that
\begin{equation}\label{eq:clt_for_hatS}
\frac{S(t)-\kappa_2s(t)}{\sqrt{\kappa_3s(t)}}
\xrightarrow{\sss d} \mc N(0,1).
\end{equation}
Let us first prove that $A(t)/c_t\to 0$ in probability whenever $c_t\to \infty$. 
Indeed, if $L(t)=T_{N(t)+1}-T_{N(t)}$ is difference of the relocation times that straddle $t$, then $A(t)\leq L(t)$ and $L(t)$ converges in law to a finite random variable (see, e.g.~\cite{bertoin1999renewal}). 
%% Find reference. In discrete renewall this is an analysis of renewal equation.
%% Need something like this in continuous renewall setting, related to Blackwell's renewal theorem. 
%@article{bertoin1999renewal,
%  title={Renewal theory and level passage by subordinators},
%  author={Bertoin, Jean and Van Harn, K and Steutel, Frederik Willem},
%  journal={Statistics \& probability letters},
%  volume={45},
%  number={1},
%  pages={65--69},
%  year={1999},
%  publisher={Elsevier}
%}
Hence, by Slutsky's theorem,  $L(t)/c_t\to 0$ in probability and so does $A(t)/c_t$. 

Note that, given 
%$(X(s), s<t-A(t))$ and $\big(\tau_{N(t)},\ell_{N(t)}\big)$, 
$\tau_{N(t)}$, 
the random node $\omega_{N(t)}$ is chosen at random 
in $\tau_{N(t)}$ with probability proportional to the weights.
Moreover, using the notations of Section~\ref{sec:wrrt}, 
we have that $S(t) = \Phi(\omega_{N(t)}) + A(t)$.
%implying that, in 
Note that $N(t)$ is $\bs L$-measurable, and that, $\bs L$-almost surely, $N(t)$ increases to $+\infty$. 
The convergence of the first marginal in Theorem~\ref{th:sums_along_branches} applies under the assumptions of Theorem~\ref{monkeyMarkovLimitTheorem}; in particular, we have assumed that $\delta\in(0,\nicefrac12]$ if $\mu=\mu_2$. Therefore, we get that, conditionally on~$\bs L$, %(the convergence of the marginals is enough),
\begin{equation}\label{eq:TclPhi}
\frac{\Phi\big(\omega_{N(t)}\big)-\hat\kappa_2s(N(t))}{\sqrt{\hat\kappa_3s(N(t))}}
\xrightarrow{\sss d} \mc N(0,1),
\end{equation}
when $t\to\infty$.
(Hence, the above limit theorem also holds unconditionally.) 
Recall that $S(t) = \Phi(\omega_{N(t)}) + A(t)$, 
and $A(t)/\sqrt{s(t)}\xrightarrow{\sss p} 0$;
thus, to prove Equation~\eqref{eq:clt_for_hatS}, 
it is enough to prove that, 
for all $i\in\{2,3\}$, 
\begin{equation}
\label{eq:changes(N(t))tos(t)}
\hat \kappa_i s(N(t))=\kappa_i s(t)+o\big(\sqrt{s(t)}\big). 
\end{equation}

Let us first assume that $\mu=\mu_1$; in that case, $\kappa_i=\hat\kappa_i$.
By definition, we have that
$\sum_{i=1}^{N(t)} L_i  + A(t) = t$, 
and by the strong law of large numbers,
$\sum_{i=1}^{N(t)} L_i /N(t)\to \mathbb EL$ almost surely when $t\to\infty$.
Therefore, we have $N(t) = t(1+o(1))/\mathbb EL$ almost surely when $t$ goes to infinity, implying that
$\log^\alpha N(t) = \log^\alpha t + \imf{\mathcal O}{1}$, which in turn implies Equation~\eqref{eq:changes(N(t))tos(t)}. 

If $\mu=\mu_2$, we need a more precise estimate for $N(t)$: since $\sum^n_{i=1} L_i-n\mathbb E L=\imf{o}{\sqrt{n\log n}}$ almost surely (for example by the law of the iterated logarithm), the bounds $t-A(t)\leq \sum_{i=1}^{N(t)} L_i\leq t$ imply that
\[
\frac{t-N(t)\mathbb E L}{\sqrt{N(t)\log N(t)}}\to 0
\]almost surely as $t\to\infty$ and thus
\[
N(t)=\frac{t}{\mathbb E L}+\imf{o}{\sqrt{t\log t}}. 
\]Hence we have\[
\hat\kappa_i s(N(t))
=\hat\kappa_iN(t)^\delta
=\hat\kappa_i\left( \frac{t}{\mathbb E L}\right)^{\!\!\delta}+ o\paren{\sqrt{\log(t)} t^{\delta-\nicefrac12}}. 
\]Since $t^{\delta-\nicefrac12}=o(t^{\nicefrac\delta 2})$ for all $\delta\in (0,1)$ and $\kappa_i=\hat \kappa_i/(\mathbb E L)^\delta$, we see that Equation~\eqref{eq:changes(N(t))tos(t)} holds. 

Thus, in both cases, using the fact that $S(t) = \Phi(\omega_{N(t)}) + A(t)$, and since $A(t)/\sqrt{s(t)}\to 0$ in probability, we deduce that Equation~\eqref{eq:clt_for_hatS} holds.

We now show that Theorem~\ref{monkeyMarkovLimitTheorem} follows from Equation~\eqref{eq:clt_for_hatS}.
By Skorokhod's representation theorem, 
we can find a probability space on which 
there exists a standard Gaussian random variable~$\Lambda$ 
as well as $\tilde S(t)$ such that $\tilde S(t) \overset{\sss d}{=} S(t)$ for all $t\geq 0$, such that
\begin{linenomath}
\[
\tilde S(t) = \kappa_2 s(t) + \Lambda \sqrt{\kappa_3 s(t)} + 
o\big(\sqrt{s(t)}\big), 
%\eta_t
\]
%where $\eta_t = o(\sqrt{\log t})$, 
almost surely when $t\to\infty$. 
\end{linenomath}Since, in our original probability space, $Z$ is independent of $S$, 
on an extension of the probability space where $(\tilde S(t))_{t\geq 0}$ is defined, 
there exists a sequence $(\tilde Z(t))_{t\geq 0}$ such that $\tilde Z$ is independent of $\tilde S$, 
$\tilde Z(t)\overset{\sss d}{=} Z(t)$ for all $t\geq 0$, and a random variable $\Gamma$ 
with law $\gamma$ such that $(\tilde Z(t)-b_t)/a_t\to \Gamma$ almost surely. 
Therefore, almost surely when $t\to\infty$, we have
\begin{linenomath}
\begin{align*}
\frac{\tilde Z(\tilde S(t)) - b_{\kappa_2 s(t)}}{a_{\kappa_2 s(t)}}
&= \frac{\tilde Z(\tilde S(t)) - b_{\tilde S(t)}}{a_{\tilde S(t)}}
\cdot
\frac{a_{\tilde S(t)}}{a_{\kappa_2 s(t)}}
+\frac{b_{\tilde S(t)} - b_{\kappa_2 s(t)}}{a_{\kappa_2 s(t)}}\\
& \to \imf{f}{\Lambda\sqrt{\nicefrac{\kappa_3}{\kappa_2}}}\Gamma +\imf{g}{\Lambda\sqrt{\nicefrac{\kappa_3}{\kappa_2}}}
\end{align*}
\end{linenomath}by the \defin{(Scaling)} hypothesis;
this concludes the proof since $\Omega:=\Lambda\sqrt{\nicefrac{\kappa_3}{\kappa_2}}\sim \mathcal N(0, 2\mathbb E L^3/(3\mathbb E L^2))$,
and since $\tilde Z(\tilde S(t)) \overset{\sss d}{=} Z\big(S(t)\big) \overset{\sss d}{=} X(t)$ for all $t\geq 0$.

\subsection{Proof of Theorem~\ref{th:LLT}}
Using the notations of Subsection~\ref{sub:CLT}, we have 
$X(t) \overset{\sss d}{=} Z(S(t))$.
Since $Z$ is independent of $S$, we get that, for all $m\in\mathbb Z^d$
\[\mathbb P(X(t) = m)
= \mathbb E\Big[\mathbb P\big(Z(S(t))=m\,|\,S(t)\big)\Big],\] 
implying that
\begin{linenomath}
\begin{align}
\label{eq:localLimitBound}
&a_{\kappa_2 s(t)}\left|\mathbb P(X(t)=m) - \frac{1}{a_{\kappa_2 s(t)}}\psi\Big(\frac{m - b_{\kappa_2s(t)}}{a_{\kappa_2s(t)}}\Big)\right|\nonumber \\
&\nonumber\leq
a_{\kappa_2 s(t)}\left|\mathbb E\bigg[\mathbb P\big(Z(S(t))=m\,|\,S(t)\big) 
- \frac{1}{a_{\kappa_2 s(t)}}\phi\Big(\frac{m - b_{S(t)}}{a_{S(t)}}\Big)\bigg]\right|
+ \left|\mathbb E\phi\Big(\frac{m - b_{S(t)}}{a_{S(t)}}\Big) 
- \psi\Big(\frac{m - b_{\kappa_2s(t)}}{a_{\kappa_2s(t)}}\Big)\right|\\
&\leq
a_{\kappa_2 s(t)}\mathbb E\left[\left|\mathbb P\big(Z(S(t))=m\,|\,S(t)\big) 
- \frac{1}{a_{\kappa_2 s(t)}}\phi\Big(\frac{m - b_{S(t)}}{a_{S(t)}}\Big)\right|\right]
+ \left|\mathbb E\phi\Big(\frac{m - b_{S(t)}}{a_{S(t)}}\Big) 
- \psi\Big(\frac{m - b_{\kappa_2s(t)}}{a_{\kappa_2s(t)}}\Big)\right|.
\end{align}
\end{linenomath}Since $Z$ is independent of $S$, 
and $S(t)\to\infty$ almost surely when $t\to\infty$, 
the fact that
\[\sup_{m\in\mathbb Z^d} a_t\left|\mathbb P(Z(t) = m) - \frac{1}{a_t}\phi\Big(\frac{m - b_t}{a_t}\Big)\right|\to 0\]
implies that, almost surely when $t\to\infty$,
\[ 
\imf{\Delta}{t}=
\sup_{m\in\mathbb Z^d} a_{S(t)}\left|\mathbb P(Z(S(t)) = m \,|\,S(t)) 
-\frac{1}{a_{S(t)}}\phi\Big(\frac{m - b_{S(t)}}{a_{S(t)}}\Big)\right|\to 0.\]
We can therefore bound the first {\cec term} in the right-hand side of \eqref{eq:localLimitBound} as follows: for all $m\in \z^d$
\begin{linenomath}
\begin{align*}
&a_{\kappa_2 s(t)}\mathbb E\left|\mathbb P\big(Z(S(t))=m\,|\,S(t)\big) 
- \frac{1}{a_{\kappa_2 s(t)}}\phi\Big(\frac{m - b_{S(t)}}{a_{S(t)}}\Big)\right|
\\
&\leq a_{\kappa_2 s(t)}\mathbb E\left|\mathbb P\big(Z(S(t))=m\,|\,S(t)\big) 
- \frac{1}{a_{S(t)}}\phi\Big(\frac{m - b_{S(t)}}{a_{S(t)}}\Big)\right|
+a_{\kappa_2 s(t)}\abs{\frac{1}{a_{S(t)}}-\frac{1}{a_{\kappa_2s(t)}}}
\phi\Big(\frac{m - b_{S(t)}}{a_{S(t)}}\Big)
\\
&\leq \frac{a_{S(t)}}{a_{\kappa_2 s(t)}} \imf{\Delta}{t}+ \abs{\frac{a_{\kappa_2 s(t)}}{a_{S(t)}}-1}
\|\phi\|_\infty.
\end{align*} 
\end{linenomath}
(Recall that $\phi$ is bounded by assumption.) 
Since $\Psi(t)\to 0$ as $t\to\infty$ and, since $f=1$ by assumption, we see that $a_{\kappa_2 s(t)}/a_{S(t)}\to 1$ by the \defin{(Scaling)} assumption. 
Therefore
\begin{equation}
\label{eq:secondLocalLimitBound}
\sup_{m\in\z}a_{\kappa_2 s(t)}\mathbb E\bigg[\left|\mathbb P\big(Z(S(t))=m\,|\,S(t)\big) 
- \frac{1}{a_{\kappa_2 s(t)}}\phi\Big(\frac{m - b_{S(t)}}{a_{S(t)}}\Big)\right|\bigg]\to 0
\end{equation}as $t\to\infty$. 

We now consider the second term in the right-hand side of Inequality~\eqref{eq:localLimitBound}. 
{\cec Using Skorokhod's representation theorem (similarly to the way it was used in the proof of Theorem~\ref{monkeyMarkovLimitTheorem}), we find a probability space on which we can define $\tilde S(t)$ for all $t\geq 0$ such that,} for all~$t\geq 0$, 
$S(t)\overset{\sss d}{=} \tilde S(t)$ and, almost surely when $t\to\infty$,
\[\tilde S(t) = \kappa_2s(t) + \Lambda\sqrt{\kappa_3s(t)} + o(\sqrt{s(t)}).\]
We apply the {\bf (Scaling)} assumption, which gives
\begin{linenomath}
\begin{align*}
\mathbb E\phi\Big(\frac{m - b_{S(t)}}{a_{S(t)}}\Big) 
= \mathbb E\phi\Big(\frac{m - b_{\tilde S(t)}}{a_{\tilde S(t)}}\Big) 
%=\mathbb E\left[\mathbb E\bigg[\phi\Big(\frac{m - b_{\kappa_2\log t}}{a_{\kappa_2\log t}}
%\frac{a_{\kappa_2\log t}}{a_{\tilde S(t)}} + \frac{b_{\kappa_2\log t}-b_{\tilde S(t)}}{a_{\tilde S(t)}}\Big)\bigg| \tilde S(t)\bigg]\right]
=\mathbb E\phi\Big(\frac{m - b_{\kappa_2s(t)}}{a_{\kappa_2s(t)}} +\varepsilon_t - g(\Omega)\Big) ,
\end{align*}\end{linenomath}where $\varepsilon_t$ is $S(t)$-measurable
and tends to zero almost surely when $t\to\infty$, and $\Omega = \Lambda\sqrt{\nicefrac{\kappa_3}{\kappa_2}}$.
Note that, by definition,
\[\psi(x) 
= \frac{\mathrm d}{\mathrm d x}\mathbb P\big( \Gamma + g(\Omega)\leq x\big)
= \frac{\mathrm d}{\mathrm d x}
\mathbb P\big(\Gamma \leq x-g(\Omega) \big)
=\mathbb E\phi\big( x-g(\Omega) \big).\]
Hence, we get
\[\mathbb E\phi\Big(\frac{m - b_{S(t)}}{a_{S(t)}}\Big) 
= \mathbb E\psi\Big(\frac{m - b_{\kappa_2s(t)}}{a_{\kappa_2s(t)}} +\varepsilon'_t\Big),
\]
where $\varepsilon'_t$ is $S(t)$-measurable
and tends to zero almost surely when $t\to\infty$.
Note that, since $\phi$ is bounded and Lipschitz, then so is $\psi$;
we denote by $\|\psi\|_\infty$ its bound and by $\vartheta$ its Lipschitz constant, and get
\begin{linenomath}
\begin{align*}
\sup_{m\in\mathbb Z^d}\left|\mathbb E\phi\Big(\frac{m - b_{S(t)}}{a_{S(t)}}\Big) 
- \psi\Big(\frac{m - b_{\kappa_2s(t)}}{a_{\kappa_2s(t)}}\Big)\right|
&\leq \sup_{m\in\mathbb Z^d}\left|\mathbb E\psi\Big(\frac{m - b_{\kappa_2s(t)}}{a_{\kappa_2s(t)}} +\varepsilon'_t\Big)
- \psi\Big(\frac{m - b_{\kappa_2s(t)}}{a_{\kappa_2s(t)}}\Big)\right|\\
&\leq \sup_{m\in\mathbb Z^d}2\vartheta \mathbb E[|\varepsilon'_t|\wedge \|\psi\|_\infty]
\to 0,
\end{align*}\end{linenomath}by the dominated convergence theorem.
Therefore, the last convergence to zero together with Equations~\eqref{eq:localLimitBound} and~\eqref{eq:secondLocalLimitBound} implies the claim.

\subsection{Proof of Theorem~\ref{th:rec}}
{\it (a)} Let us first assume that $Z$ is the lazy simple random walk on~$\mathbb Z^d$ started at $0$, 
meaning that, at every time-step it stays at the state it is occupying with constant and positive probability. 
It is standard to consider the lazy version of the simple random walk in order to avoid parity problems. 
It is known (see, e.g.\ \cite{Polya} for the simple symmetric case) that, for all $t\in\{1, 2, \ldots\}$, there exists $c>0$ such that
\[\mathbb P(Z(t) = 0) \sim \frac c{t^{\nicefrac{d}{2}}} \text{ when }t\to\infty.\]
Recall that, $X(t) \overset{\sss d}{=} Z(S(t))$ for all $t\geq 1$,
and that
\begin{equation}\label{eq:LLN_S}
\frac{S(t)}{\hat\kappa_2s(t)} 
%=
%(1+o(1))\,\frac{S(t)}{\mathbb E_{\bs W, \bs U} S(t)}
\to 1,
\end{equation}
almost surely when $t\to\infty$  (see~Theorem~\ref{th:sums_along_branches}).
We use L\'evy's version of the Borel-Cantelli lemma 
(see, e.g.,~\cite[Ch.12, Sec.5]{MR1155402} or \cite[Corollary VII-2-6]{MR0402915}):
Consider $\F^S_t=\sag{S_u: u\leq t}$;
since $Z$ and $S$ are independent, we get that, for all $t\geq 1$,
\[
%\mathbb P(X(t) = 0)
\probac{X(t)=0}{\F^S_t}
= \probac{Z(S(t)) = 0}{\F^S_t}
= \frac{c+o(1)}{S(t)^{\nicefrac{d}{2}}}
= \frac{c+o(1)}{(\hat\kappa_2s(t))^{\nicefrac{d}{2}}},
\]almost surely as $t\to\infty$, where we have used Equation~\eqref{eq:LLN_S}. 

Note that $\sum_t s(t)^{-d/2}=\infty$ when 
$\mu=\mu_1$ or when $\mu=\mu_2$ and $\delta d\leq 2$. 
By the cited extension of Borel-Cantelli, 
we get that $X(t)=0$ infinitely often almost surely. 
By the same result, if $\mu=\mu_2$ and $\delta d>2$ then $X(t)\neq 0$ from a given (random) index onwards, implying that the monkey walk is not recurrent. 

The above argument can be generalized to conclude that, 
for any fixed $\eta>0$, 
the set  $\set{t\geq 0: \|X(t)\|\leq \eta}$ is almost surely finite 
when $\mu=\mu_2$ and $\delta d>2$. 
This holds because the local limit theorem for lazy random walks 
implies the existence of a constant $c>0$ such that, for all $\|x\|\leq \eta$,
\[\mathbb P(Z(t) = x) 
\leq  \frac c{t^{\nicefrac{d}{2}}}.\]
Hence, for all $\eta>0$, we have
\[
%\mathbb P(X(t) = 0)
\probac{\|X(t)\|\leq \eta}{\F^S_t}
\leq  \frac{c+o(1)}{(\hat\kappa_2s(t))^{\nicefrac{d}{2}}}, \text{ when }t\to\infty,
\]
implying that $\|X(t)\|\to \infty$ almost surely as $t\to\infty$. 
\bigskip

{\it (b)} Let us now treat the case when $Z = (Z(t))_{t\in[0,\infty)}$ is the Brownian motion on $\mathbb R^d$. 
Fix $\eta >0$. 
We know that, for an appropriate $c>0$,  $\proba{\|Z(t)\|\leq \eta}=ct^{-d/2}$. 
We can then reproduce the Borel-Cantelli argument (using times $t=1,2,\ldots$) to show that $\mathcal Z$ is unbounded  when $\mu=\mu_1$ or when $\mu=\mu_2$ and $\delta d\leq 2$. 

%Regarding the transience assertion, consider $\mu=\mu_2$ and $\delta d>2$. 
%Here we use a different filtration: consider $\G_i=\F^S_{T_i}\vee \sag{L_i:i\geq 1}$. 
%Then, we have
%\[
%\probac{\|X(T_i)\|\leq \eta}{\G_i}= \frac{c+o(1)}{S(T_i)^{d/2}}.  
%\]Using Theorem \ref{th:sums_along_branches}, we see that, conditionally on $L$,  
%$S(T_i)^{d/2}\sim (\hat \kappa_2 s(T_i))^{d/2}\sim c' T_i^{\delta d/2}\sim c'' i^{\delta d/2}$; 
%the last asymptotic equality holds  because of the strong law of large numbers.  
%Since $\delta d>2$, we get
%\[
%\sum _i\probac{\|X(T_i)\|\leq \eta}{\G_i}<\infty 
%\]for all $\eta>0$. As argued before we deduce that $\|X(T_i)\|\to \infty$ as $i\to\infty$ almost surely. 
%
%Let us now prove that $\|X(t)\|\to\infty$ as $t\to\infty$ almost surely. 
%It  is well known\footnote{See \cite[Corollary VI.3.4, p. 253]{MR1725357} for a full proof when $d=3$; the general case uses the same arguments and the scale functions of XI\S 1 op.\ cit.} that
%\[ 
%\imf{\p_x}{\inf_{t\geq 0}\|Z(t)\|>a}=1-\paren{\frac{a}{\|x\|}}^{d-2}. 
%\]Therefore, for all $i$, 
%\[
%\imf{\p_{x}}{\inf_{t\geq 0}\|Z(t)\|\leq x/i^{2/(d-2)}}=\frac{1}{i^2}. 
%\]Define $\G_i=\sag{L}\vee \F^X_{T_i}$, so that
%\[
%\proba{\cond{\inf_{t\in [T_{i},T_{i+1})}\|X(t)\|\leq \|X(T_i)\|/i^{2/(d-2)}}{\G_i}}
%\leq 
%\imf{\p_X(T_i)}{\inf_{t\geq 0}\|Z(t)\|\leq \|X(T_i)\|/i^{2/(d-2)}}=\frac{1}{i^2}.
%\]By Borel-Cantelli and the fact that \comment{Remove transience for BM since it is technical and kind of a side quest. }

\section{Proof of Theorem~\ref{th:cv_occupation}}\label{sec:occupation_measure}
%Note that $\pi_t$ is the law of $X_{U_t}$ where $U_t$ is uniform on $[0,t]$ and independent of $X$. 
%
%The idea of the proof is to first focus on the behavior 
%of the occupation measure of the monkey Markov process along each branch
%of the branching structure, and, then, use our results on the {\sc wrrt} 
%to deduce that this is enough to imply the statement. 
We reason conditionally on $\bs L$ (and thus $\bs W$) and $\bs F$. % and $\bs U$.
Let us denote by $\varpi_i$ the weighted occupation measure of the $i$-th run: in other words, for all $i\geq 1$, for all Borel set $\mathcal B\subseteq \mathbb R^d$, we let
\[\varpi_i(\mathcal B) 
= \frac{1}{{\bl W_i}}\int_{T_{i-1}}^{T_i} {\bl\mu(s)}\boldsymbol 1_{X(s)\in \mathcal B} \,\mathtt ds,\]
where we recall that the sequence $(T_i)_{i\geq 1}$ with $T_0 = 0$ is the sequence of relocation times: $T_n = L_1+\cdots + L_n$ for all $n\geq 1$.
Recall that we denote by~$A(t)$ the time from $t$ back to the last relocation time, and by $N(t)$ the number of runs before time~$t-A(t)$.
With these notations, we have
\begin{lesn}
\pi_t=\frac{1}{\bl\bar\mu(t)}\bra{\sum_{i=1}^{N(t)}\varpi_i + R_t} 
%I choose R like "rest" because B really looked too much like the Brownian motion
\end{lesn}where {\bl$\bar\mu(t) = \int_0^t \mu$} and
\begin{lesn}
R_t(\mathcal B)=\int_{t-A(t)}^t{\mu(s)}\mathbf{1}_{X_s\in \mathcal B}\, ds, \text{ for all Borel set }\mathcal B\in\mathbb R^d. 
\end{lesn}
%Note that $R_t(\re^d)={\int_{t-A(t)}^t \mu}=\bar\mu(t)-\bar\mu(t-A(t))$. 
{\bl$\bullet$ Let us first prove that \[R_t(\re^d)={\int_{t-A(t)}^t \mu}=\bar\mu(t)-\bar\mu(t-A(t))=o(\bar\mu(t)),\]when $t\to\infty$; since we have already argued that $A(t)$ is negligible in front of any increasing, diverging function of $t$, it only remains to check that our memory kernels do not increase too fast.
First assume that $\mu=\mu_1$; we have 
\[\bar\mu(t)-\bar\mu(t-A(t))
= \frac1\beta \left(\mathrm e^{\beta\log^\alpha t} - \mathrm e^{\beta\log^\alpha(t-A(t))}\right)
=\frac{\mathrm e^{\beta\log^\alpha t} }\beta \left(1-\mathrm e^{\beta[\log^\alpha(t-A(t))-\log^\alpha t]}\right).\]
Note that
\begin{align*}
\log^\alpha(t-A(t))-\log^\alpha t
&= \big(\log^\alpha t\big) \bigg(\bigg(\frac{\log t + \log \big(1-\frac{A(t)}{t}\big)}{\log t}\bigg)^{\!\!\alpha}-1\bigg)\\
&= -(1+o(1))\,\frac{\alpha A(t) \log^\alpha t}{t\log t}
= -(\alpha+o(1))\,\frac{A(t) \log^{\alpha-1} t}{t},
\end{align*}
in probability when $t\to\infty$, since $A(t)/t \xrightarrow{\sss p} 0$.
Therefore, we get
\[\bar\mu(t)-\bar\mu(t-A(t)) = 
\frac{\alpha \mathrm e^{\beta\log^\alpha t}  \log^{\alpha-1} t}{t}
= \alpha\beta\bar\mu(t) \,\frac{A(t)\log^{\alpha-1} t}{t}
= o(\bar\mu(t)),
\]
in probability since $A(t)/c_t\xrightarrow{\sss p} 0$ as soon as $c_t\to+\infty$.
The case $\mu=\mu_2$ can be treated similarly and
we get that, in all cases, $\bar\mu(t-A(t))/\bar\mu(t)\xrightarrow{\sss p} 1$.} 

$\bullet$ It is therefore enough to prove that
\begin{equation}\label{eq:enough}
\frac{1}{\bl\bar\mu(t-A(t))}\sum_{i=1}^{N(t)}\imf{\varpi_i}{a_{\kappa_2s(t)}\cdot +b_{\kappa_2s(t)}}
\xrightarrow{\sss p} \pi_\infty, \text{ weakly}.
\end{equation}
Note that $\bar\mu(t-A(t))=\sum_{i=1}^{N(t)}{W_i}$, and, by definition of the model, conditionally on $(X(s), s\leq t-A(t))$ and $(\tau_{N(t)}, \ell_{N(t)})$, $\frac{1}{\bl\bar\mu(t-A(t))}\sum_{i=1}^{N(t)}\varpi_i$ is the distribution of the relocation position at time $t-A(t)$, which we denote by $V_t$.%\footnote{\cec Not true!}
%Recall that $V_t = \Phi(\omega_{N(t)})$, where $\omega_{N(t)}$ is taken weight-proportionally at random in $\tau_{N(t)}$.

By~\cite[Lemma~3.1]{MR3629870}, 
to prove Equation~\eqref{eq:enough}, it is enough to show that, when $t\to\infty$,
\begin{equation}\label{eq:variance-like-cv}
\left(\frac{V_t^{\sss (1)}- b_{\kappa_2s(t)}}{a_{\kappa_2 s(t)}}, 
\frac{V_t^{\sss (2)}- b_{\kappa_2s(t)}}{a_{\kappa_2 s(t)}}\right)
\xrightarrow{\sss d} (\Xi_1, \Xi_2),
\end{equation}
where $V_t^{\sss (1)}$ and $V_t^{\sss (2)}$ are two independent copies of $V_t$, 
and $\Xi_1$ and $\Xi_2$ two independent random variables 
distributed according to~$\pi_{\infty}$.

\begin{figure}
\begin{center}
\includegraphics[width=.8\textwidth]{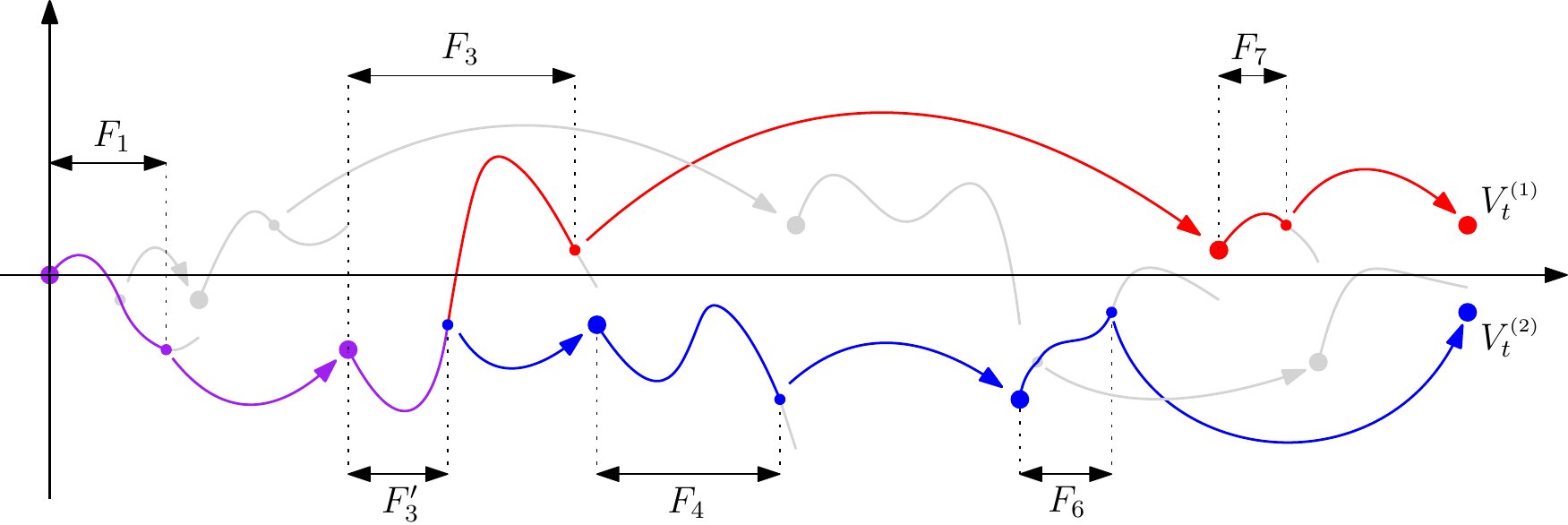}
\end{center}
\caption{This picture illustrates the reasoning of the proof of Theorem~\ref{th:cv_occupation}. In this particular example, the two trajectories leading to $V_t^{\sss (1)}$ and $V_t^{\sss (2)}$ are colored in red and blue respectively, and in purple when they coincide. Note that, in this example, they coincide for an amount of time equal to $\Delta(t)=F_1+F'_3$; when they split, they are both equal to $Z(F_1+F'_3)$, and then become independent. Starting from there, the red trajectory evolves for an amount of time equal to $\Delta^{\sss (1)}(t) = F_3-F'_3+F_7$, and the blue trajectory for an amount of time equal to $\Delta^{\sss (2)}(t) = F_4+F_6$.
Note that, using the notations of the proof of Theorem~\ref{th:cv_occupation}, $k(t)=3$ (the number corresponding to the run during which the red and blue trajectories split).}
\label{fig:proof_th2}
\end{figure}

We refer the reader to Figure~\ref{fig:proof_th2} for a visual aid to the rest of the proof and notations.
Let us denote by $\omega_t$ and $\omega'_t$ taken independently and weight-proportionally at random in $\tau_{N(t)}$;
let $\kappa_t$ be the parent of $\omega_t \wedge \omega'_t$ 
and $k(t)$ the index of node $\omega_t \wedge \omega'_t$ in $\tau_n$ (i.e.\ $\omega_t \wedge \omega'_t = \nu_{k(t)}$);
we have
\[
\big(V_t^{\sss (1)},V_t^{\sss (2)}\big)
= \Big(Z^{\sss (1)}_{Z(\Delta(t))}\big(\Delta_1(t)\big),
Z^{\sss (2)}_{Z(\Delta^{\sss (1)}(t))}\big(\Delta^{\sss (2)}(t)\big)\Big),
\]
where 
$Z^{\sss (1)}_{Z(\Delta(t))}$ and $Z^{\sss (2)}_{Z(\Delta(t))}$ are two independent random Markov processes of semi-group $P$ started at $Z(\Delta(t))$, and
\begin{align*}
\Delta(t) &= \sum_{i=1}^n F_i \bs 1_{\nu_i\preccurlyeq \kappa_t} 
+ \min\{F_{k(t)},F_{k(t)}'\}
=\Phi(\kappa_t) + \min\{F_{k(t)},F_{k(t)}'\},\\
\Delta^{\sss (1)}(t) &= (F_{k(t)} - \min\{F_{k(t)},F_{k(t)}'\}) +
\sum_{i=1}^n F_i \bs 1_{\omega_t\wedge \omega'_t\prec \nu_i\preccurlyeq \omega_t}
=\Phi(\omega_t) - \Delta(t)\\
\Delta^{\sss (2)}(t) &=  (F_{k(t)}' - \min\{F_{k(t)},F_{k(t)}'\}) +
\sum_{i=1}^n F_i \bs 1_{\omega_t\wedge \omega'_t\prec \nu_i\preccurlyeq \omega_t}
=F_{k(t)}'-F_{k(t)} + \Phi(\omega'_t) - \Delta(t),
\end{align*}
where $\bs F'=(F'_i)_{i\geq 1}$ is an independent copy of $\bs F=(F_i)_{i\geq 1}$.
We reason conditionally on ($\bs F$ and)~$\bs {F'}$, and
%Thus, we have
%\[\big(V_t^{\sss (1)},V_t^{\sss (2)}\big)
%=\Big(Z^{\sss (1)}\big(\Phi(\omega_t) - \Phi(\kappa_t)\big),
%Z^{\sss (2)}\big(\Phi(\omega_t) - \Phi(\kappa_t)+F_{k(t)}'-F_{k(t)}\big)\Big),
%\]
%where $F'_i$ is an independent copy of $F_i$, for all $i\geq 1$.
let
\[\mathcal X = \bigg\{\sum_{i\in I} F_i + \min\{F_k,F'_k\}\colon I \text{ is a finite subset of }\mathbb N, k\in\mathbb N\bigg\},\]
{\cec being the set of all possible values for $\Delta(t)$;}
we have, for all $w_1, w_2\in\mathbb R$,
\begin{linenomath}
\begin{align*}
&\mathbb P\left(
\frac{V_t^{\sss (1)} - b_{\kappa_2s(t)}}{a_{\kappa_2s(t)}}\geq w_1,
\frac{V_t^{\sss (2)} - b_{\kappa_2s(t)}}{a_{\kappa_2s(t)}}\geq w_2
\right)\\
%&= \sum_{x\in\mathcal X}
%\mathbb P\left(
%\frac{V_t^{\sss (1)} - b_{\kappa_2s(t)}}{a_{\kappa_2s(t)}}\geq w_1,
%\frac{V_t^{\sss (2)} - b_{\kappa_2s(t)}}{a_{\kappa_2s(t)}}\geq w_2,
%\Phi(\omega_t\wedge \omega'_t) = x
%\right)\\
&= \sum_{x\in\mathcal X}
\mathbb P\left(
\frac{Z^{\sss (1)}_{Z(x)}(\Delta^{\sss (1)}_x(t)) - b_{\kappa_2s(t)}}{a_{\kappa_2s(t)}}\geq w_1,
\frac{Z^{\sss (2)}_{Z(x)}(\Delta^{\sss (2)}_x(t)) - b_{\kappa_2s(t)}}{a_{\kappa_2s(t)}}\geq w_2,
\Delta(t) = x
\right).
\end{align*}\end{linenomath}where 
$\Delta^{\sss (1)}_x(t)= \Phi(\kappa_t) - x$, and 
$\Delta_x^{\sss (2)}(t)=\Phi(\kappa_t) - x + F'_{k(t)} - F_{k(t)}$.
%where $Z_x^{\sss (1)}$ and $Z_x^{\sss (2)}$ are two independent 
%Markov processes of semi-group $\mathcal P$, both started from~$x$.
%The last equality holds since~$\Phi(\omega_t) - \Phi(\omega_t\wedge \omega'_t)$ and $\Phi(\omega'_t) - \Phi(\omega_t\wedge \omega'_t)$ are independent by definition.
We have\begin{linenomath}
\begin{align}
\mathbb P\left(
\frac{V_t^{\sss (1)} - b_{\kappa_2s(t)}}{a_{\kappa_2s(t)}}\geq w_1,
\frac{V_t^{\sss (2)} - b_{\kappa_2s(t)}}{a_{\kappa_2s(t)}}\geq w_2
\right)
\hspace{6cm}&\notag
\\
= \sum_{x\in\mathcal X}
\mathbb P\bigg(
\frac{Z^{\sss (1)}_{Z(x)}(\Delta^{\sss (1)}_x(t))  - b_{\Delta^{\sss (1)}_x(t)}}{a_{\Delta^{\sss (1)}_x(t)}}\cdot\frac{a_{\Delta^{\sss (1)}_x(t)}}{a_{\kappa_2 s(t)}} 
+ \frac{b_{\Delta^{\sss (1)}_x(t)}-b_{\kappa_2 s(t)}}{a_{\kappa_2 s(t)}}\geq w_1,&\notag\\
\frac{Z^{\sss (2)}_{Z(x)}(\Delta^{\sss (2)}_x(t))  - b_{\Delta^{\sss (2)}_x(t)}}{a_{\Delta^{\sss (2)}_x(t)}}\cdot\frac{a_{\Delta^{\sss (2)}_x(t)}}{a_{\kappa_2 s(t)}} 
+ \frac{b_{\Delta^{\sss (2)}_x(t)}-b_{\kappa_2 s(t)}}{a_{\kappa_2 s(t)}}\geq w_2,&
~\Delta(t) = x
\bigg);\label{eq:truc}
\end{align}\end{linenomath}call the summands in the right-hand side $\imf{p_t}{w_1,w_2,x}$. 
We now prove that for all $w_1,w_2\geq 0$ and~$x\in\mathcal X$, 
\begin{linenomath}
\begin{align}
\label{eq:ptdef}
\imf{p_t}{w_1,w_2,x}
\to& \mathbb P\Big(f(\Omega_1)\Gamma_1 + g(\Omega_1)\geq w_1,
f(\Omega_2)\Gamma_2 + g(\Omega_2)\geq w_2,
\Theta = x\Big)\\
&=: \imf{p_\infty}{w_1,w_2, x}, \nonumber
\end{align} 
\end{linenomath}where $\Omega_1$ and $\Omega_2$ are two independent variables of distribution $\mathcal N(0,\nicefrac{\kappa_3}{\kappa_2})$.
Almost surely when $t\to\infty$, for all $x\in\mathcal X$,  
$\Delta^{\sss (1)}_x(t)\to\infty$, and $\Delta^{\sss (2)}_x(t)\to\infty$;
this implies that, when~$t\to\infty$,
\begin{equation}\label{eq:cv_Zs}
\left(\frac{Z_{Z(x)}^{\sss (1)}(\Delta^{\sss (1)}_x(t)) - b_{\Delta^{\sss (1)}_x(t)}}{a_{\Delta^{\sss (1)}_x(t)}},
\frac{Z_{Z(x)}^{\sss (2)}(\Delta^{\sss (2)}_x(t)) - b_{\Delta^{\sss (2)}_x(t)}}{a_{\Delta^{\sss (2)}_x(t)}}\right)
\xrightarrow{\sss d} (\Gamma_1, \Gamma_2),
\end{equation}
where
$\Gamma_1$ and $\Gamma_2$ are two independent random variables of distribution~$\gamma$.
By ergodicity of the Markov process of semi-group $\mathcal P$,
$\Gamma_1$ and $\Gamma_2$ are independent of the starting point~$Z(x)$ of~$Z_{Z(x)}^{\sss (1)}$ and~$Z_{Z(x)}^{\sss (2)}$.
Using Theorem~\ref{th:sums_along_branches} (we have indeed assumed that $\delta\in(0,\nicefrac12)$ if $\mu=\mu_2$), we get that, conditionally on $\bs L$ and $\bs F$, %and $\bs U$,
\begin{equation}\label{eq:sum_along_branches}
\left(\frac{\Phi(\omega_t)- \hat \kappa_2 s(N(t))}{\sqrt{\hat \kappa_3s( N(t))}},
\frac{\Phi(\omega'_t) - \hat\kappa_2s(N(t)) }{\sqrt{\hat\kappa_3s( N(t))}}
\right)
\xrightarrow{\sss d} (\Lambda_1, \Lambda_2),
\end{equation}
where $\Lambda_1$ and $\Lambda_2$ are 
two independent standard Gaussian random variables.
By Slutsky's lemma and Equation~\eqref{eq:changes(N(t))tos(t)}, 
Equation~\eqref{eq:sum_along_branches} becomes:  
conditionally on $\bs L$ and $\bs F$, %and $\bs U$, 
we have
\begin{equation}\label{eq:cv_Phis}
\left(\frac{\Phi(\omega_t) - \kappa_2 s(t)}{\sqrt{\kappa_3 s(t)}}, 
\frac{\Phi(\omega'_t) - \kappa_2 s(t)}{\sqrt{\kappa_3  s(t)}}\right)
\xrightarrow{\sss d} (\Lambda_1, \Lambda_2),
\end{equation}
when $t\to\infty$,
where $\Lambda_1$ and $\Lambda_2$ are independent. 
Using Proposition~\ref{prop:2Dobrow} and Borel-Cantelli's lemma, we get that~$k(t)$ converges in distribution to an almost surely finite integer $K=\max\{i\geq 1\colon B_i=B'_i=1\}$.
Thus, conditionally on $\bs F$ and $\bs F'$, 
we have that~$(F'_{k(t)}-F_{k(t)})/\sqrt{s(t)}\to 0$ in probability.
Therefore, by Equation~\eqref{eq:cv_Phis} and Slutzky's lemma, we get that, for all $x\in\mathcal X$,
\begin{equation}\label{eq:cv_Deltas}
\left(\frac{\Delta_x^{\sss (1)}(t) - \kappa_2 s(t)}{\sqrt{\kappa_3 s(t)}}, 
\frac{\Delta_x^{\sss (2)}(t) - \kappa_2 s(t)}{\sqrt{\kappa_3  s(t)}}\right)
\xrightarrow{\sss d} (\Lambda_1, \Lambda_2),
\end{equation}

Therefore, Assumption~{\bf (Scaling)},~Equations~\eqref{eq:cv_Zs} and~\eqref{eq:cv_Deltas} imply Equation~\eqref{eq:ptdef}. 
We now aim to apply the general Lebesgue dominated convergence theorem (see, e.g.\ \cite{MR1013117}): 
First, we see that $p_t(w_1,w_2,x)\leq p_t(0,0,x)$ for all $w_1,w_2\geq 0$. 
From Equation~\eqref{eq:truc}, we get that $1=\sum p_t(0,0,x)$.
The forthcoming Lemma~\ref{lem:LCA} implies that $\Phi(\omega_t\wedge \omega'_t)$ converges in distribution to $\Theta$, an almost surely finite random variable.
Thus, we have
\[\Delta(t) = \phi(\omega_t\wedge \omega'_t) +\min\{F_{k(t)},F'_{k(t)}\} - F_{k(t)}
\xrightarrow{\sss d} \Theta +\min\{F_{K},F'_{K}\} - F_{K} =\hat\Theta \in\mathcal X,\]
implying that $1=\sum p_\infty(0,0,x)$.  
Therefore, by dominated convergence, and since $(\Omega_1, \Omega_2)$ is equal in distribution to $\big(\Lambda_1\sqrt{\nicefrac{\kappa_3}{\kappa_2}},\Lambda_2\sqrt{\nicefrac{\kappa_3}{\kappa_2}}\big)$, we get
\begin{linenomath}
\begin{align}
&\mathbb P\left(
\frac{V_t^{\sss (1)} - b_{\kappa_2 s(t)}}{a_{\kappa_2 s(t)}}\geq w_1,
\frac{V_t^{\sss (2)} - b_{\kappa_2 s(t)}}{a_{\kappa_2 s(t)}}\geq w_2
\right)\notag\\
&\hspace{1.5cm}\to \sum_{x\in\mathcal X}
\mathbb P\Big(f(\Omega_1)\Gamma_1 + g(\Omega_1)\geq w_1,
f(\Omega_2)\Gamma_2 + g(\Omega_2)\geq w_2,
\hat\Theta = x\Big)\notag\\
&\hspace{2cm}= \mathbb P\big(f(\Omega_1)\Gamma_1 + g(\Omega_2)\geq w_1,
f(\Omega_2)\Gamma_2 + g(\Omega_2)\geq w_2\big),\label{eq:cv_occupation}
\end{align}\end{linenomath}where we have used again that $\hat\Theta\in\mathcal X$. 
We have thus proved Equation~\eqref{eq:variance-like-cv}, which implies the claim.

\section{Proof of Theorem~\ref{th:sums_along_branches}}\label{sec:proof_prop}
\subsection{Main steps of the proofs and heuristics}\label{sub:plan_proof}
{\cec Recall that 
\[\Phi(u_n) = \sum_{\nu_i\preccurlyeq u_n} F_i = \sum_{i=1}^n F_i\bs 1_{\nu_i\preccurlyeq u_n},\]
where the $F_i$'s are defined in Equation~\eqref{eq:distr_F}, and $u_n$ is a node taken weight-proportionally at random in the $n$-node {\sc wrrt}.}
To prove Theorem~\ref{th:sums_along_branches}, 
we proceed along the following steps (where the indices $\bs L, \bs F$ indicate that the expectation is taken conditionally on the sequences $\bs L$ and $\bs F$):
\begin{enumerate}[(a)]
\item We first prove that, conditionally on $\bs L$ and $\bs F$,
\begin{equation}\label{eq:CLT}
\frac{\Phi(u_n) - \mathbb E_{\bs L, \bs F}\Phi(u_n)}{\sqrt{\mathtt{Var}_{\bs L, \bs F}\Phi(u_n)}}
\xrightarrow{\sss d} \mathcal N(0,1);
\end{equation}
%and similarly for $\Psi$;
this is done by applying Lindeberg's central limit theorem. {\cec The rest of the proof aims at finding almost sure estimates for the random variables $\mathbb E_{\bs L, \bs F}\Phi(u_n)$ and $\mathtt{Var}_{\bs L, \bs F}\Phi(u_n)$.}
\item {\cec Using Proposition~\ref{prop:Dobrow}, it is straightforward that
\[\mathbb E_{\bs L,\bs F}\Phi(u_n)
= \sum_{i=1}^n \frac{F_i W_i}{S_i}
\quad\text{ and }\quad
\mathtt{Var}_{\bs L,\bs F}\Phi(u_n) = \sum_{i=1}^n  \frac{F_i W_i}{S_i}\bigg(1-\frac{W_i}{S_i}\bigg),\]
where $S_i = \sum_{j=1}^i W_i = \bar\mu(T_i)$ (recall that $\bar \mu(x)=\int_0^x \mu$, for all $x\geq 0$).}
\item {\cec Using martingale arguments, we prove that $\mathbb E_{\bs L,\bs F}\Phi(u_n)$ and $\mathtt{Var}_{\bs L,\bs F}\Phi(u_n)$ are asymptotically equivalent to their respective expectations (conditionally on $\bs L$), which are given by
\[\sum_{i=1}^n \frac{W_i \mathbb E_{\bs L} [F_i]}{S_i}
\quad\text{ and }\quad
\sum_{i=1}^n \frac{W_i \mathbb E_{\bs L} [F_i]}{S_i}\bigg(1-\frac{W_i}{S_i}\bigg),\]
respectively. This is done in Lemma~\ref{lem:sumsF_iW_i}.}
\item The expectation of $F_i$ given $\bs L$ is calculated in Lemma~\ref{lem:exp_Fi}, and the sums of (b) are estimated in Lemma~\ref{lem:sums}, {\cec where we use that, by the law of large numbers, $T_i \sim i\mathbb E L$, implying that $S_i \sim \bar\mu(i\mathbb EL)$.} We then infer that, $(\bs L,\bs F)$-almost surely when $n\to\infty$,
\[\mathbb E_{\bs L,\bs F}\Phi(u_n) = \hat\kappa_2 s(n) + o(\sqrt{s(n)})
\quad\text{ and }\quad
\mathtt{Var}_{\bs L,\bs F}\Phi(u_n) = \hat\kappa_3s(n) + o(s(n)),\]
where $s(n)$ is defined in Equation~\eqref{eq:def_s}, and $\hat \kappa_i$ in Theorem~\ref{th:sums_along_branches}.
\end{enumerate}
These four steps lead to the following result:
\begin{proposition}\label{prop:marginals}
Under the assumptions of Theorem~\ref{th:sums_along_branches},
conditionally on $(\bs L, \bs F)$, $(\bs L, \bs F)$-almost surely
\[\frac{\Phi(u_n) - \hat\kappa_2s(n)}{\sqrt{\hat\kappa_3 s(n)}}
\xrightarrow{\sss d} \mathcal N(0,1),\]
when $n\to\infty$.
\end{proposition}
The rest of the proof of Theorem~\ref{th:sums_along_branches} is made by using the fact that the height of the last common ancestor of $u_n$ and $v_n$ converges almost surely to a finite random variable. This is done in Section~\ref{sub:last_proof}.

\subsection{Preliminary lemmas}
Before proceeding with the proofs of Proposition~\ref{prop:marginals} and Theorem~\ref{th:sums_along_branches}, we need to prove the following preliminary lemmas. {\cec The first of these lemmas gives a good approximation for the expectation of $F_i$ conditionally on $\bs L$; this correspond to Step (d) in the plan of the proof of Section~\ref{sub:plan_proof}.}
\begin{lemma}\label{lem:exp_Fi}
Under the hypotheses of Theorem \ref{monkeyMarkovLimitTheorem} and with $\delta\in [0,1/2]$ for the memory kernel $\mu_2$, 
for all $\ell\in\{2, 3, \ldots \}$, for all $i\in\{0, 1\ldots \}$, 
we have
\[W_{i+1}\mathbb E_{\bs L}[F_{i+1}^{\ell-1}] 
=\mu(T_i) R_{i+1}^{\sss (\ell)},\]
%where, if $L_{i+1}/T_i\to0$ when $i\to\infty$, then,
and there exists a random integer $I_0$ and a constant $c>0$ such that, for all $i\geq I_0$,
\[\left|R_{i+1}^{\sss (\ell)}-\frac{L^{\ell}_{i+1}}{\ell}\right|\leq 
\begin{cases}
\displaystyle\frac{cL_{i+1}^{\ell+1}(\log T_i)^{\tilde\alpha-1}}{T_i} & \text{ if }\mu=\mu_1\\
&\\
\displaystyle\frac{cL_{i+1}^{\ell+1}}{T_i^{1-\delta}} & \text{ if }\mu=\mu_2,
\end{cases}\]
where $\tilde\alpha = \max\{1, \alpha\}$.
\end{lemma}

\begin{proof}
%Since $\esp{L_i}<\infty$, the Borel-Cantelli lemma implies that $L_i/i\to 0$ almost surely. 
%By the law of large numbers, we also get that $L_i/T_i\to 0$ almost surely. 
Since $\delta\in [0,1/2)$, then $1/(1-\delta)\in [1,2)$. 
Then, the Borel-Cantelli lemma implies that $L_i/i^{1-\delta}\to 0$ almost surely. 
By the law of large numbers, we also get that $L_i/T_i^{1-\delta}$ and thus $L_i/T_i$ both converge to~$0$ almost surely.

Note that
\[
\mathbb E_{\bs L}F_{i+1}^{\ell-1}
= \frac{\int_{T_i}^{T_{i+1}}  (u-T_i)^{\ell-1} \mu(u)\,\mathrm du}{\int_{T_i}^{T_{i+1}} \mu(u)\mathrm du}
=\frac1{W_i}\int_{T_i}^{T_{i+1}}  (u-T_i)^{\ell-1} \mu(u)\,\mathrm du,
%\approx\mathtt e^{\log^\alpha((i+1)m)} - \mathtt e^{\log^\alpha(im)},
\]
and, therefore,
\[
W_{i+1}\mathbb E_{\bs L}[F_{i+1}^{\ell-1}] 
= \int_{T_i}^{T_{i+1}} (u-T_i)^{\ell-1} \,\mu(u)\mathrm du
= \int_{0}^{L_{i+1}} u^{\ell-1} \,\mu(T_i+u)\mathrm du
= \mu(T_i) R^{\sss \ell}_{i+1},
\]
where 
\begin{align*}
R^{\sss (\ell)}_{i+1}
= \int_{0}^{L_{i+1}} u^{\ell-1} \, \frac{\mu(T_i+u)}{\mu(T_i)}\mathrm du.
%=\int_0^{L_{i+1}} &\frac{u^{\ell-1}}{1+\nicefrac u{T_i}} \left(1+\frac{\log(1+\nicefrac u{T_i})}{\log T_i}\right)^{\alpha-1}\exp\!\left((\log T_i)^\alpha\Big[\Big(1+\frac{\log(1+\nicefrac u{T_i})}{\log T_i}\Big)^\alpha-1\Big]\right)\mathrm du.
\end{align*}
We treat each memory kernel separately:

\medskip
(1.1) We first assume that $\mu =\mu_1$, implying that
\[R^{\sss (\ell)}_{i+1}
=\int_0^{L_{i+1}} \frac{u^{\ell-1}}{1+\nicefrac u{T_i}} \left(1+\frac{\log(1+\nicefrac u{T_i})}{\log T_i}\right)^{\alpha-1}\exp\!\left(\beta(\log T_i)^\alpha\Big[\Big(1+\frac{\log(1+\nicefrac u{T_i})}{\log T_i}\Big)^\alpha-1\Big]\right)\mathrm du.\]
Since $L_{i+1}/T_i\to0$ when $i\to\infty$, then there exists a random integer $I_0$ such that, for all $i\geq I_0$, $0\leq L_{i+1}/T_i\leq 1$.
Let us first assume that $\alpha>1$:
We have, for all $i\geq I_0$, using that $0\leq u\leq L_{i+1}$ in the integrand,
\[R^{\sss (\ell)}_{i+1}\geq \frac{L^{\ell}_{i+1}/\ell}{1+L_{i+1}/T_i}\geq \frac{L^{\ell}_{i+1}(1-c_1 L_{i+1}/T_{i})}{\ell},\]
where $c_1 = \sup_{x\in[0,1]}\big(\frac1x(1-\frac1{(1+x)})\big)$.
There exists $I_1>I_0$ such that, for all $i\geq I_1$, $\log T_i\geq 1$, implying in particular that $L_{i+1}/(T_i\log T_i)\leq 1$.
Moreover, we have
\begin{linenomath}
\begin{align*}
R^{\sss (\ell)}_{i+1}
&\leq 
\frac{L^{\ell}_{i+1}}{\ell} \left(1+\frac{\log(1+L_{i+1}/T_i)}{\log T_i}\right)^{\!\alpha-1}
\exp\left(\beta(\log T_i)^{\alpha}\Big[\Big(1+\frac{\log(1+L_{i+1}/T_i)}{\log T_i}\Big)^{\!\alpha}-1\Big]\right)\\
&\leq 
\frac{L^{\ell}_{i+1}}{\ell} \left(1+\frac{c_2L_{i+1}}{T_i\log T_i}\right)^{\!\alpha-1}
\exp\left(\beta(\log T_i)^\alpha\Big[\Big(1+\frac{c_2L_{i+1}}{T_i\log T_i}\Big)^{\!\alpha}-1\Big]\right)\\
&\leq 
\frac{L^{\ell}_{i+1}}{\ell} \left(1+\frac{c_3 L_{i+1}}{T_i\log T_i}\right)
\exp\left(\frac{c_4L_{i+1}(\log T_i)^{\alpha-1}}{T_i}\right),
\end{align*} \end{linenomath}where $c_2 = \sup_{[0,1]} \frac1x \log(1+x)$, 
$c_3=\sup_{[0,1]}\frac1x((1+c_2x)^\alpha-1)$, and
$c_4=\sup_{[0,1]}\frac{\beta}{x}((1+c_2x)^\alpha-1)$.
Finally, there exists $I_\alpha>I_1$ such that, for all $i\geq I_\alpha$, $L_{i+1}(\log T_i)^{\alpha-1}/T_i\leq 1$, and therefore, we get that, for all $i\geq I_\alpha$,
\[
R^{\sss (\ell)}_{i+1}
\leq\frac{L^{\ell}_{i+1}}{\ell} \left(1+\frac{c_3 L_{i+1}}{T_i\log T_i}\right)
\left(1+\frac{c_5L_{i+1}(\log T_i)^{\alpha-1}}{T_i}\right),
\]
where $c_5 = \sup_{[0,1]}\frac1x (\mathrm e^{c_4x}-1)$, and, in total, for all $i\geq I_1$,
\[R^{\sss (\ell)}_{i+1}
\leq \frac{L^{\ell}_{i+1}}{\ell}\left(1+\frac{cL_{i+1}(\log T_i)^{\alpha-1}}{T_i}\right),\]
where $c=c_3 + c_5 +  c_3 c_5$,
which concludes the proof.

\medskip
(1.2) Similarly, if $\alpha\leq 1$, we have 
\begin{linenomath}
\begin{align*}
R^{\sss (\ell)}_{i+1}
&\geq \frac{L^{\ell}_{i+1}/(\ell)}{1+L_{i+1}/T_i}\left(1+\frac{\log(1+L_{i+1}/T_i)}{\log T_i}\right)^{\alpha-1}\\
&=\frac{L^{\ell}_{i+1}}{\ell}\big(1+\mathcal O(L_{i+1}/T_i)\big)
\big(1+\mathcal O(L_{i+1}/(T_i\log T_i))\big)\\
&= \frac{L^{\ell}_{i+1}}{\ell}\big(1+\mathcal O(L_{i+1}/T_i)\big),
\end{align*}\end{linenomath}where the constants in the $\mathcal O$-terms can be chosen deterministically and independently of $i\geq I_1$ (we do not give explicit values for the constants involved). Finally, we have
\begin{linenomath}
\begin{align*}
R^{\sss (\ell)}_{i+1}
&\leq \frac{L^{\ell}_{i+1}}{\ell}\exp\left(\beta(\log T_i)^\alpha \Big[\Big(1+\frac{\log (1+L_{i+1}/T_i)}{\log T_i}\Big)^\alpha-1\Big]\right)\\
&\leq \frac{L^{\ell}_{i+1}}{\ell}\big(1+\mathcal O(L_{i+1}/(T_i\log^{1-\alpha} T_i))\big),
\end{align*}\end{linenomath}which concludes the proof if $\mu=\mu_1$.

\medskip
(2) If $\mu=\mu_2$, we have
\[R_{i+1}^{\sss (\ell)} =
\int_0^{L_{i+1}} u^{\ell-1} \Big(1+\frac{u}{T_i}\Big)^{\delta-1}
\exp\!\Big[\gamma T_i^{\delta}\Big(\Big(1+\frac{u}{T_i}\Big)^{\delta}-1\Big)\Big]
\,\mathrm du
\geq \Big(1+\frac{L_{i+1}}{T_i}\Big)^{\delta-1} \frac{L_{i+1}^{\ell}}{\ell},\]
because $\delta<1$, $\gamma\geq 0$, and $0\leq u\leq L_{i+1}$ in the integrand.
Therefore, for all $i\geq I_0$,
\[R_{i+1}^{\sss (\ell)}
\geq \Big(1-c_6\frac{L_{i+1}}{T_i}\Big) \frac{L_{i+1}^{\ell}}{\ell},\]
where $c_6 = \sup_{[0,1]} \frac1x(1-(1+x)^{\delta-1})$, because, by definition of $I_0$, $L_{i+1}/T_i\in [0,1]$ for all $i\geq I_0$. Similarly, for all $i\geq I_0$, we have
\[R_{i+1}^{\sss (\ell)}\leq 
\exp\!\Big[\gamma T_i^{\delta}\Big(\Big(1+\frac{L_{i+1}}{T_i}\Big)^{\delta}-1\Big)\Big]
\frac{L_{i+1}^{\ell}}{\ell}
\leq \exp\!\big(\gamma c_7 T_i^{\delta-1}L_{i+1}\big)
\frac{L_{i+1}^{\ell}}{\ell},\]
where $c_7 = \sup_{[0,1]}\frac1x((1+x)^\delta-1)$. Note that
$T_i^{\delta-1}L_i\to0$ almost surely when $i\to\infty$, and thus there exists $I_1>I_0$ such that, for all $i\geq I_1$, $\gamma c_7 T_i^{\delta-1}L_{i+1}\leq 1$, and thus
\[R_{i+1}^{\sss (\ell)}\leq \paren{1+c_8 \frac{L_{i+1}}{T_i^{1-\delta}}}\frac{L_{i+1}^\ell}{\ell},\]
where $c_8 = \sup_{[0,1]} \frac1x(\mathrm e^{\gamma c_7 x}-1)$,
which concludes the proof.
\end{proof}

{\cec If we plug the result of Lemma~\ref{lem:exp_Fi} into the sums of Step (c) of the plan of the proof (see Section~\ref{sub:plan_proof}), we get sums of the following form that need to be estimated. The idea behind the following proposition is that $T_i \sim i\mathbb E L$, by the strong law of large numbers, and thus, assuming that $\mathbb E L^a$ is finite, $\sum_{i=1}^n L_i^a/T_i \sim \mathbb E L^a \log n/\mathbb E L$ because $\sum_{i=1}^n \nicefrac1i \sim \log n$. Similarly, we have $\sum_{i=1}^n L_i^a/T_i^2 =\mathcal O(1)$ since $\sum_{i=1}^{\infty} \nicefrac1{i^2}<+\infty$.}
\begin{lemma}\label{lem:sums}
For all $a>0$ and $b\in\mathbb R$,
if $\mathbb E L^{2a}<\infty$, 
then, almost surely when $i\to\infty$,
\[\sum_{i=1}^n\frac{L_{i+1}^a(\log T_i)^b}{T_i}
= 
\begin{cases}
\frac{\mathbb E L^a}{(b+1)\mathbb EL}\,\log^{b+1} n + \mathcal O(1) &\text{ if } b\neq -1\\
\frac{\mathbb EL^a}{\mathbb EL}\log\log n+\mathcal O(1)&\text{ if }b=-1,
\end{cases}\]
and, for all $\ell>1$,
\[\sum_{i=1}^n\frac{L_{i+1}^{a}(\log T_i)^{b}}{T_i^{\ell}}
=\mathcal O(1).\]
Also, for all $a>0$ and $\ell\in(0,\nicefrac12]$, we have
\[\ell\sum_{i=1}^nL_{i+1}^{a}T_i^{\ell-1}
= \frac{\mathbb EL^a}{(\mathbb EL)^{1-\ell}} \,n^{\ell} + \mathcal O(\log n).\]
\end{lemma}

\begin{proof}
We know that (e.g.\ by the law of the iterated logarithm), almost surely,
\[\sup_{n\to\infty} \frac{|T_n - mn|}{\sqrt{n\log n}}<+\infty.\]
Therefore, almost surely when $n\to\infty$,
\begin{linenomath}
\begin{align*}
\sum_{i=1}^n \frac{L_{i+1}^{a}(\log T_i)^{b}}{T_i}
&= \sum_{i=1}^n \frac{L_{i+1}^{a}(\log (mi + \mathcal O(\sqrt {i\log i})))^{b}}{mi + \mathcal O(\sqrt {i\log i})}\\
&=\sum_{i=1}^n \frac{L_{i+1}^{a}\log^b (mi)}{mi}\,\left(1+\mathcal O\!\left(\sqrt{\frac{\log i}{i}}\right)\right),
\end{align*}\end{linenomath}where the constants in the $\mathcal O$-terms can be chosen uniformly in $i$.
We thus get
\begin{equation}\label{eq:somme}
\sum_{i=1}^n \frac{L_{i+1}^a(\log T_i)^b}{T_i}
=\sum_{i=1}^n \frac{L_{i+1}^a\log^b (mi)}{mi}
+ \mathcal O\!\left(\sum_{i=1}^n \frac{L_{i+1}\log^b(mi)\sqrt{\log i}}{mi^{\nicefrac32}}\right).
\end{equation}
{\cec $\bullet$ We first estimate of the expectation of the first term in the right-hand side of Equation~\eqref{eq:somme};} note that
\[\mathbb E \left[\sum_{i=1}^n \frac{L_{i+1}^a\log^b (mi)}{mi}\right]
= \mathbb EL^a \sum_{i=1}^n \frac{\log^b (mi)}{mi},\]
and, when $n\to\infty$, %\footnote{justify the approx of the sum by the integral}
\begin{linenomath}
\begin{align*}
\sum_{i=1}^n \frac{\log^b (mi)}{mi}
&=\int_1^n \frac{\log^b(mx)}{mx}\mathrm dx+\mathcal O(1)
=\frac1m\int_1^{mn} \frac{\log^b y}{y}\mathrm dy+\mathcal O(1)\\
&=\frac{1}{m(b+1)} \, \log^{b+1} (mn) + \mathcal O(1)
=\frac{1}{m(b+1)} \, \log^{b+1} n + \mathcal O(1),
\end{align*}\end{linenomath}if $b\neq -1$. If $b=-1$, we have
\[\sum_{i=1}^n \frac{\log^b (mi)}{mi}
=\sum_{i=1}^n \frac{1}{mi\log (mi)} =  \frac1m\log \log n+\mathcal O(1).\]
Thus, we have
\[\mathbb E \left[\sum_{i=1}^n \frac{L_{i+1}^a\log^b (mi)}{mi}\right]
= \begin{cases}
\frac{\mathbb E L^a}{(b+1)m}\log^{b+1} n + \mathcal O(1)&\text{ if }b\neq -1\\
\frac{\mathbb E L^a}{m}\log\log n+\mathcal O(1) & \text{ if }b=-1.
\end{cases}\]
{\cec $\bullet$ We now use martingale theory to show that the first term in the right-hand side of Equation~\eqref{eq:somme} is almost-surely asymptotically equivalent to its expectation.} Note that, since the $L_i's$ are independent,
\[\left(M_n := \sum_{i=1}^n \frac{(L_{i+1}^a-\mathbb E L^a)\log^b (mi)}{mi}\right)_{n\geq 1}\]
is a martingale, and, for all $n\geq 0$,
\[\mathbb E M_n^2
= \sum_{i=1}^n \frac{\mathtt{Var}(L^a)\log^{2b} (mi)}{(mi)^2}
\leq \sum_{i=1}^{\infty} \frac{\mathtt{Var}(L^a)\log^{2b} (mi)}{(mi)^2} <\infty.\]
Therefore, $M_n$ converges almost surely to an almost surely finite random variable,
implying that
\[
\sum_{i=1}^n \frac{L_{i+1}^{a}\log^b (mi)}{mi}
= \mathbb E \left[\sum_{i=1}^n \frac{L_{i+1}^a\log^b (mi)}{mi}\right]
+ \mathcal O(1).
\]
{\cec $\bullet$ We treat the second term in the right-hand side of Equation~\eqref{eq:somme} using similar arguments: first estimating its expectation and then showing using martingale arguments that it is almost surely equivalent to it.} Doing that, we get that, almost surely when $n\to\infty$,
\[\sum_{i=1}^n \frac{L_{i+1}\log^b(mi)\sqrt{\log i}}{mi^{\nicefrac32}}
=\sum_{i=1}^n \frac{\mathbb E [L_{i+1}]\log^b(mi)\sqrt{\log i}}{mi^{\nicefrac32}} + \mathcal O(1)
= \mathcal O(1).\]
{\cec $\bullet$ All these estimates, together with Equation~\eqref{eq:somme}, give}
\[\sum_{i=1}^n \frac{L_{i+1}^a\log^b (mi)}{mi}=
\begin{cases}
\frac{\mathbb E L^a}{(b+1)m}\log^{b+1} n + \mathcal O(1)
&\text{ if }b\neq -1\\
\frac{\mathbb E L^a}{m}\log \log n + \mathcal O(1) & \text{ if }b=-1,
\end{cases}
\]
as claimed. 
The second and third statements of Lemma~\ref{lem:sums} can be proved similarly.
\end{proof}

{\cec The following lemma corresponds to proving Step (c) in the plan of the proof (see Section~\ref{sub:plan_proof}): we use martingale arguments to prove that the following sums are asymptotically equivalent to their expectations conditionally on $\bs L$, and then use Lemmas~\ref{lem:exp_Fi} and~\ref{lem:sums} to estimate these expectations.}
\begin{lemma}\label{lem:sumsF_iW_i}
Fix $a\geq 1$, and assume that $\mathbb E L^{2(a+1)}<+\infty$, 
then, conditionally on $\bs L$, $\bs L$-almost surely, we have
\[\sum_{i=1}^n \frac{F_i^{a-1} W_i}{S_i}=
\begin{cases}
\frac{\mathbb E L^a}{a\mathbb E L}\, s(n) + o\big(\sqrt{s(n)}\big) & \text{ if }\mu=\mu_1\\
\frac{\mathbb E L^a}{a(\mathbb E L)^{1-\delta}}\, s(n) + o\big(\sqrt{s(n)}\big) & \text{ if }\mu=\mu_2 \text{ and } \delta\in[0,\nicefrac12],
\end{cases}
\]
and,
\[\sum_{i=1}^n \frac{F_i^{a-1} W^2_i}{S^2_i}=\mathcal O(1).\]
\end{lemma}

\begin{proof}
Note that
\[\mathbb E_{\bs L}\left[\sum_{i=1}^n \frac{F_i^{a-1}W_i^b}{S_i^b}\right]
=\sum_{i=1}^n \frac{\mathbb E_{\bs L}\big[F_i^{a-1}\big]W_i^b}{S_i^b}.
%=\sum_{i=1}^n \frac{\mu(T_i)R_i^a W_i^{b-1}}{S_i^b},
\]
We also recall that, for all $i\geq 1$,
\[W_i = \int_{T_{i-1}}^{T_i} \mu
\quad\text{ and}\quad
S_i = \int_0^{T_i} \mu.\]
{\cec We first consider the expectations of these sums conditionally on $\bs L$, and then show that the sums are asymptotically equivalent to their expectations using martingale arguments.} % and then the conditional variance. 

\medskip
{\cec $\bullet$ To estimate the expectations (conditionally on $\bs L$) of these sums, 
we treat the different memory kernels separately.}

(1.1) If $\mu=\mu_1$ and $\beta=0$. If we let $\tilde\alpha  = \max\{\alpha, 1\}$, then Lemma~\ref{lem:exp_Fi} implies that, for all $b\in \{1,2\}$,
\begin{linenomath}
\begin{align}
\sum_{i=1}^n \frac{\mathbb E_{\bs L}[F_i^{a-1}]W_i^b}{S_i^b}
=& \sum_{i=1}^n \frac{\alpha(\log T_{i-1})^{\alpha-1}L_i^aW_i^{b-1}}{aT_{i-1}S_i^b}
+ \mathcal O\!\left(\sum_{i=1}^n  \frac{(\log T_{i-1})^{\alpha+\tilde\alpha-2}L_i^{a+1}W_i^{b-1}}{T_{i-1}^2S_i^b}\right)\nonumber\\
=&\sum_{i=1}^n \frac{\alpha(\log T_{i-1})^{\alpha-1}L_i^a(\log^\alpha T_i-\log^\alpha T_{i-1})^{b-1}}{aT_{i-1}\log^{\alpha b} T_i}\nonumber\\
&+ \mathcal O\!\left(\sum_{i=1}^n  \frac{(\log T_{i-1})^{\alpha+\tilde\alpha-2}L_i^{a+1}(\log^\alpha T_i-\log^\alpha T_{i-1})^{b-1}}{T_{i-1}^2 \log^{\alpha b}T_i}\right),
\label{eq:b}
\end{align}\end{linenomath}where we have used that $S_i = \log^\alpha T_i$ and $W_i = \log^\alpha T_i - \log^\alpha T_{i-1}$.
Note that, for all $i$ sufficiently large, we have $L_i/T_{i-1}\in[0,1]$, and thus
\[1\geq \left(\frac{\log T_{i-1}}{\log T_i}\right)^\alpha
= \left(\frac{\log T_{i-1}}{\log T_{i-1} + \log (1+L_i/T_{i-1})}\right)^{\!\!\alpha}
\geq \left(\frac{\log T_{i-1}}{\log T_{i-1} + c_1 L_i/T_{i-1}}\right)^{\!\!\alpha},\]
where $c_1 = \sup_{[0,1]} \frac1x \log (1+x)$. 
This implies
\[1\geq \left(\frac{\log T_{i-1}}{\log T_i}\right)^{\!\!\alpha}
\geq \left(\frac{1}{1 + \frac{c_1 L_i}{T_{i-1}\log T_{i-1}}}\right)^{\!\!\alpha}
\geq 1-c_2c_1\frac{L_i}{T_{i-1}\log T_{i-1}},\]
where $c_2 = \sup_{[0,1]} \frac1x(1-(1+x)^{-\alpha})$,
and, finally,
\[\log^\alpha T_i - \log^\alpha T_{i-1}
= \log^\alpha T_i \left(1-\left(\frac{\log T_{i-1}}{\log T_i}\right)^{\!\!\alpha}\right)
\leq c_1c_2\,\frac{L_i\log^\alpha T_i }{T_{i-1}\log T_{i-1}}.\]
Therefore, we get that
\begin{linenomath}
\begin{align*}
\sum_{i=1}^n \frac{\mathbb E_{\bs L}[F_i^{a-1}]W_i}{S_i}
&=\sum_{i=1}^n \frac{\alpha L_i^a}{aT_{i-1}\log T_{i-1}}
+\mathcal O\!\left(\sum_{i=1}^n\frac{L_i^{a+1}}{T_{i-1}^2\log T_{i-1}}\right)
+ \mathcal O\!\left(\sum_{i=1}^n  \frac{(\log T_{i-1})^{\tilde\alpha-2}L_i^{a+1}}{T_{i-1}^2}\right)\\
&=\sum_{i=1}^n \frac{\alpha L_i^a}{aT_{i-1}\log T_{i-1}}
+ \mathcal O\!\left(\sum_{i=1}^n  \frac{(\log T_{i-1})^{\tilde\alpha-2}L_i^{a+1}}{T_{i-1}^2}\right),
\end{align*}
\end{linenomath}because $\tilde \alpha -2\geq -1$, by definition.
We also have that (take $b=2$ in Equation~\eqref{eq:b})
\[\sum_{i=1}^n \frac{\mathbb E_{\bs L}[F_i^{a-1}]W^2_i}{S^2_i}
= \mathcal O\!\left(\sum_{i=1}^n\frac{L_i^{a+1}}{T_{i-1}^2\log^2 T_{i-1}}\right).\]
Using Lemma~\ref{lem:sums}, and since $s(n) = \alpha\log\log n$ in this case (see~Equation~\eqref{eq:def_s}), we deduce that
\[\sum_{i=1}^n \frac{\mathbb E_{\bs L}[F_i^{a-1}]W_i}{S_i}
=\frac{\mathbb E L^a}{a\mathbb E L}\,\alpha\log\log n
+ \mathcal O(1),\]
and
\[\sum_{i=1}^n \frac{\mathbb E_{\bs L}[F_i^{a-1}]W^2_i}{S^2_i}
= \mathcal O(1).\]
%as claimed (

\medskip
(1.2) If $\mu = \mu_1$ and $\beta\neq 0$ 
%, 
%and if we denote by $\tilde\alpha = \max\{1,\alpha\}$, 
%then 
Lemma~\ref{lem:exp_Fi} implies that, for all $b\in\{1,2\}$,
\begin{linenomath}
\begin{align*}
\sum_{i=1}^n  \frac{\mathbb E_{\bs L}\big[F_i^{a-1}\big]W_i^b}{S_i^b}
=& \sum_{i=1}^n \frac{\mu(T_{i-1})L_i^aW_i^{b-1}}{aS_i^b}
+ \mathcal O\!\left(\sum_{i=1}^n  \frac{\mu(T_{i-1})L_i^{a+1}W_i^{b-1}}{T_{i-1}S^b_i}\right)\\
=& \sum_{i=1}^n \frac{\alpha (\log T_{i-1})^{\alpha-1} L_i^a \mathtt e^{\beta(\log^\alpha T_{i-1} - \log^\alpha T_i)}}{aT_{i-1}\mathrm e^{\beta(b-1) \log^\alpha T_i}}
\left(\mathtt e^{\beta\log^\alpha T_i}-\mathtt e^{\beta\log^\alpha T_{i-1}}\right)^{b-1}\\
&+ \mathcal O\!\left(\sum_{i=1}^n  \frac{(\log T_{i-1})^{\alpha-1}L_i^{a+1}\mathtt e^{\beta(\log^\alpha T_{i-1} - \log^\alpha T_i)}}{T_{i-1}^2\mathrm e^{\beta(b-1) \log^\alpha T_i}}
\left(\mathtt e^{\beta\log^\alpha T_i}-\mathtt e^{\beta\log^\alpha T_{i-1}}\right)^{b-1}\right),
\end{align*}\end{linenomath}where we have used that $S_i = \mathrm e^{\beta \log^\alpha T_i}$
and $W_i = \mathrm e^{\beta \log^\alpha T_i}-\mathrm e^{\beta \log^\alpha T_{i-1}}$.
Note that, when $i\to\infty$,
\begin{linenomath}
\begin{align}
\exp\big(\beta(\log^\alpha T_{i-1}-\log^\alpha T_{i})\big)
&=\exp\big(\beta(\log^\alpha T_{i-1}-\log^\alpha (T_{i-1}+L_i))\big)\notag\\
&= \exp\left(\beta\left(\log^\alpha T_{i-1} - \log^\alpha(T_{i-1})\Big(1+\frac{\log(1+L_{i}/T_{i-1})}{\log T_{i-1}}\Big)^\alpha\right)\right)\notag\\
%&= (1+o(1))\exp\Big(-\frac{\alpha L_{i+1}}{T_i(\log T_i)^{1-\alpha}}\Big) 
&= \left(1+\mathcal O\!\left(\frac{L_{i}(\log T_{i-1})^{\alpha-1}}{T_{i-1}}\right)\right),
\label{eq:W_i}
\end{align}\end{linenomath}and thus
\begin{linenomath}
\begin{align*}
\mathtt e^{\beta\log^\alpha T_i}-\mathtt e^{\beta\log^\alpha T_{i-1}}
= 
\mathcal O\!\left(\frac{L_{i}(\log T_{i-1})^{\alpha-1}\mathtt e^{\beta\log^\alpha T_i}}{T_{i-1}}\right)
\end{align*}\end{linenomath}where the constants in the $\mathcal O$-terms can be chosen deterministic and independent of~$i$.
Thus, we get
\[\sum_{i=1}^n \frac{\mathbb E_{\bs L} [F_i^a]W_i}{S_i}
=\sum_{i=1}^n \frac{\alpha (\log T_{i-1})^{\alpha-1} L_i^a}{aT_{i-1}}
+ \mathcal O\!\left(\sum_{i=1}^n  \frac{(\log T_{i-1})^{\alpha-1}L_i^{a+1}}{T_{i-1}^2}\right),
\]
and
\[\sum_{i=1}^n \frac{\mathbb E_{\bs L} [F_i^a]W^2_i}{S^2_i}
=\mathcal O\left(\sum_{i=1}^n \frac{L_i^{a+1}\log^{2\alpha-2} T_{i-1}}{T_{i-1}^2}\right).\]
Applying Lemma~\ref{lem:sums}, we thus get that
\[\mathbb E_{\bs L}\left[\sum_{i=1}^n \frac{F_i^{a-1}W_i}{S_i}\right]
=\sum_{i=1}^n \frac{W_i\mathbb E_{\bs L} F_i^{a-1}}{S_i}
= \frac{\mathbb E L^a}{a\mathbb E L} \log^\alpha n + \mathcal O(1),
\]
and
\[\sum_{i=1}^n \frac{\mathbb E_{\bs L} [F_i^a]W^2_i}{S^2_i}
=\mathcal O(1),\]
as claimed (recall that $s(n) = \log^\alpha n$ in this case).

\medskip
(2) If $\mu = \mu_2$, then, Lemma~\ref{lem:exp_Fi} implies that, for all $b\in\{1,2\}$,
\begin{linenomath}
\begin{align*}
\sum_{i=1}^n  \frac{\mathbb E_{\bs L}[F_i^{a-1}]W^b_i}{S^b_i}
=& 
\sum_{i=1}^n 
\frac{\gamma\delta T_{i-1}^{\delta-1}\mathrm e^{\gamma T_{i-1}^\delta} L_i^a 
(\mathrm e^{\gamma T_i^\delta} - \mathrm e^{\gamma T_{i-1}^\delta})^{b-1}}
{a\mathrm e^{\gamma b T_i^\delta}}\\
&+ \mathcal O\!\left(
\sum_{i=1}^n 
\frac{T_{i-1}^{2(\delta-1)}\mathrm e^{\gamma T_{i-1}^\delta} L_i^{a+1}
(\mathrm e^{\gamma b T_i^\delta} - \mathrm e^{\gamma b T_{i-1}^\delta})^{b-1}}
{\mathrm e^{\gamma b T_i^\delta}}
\right) 
\end{align*}\end{linenomath}Note that
\[T_{i-1}^\delta - T_i^\delta
= T_{i-1}^\delta \Big(1 - \Big(1+\frac{L_i}{T_{i-1}}\Big)^{\!\delta}\Big),
 \]
implying that, for all $i$ large enough such that $L_i/T_{i-1}\in[0,1]$, we have
\[-\delta c_1 T_{i-1}^{\delta-1}L_i\leq T_{i-1}^\delta - T_i^\delta
\leq -\delta c_2 T_{i-1}^{\delta-1}L_i,\]
where $c_1 = \sup_{[0,1]} \frac1x((1+x)^\delta-1)$ and $c_2 = \inf_{[0,1]} \frac1x((1+x)^\delta-1)$.
Therefore, we get
\[\exp\big(\gamma (T_{i-1}^\delta-T_i^\delta)\big)
=\exp\big(-\mathcal O(T_{i-1}^{\delta-1}L_i)\big)
= 1 - \mathcal O(T_{i-1}^{\delta-1}L_i),\]
and
\[\mathrm e^{\gamma T_i^\delta} - \mathrm e^{\gamma T_{i-1}^\delta}
= \mathcal O\!\left(T_{i-1}^{\delta-1}L_i\mathrm e^{\gamma T_i^\delta}\right),\]
where the constants in the $\mathcal O$-terms can be chosen independent of~$i$.
Therefore, we get
\begin{linenomath}
\begin{align*}
\sum_{i=1}^n  \frac{\mathbb E_{\bs L}[F_i^{a-1}]W_i}{S_i}
&= \sum_{i=1}^n \frac{\gamma\delta T_{i-1}^{\delta-1} L_i^a}{a}
+ \mathcal O\!\left(\sum_{i=1}^n T_{i-1}^{2(\delta-1)} L_i^{a+1}\right)\\
&= \frac{\mathbb E L^a}{a(\mathbb E L)^{1-\delta}}\, \gamma n^\delta
+ \mathcal O(\log n),
\end{align*}\end{linenomath}and
\[\sum_{i=1}^n  \frac{\mathbb E_{\bs L}[F_i^{a-1}]W^2_i}{S^2_i}
= \mathcal O\!\!\left(\sum_{i=1}^n L_i^{a+1}T_{i-1}^{2(\delta-1)}\right)
=
\begin{cases}
\mathcal O(1) & \text{ if }\delta <\nicefrac12\\
\mathcal O(\log n) & \text{ if }\delta = \nicefrac12,
\end{cases}\]
where we have used Lemma~\ref{lem:sums}.

$\bullet$ We have thus proved that, $\bs L$-almost surely when $n\to\infty$,
\begin{equation}\label{eq:used_end_2}
\mathbb E_{\bs L}\left[\sum_{i=1}^n \frac{F_i^{a-1}W_i}{S_i}\right]=
\begin{cases}
\frac{\mathbb E L^a}{a\mathbb E L}\, s(n) + o\big(\sqrt{s(n)}\big) & \text{ if }\mu=\mu_1\\
\frac{\mathbb E L^a}{a(\mathbb E L)^{1-\delta}}\, s(n) + o\big(\sqrt{s(n)}\big) & \text{ if }\mu=\mu_2,
\end{cases}
\end{equation}
and
\begin{equation}\label{eq:used_end}
\mathbb E_{\bs L}\left[\sum_{i=1}^n \frac{F_i^{a-1}W^2_i}{S^2_i}\right]=
\begin{cases}
\mathcal O(1) & \text{ if }\mu=\mu_1 \text{ or }\mu=\mu_2 \text{ and }\delta <\nicefrac12\\
\mathcal O(\log n) & \text{ if }\mu=\mu_2\text{ and }\delta = \nicefrac12.
\end{cases}
\end{equation}

{\cec $\bullet$ We now use martingale arguments to prove that the sums are almost surely asymptotically equivalent to their expectations conditionally on $\bs L$.}
Note that, conditionally on $\bs L$,
\[\left(M_n := \sum_{i=1}^n \frac{W^b_i\big(F_i^{a-1}-\mathbb E_{\bs L}[F_i^{a-1}]\big)}{S^b_i}\right)_{\!n\geq 1}
\]
is a martingale. For all $n\geq 0$, we have {\cec (see, e.g.,~\cite[Proposition~1.3.5]{Duflo} for the definition of the increasing process of a square-integrable martingale)
\[\langle M\rangle_n = 
\sum_{i=1}^n \frac{W^{2b}_i\,\mathtt{Var}_{\bs L}(F_i^{a-1})}{S^{2b}_i}
\leq \sum_{i=1}^n \frac{W^{2b}_i\,L_i^{2a-2}}{S^{2b}_i},\]
because, $\bs L$-almost surely, $F_i\leq L_i$ for all $i\geq 1$, by definition (see Equation~\eqref{eq:distr_F}).}

\medskip
(1.1) Let us first assume that $\mu = \mu_1$ and $\beta=0$: we have
\[\sum_{i=1}^n \frac{L_i^{2a-2}W^{2b}_i}{S^{2b}_i}
= \sum_{i=1}^n L_i^{2a-2} \left(\frac{\log^{\alpha}T_i-\log^\alpha T_{i-1}}{\log^\alpha T_i}\right)^{\!\!{2b}}
= \sum_{i=1}^n L_i^{2a-2} \left(1-\frac{\log^\alpha T_{i-1}}{\log^\alpha T_i}\right)^{\!\!{2b}}.\]
Note that
\[
\frac{\log T_{i-1}}{\log T_i}
= \frac{1}{1+\frac{\log (1+L_i/T_{i-1})}{\log T_{i-1}}}
= 1- \mathcal O\left(\frac{\log (1+L_i/T_{i-1})}{\log T_{i-1}}\right)
= 1- \mathcal O\left(\frac{L_i}{T_{i-1}\log T_{i-1}}\right),
\]
where the constant in the $\mathcal O$-term can be chosen independent of~$i$.
Thus, using Lemma~\ref{lem:sums}, we get that
\[\sum_{i=1}^n \frac{L_i^{2a-2}W^{2b}_i}{S^{2b}_i}
=\mathcal O\left(\sum_{i=1}^n 
\frac{L_i^{2a+2b-2}}{T^2_{i-1}\log^2 T_{i-1}}\right)
= \mathcal O(1),\]
since we have assumed that $\mathbb E L^{2a+2}<+\infty$ (and also, $b\in\{1,2\}$). 
The cases (1.2) and (2) can be treated similarly, and we get that,
{\cec if $b=1$, then
\begin{equation}\label{eq:crochet}
\langle M\rangle_n
= \begin{cases}
\mathcal  O(1) & \text{ if } \mu=\mu_1\text{ or }\mu = \mu_2 \text{ and }\delta<\nicefrac12 \\
\mathcal O(\log n) & \text{ if } \mu = \mu_2 \text{ and }\delta=\nicefrac12 \
\end{cases},
\end{equation}
and, if $b=2$, then
\begin{equation}\label{eq:crochet2}
\langle M\rangle_n = \mathcal O(1).
\end{equation}
The law of large numbers for martingales (see, e.g.~\cite[Theorem 1.3.15]{Duflo}) states that 
\begin{enumerate}[(a)]
\item if $\langle M\rangle_n = \mathcal O(1)$ almost surely, then $M_n$ converges almost surely to a finite random variable; 
\item if $\langle M \rangle_n \to\infty$ almost surely, then $M_n = o(\langle M\rangle_n)$ almost surely.
\end{enumerate}
Equation~\eqref{eq:crochet} together with Equation~\eqref{eq:used_end_2} thus gives that, almost surely when $n\to\infty$,
\[
\sum_{i=1}^n \frac{F_i^{a-1}W_i}{S_i}
= \mathbb E_{\bs L}\left[\sum_{i=1}^n \frac{F_i^{a-1}W_i}{S_i}\right] + o\big(\langle M\rangle_n\big)\\
= \begin{cases}
\frac{\mathbb E L^a}{a\mathbb E L}\, s(n) + o\big(\sqrt{s(n)}\big) & \text{ if }\mu=\mu_1\\
\frac{\mathbb E L^a}{a(\mathbb E L)^{1-\delta}}\, s(n) + o\big(\sqrt{s(n)}\big) & \text{ if }\mu=\mu_2,
\end{cases}
\]
because $\langle M\rangle_n = o\big(s(n)\big)$ in all cases.
Similarly, Equation~\eqref{eq:crochet2} together with Equation~\eqref{eq:used_end} gives that, almost surely when $n\to\infty$,
\[\sum_{i=1}^n \frac{F_i^{a-1}W_i}{S^2_i}
= \mathbb E_{\bs L}\left[\sum_{i=1}^n \frac{F_i^{a-1}W_i}{S^2_i}\right]
+ o(1),\]
which concludes the proof.
}
\end{proof}

\subsection{Proof of Proposition~\ref{prop:marginals}}\label{sec:proof_marginals}
{\cec Using Lemmas~\ref{lem:exp_Fi}, \ref{lem:sums}, and \ref{lem:sumsF_iW_i}, we can now follow the plan of Section~\ref{sub:plan_proof} to prove Proposition~\ref{prop:marginals}.}
Recall that, given $\bs W$ (or, equivalently, $\bs L$), 
$(\bs 1_{\nu_i\preccurlyeq u_n})_{i\leq n}$ are independent Bernoulli random variables with respective parameters $W_i/S_i$, where $S_i = \sum_{j=1}^i W_j$ (see~Proposition~\ref{prop:Dobrow}). 
Also note that, by construction, the random variables $(\bs 1_{\nu_i\preccurlyeq u_n})_{i\leq n}$ are independent of $\bs F$; therefore, we have
%Using that $\bs F$ and $\bs L$ are independent, we get that
\[\mathbb E_{\bs L,\bs F} \Phi(u_n) 
= \sum_{i=1}^n \frac{F_i W_i}{S_i}=\hat \kappa_2 s(n) + o\big(\sqrt{s(n)}\big),\]
in probability when $n\to\infty$, where we have applied Lemma~\ref{lem:sumsF_iW_i} (which applies because $\mathbb E L^6<+\infty$, by assumption); we recall that 
\[\hat\kappa_2=\frac{\mathbb EL^2}{2\mathbb E L}\text{ if }\mu=\mu_1
\quad\text{ and }\quad
\hat\kappa_2=\frac{\mathbb EL^2}{2(\mathbb E L)^{1-\delta}}\text{ if }\mu=\mu_2.\]
Similarly, we have
\[\mathrm{Var}_{\bs L, \bs F}(\Phi(u_n))
= \sum_{i=1}^n \,\frac{F_i^2W_i}{S_i}\left(1-\frac{W_i}{S_i}\right)
= \sum_{i=1}^n \,\frac{F_i^2W_i}{S_i} 
- \sum_{i=1}^n \,\frac{F_i^2W^2_i}{S^2_i}
=\hat\kappa_3 s(n) + o(s(n)),
\]
by Lemma~\ref{lem:sumsF_iW_i} (which applies because $\mathbb E L^8<+\infty$, by assumption); see Theorem~\ref{th:sums_along_branches} for the definition of $\hat \kappa_3$. 
%implies that
%\[\sum_{i=1}^n \,\frac{F_i^2W_i}{S_i} = \hat\kappa_3 s(n) + o(s(n)),\]
%in probability when $n\to\infty$. One can also prove that
%\[\sum_{i=1}^n \,\frac{F_i^2W^2_i}{S^2_i} = \mathcal O(1),\]
%in probability when $n\to\infty$ (we leave the details to the reader).
%Thus, in probability when $n\to\infty$, we have
%\[\mathrm{Var}_{\bs L, \bs F}(\Phi(u_n))
%= \hat\kappa_3 s(n) + o(s(n)).\]
We can apply Lindeberg's central limit theorem to deduce Proposition~\ref{prop:marginals}.

\subsection{Proof of Theorem~\ref{th:sums_along_branches}}\label{sub:last_proof}
{\cec Proposition~\ref{prop:marginals} gives convergence of the marginals in~Theorem~\ref{th:sums_along_branches}. To get joint convergence, we need to prove that the correlation between $\Phi(u_n)$ and $\Phi(v_n)$ is negligible in front of $\sqrt{s(n)}$ so that $\Phi(u_n)/\sqrt{s(n)}$ and $\Phi(v_n)/\sqrt{s(n)}$ are asymptotically independent. This is true because the last common ancestor of $u_n$ and $v_n$ is ``high'' in the tree (i.e.\ close to the root): more precisely, its height converges in distribution to a finite random variable when $n\to+\infty$; this is stated in Lemma~\ref{lem:K}.}

For every node $\nu = \bar\nu i \in \mathcal T_n$ ($\bar\nu \in \mathcal T_n$, and $i\geq 1$),
we denote by $\mathcal T_n^{\ell}(\nu)$ the subtree of $\mathcal T_n$ rooted at $\nu$, 
and by $\mathcal T_n^{r}(\nu)$ subtree of $\mathcal T_n$ consisting of $\bar\nu$ and all the subtrees rooted at right-siblings $\bar\nu j$ ($j>i$) of $\nu$.
We informally call $\mathcal T_n^\ell(\nu)$ the ``left'' subtree of $\nu$, 
and $\mathcal T_n^r(\nu)$ its ``right'' subtree (see Figure~\ref{fig:LCA_Notations}).

\begin{definition}\label{df:LCU}
Given two nodes $u$ and $v$ of $\mathcal T_n$, 
we denote their last common ancestor $u \wedge v$.
Two children of $u \wedge v$ are respective ancestors of $u$ and $v$;
the smallest in the lexicographic order is called the {\it last common uncle} of $u$ and $v$, 
and denoted by $u_n\star v_n$.
\end{definition}

\begin{proposition}\label{lem:LCA}
Let $\bs W = (W_i)_{i\geq 1}$ be a sequence of i.i.d.\ random variables, 
and $\mathcal T_n$ be the $n$-node $\bs W$-{\sc wrrt}.
Let $u_n$ and $v_n$ be two nodes taken independently at random in $\mathcal T_n$ 
with probability proportional to the weights: for all $1\leq i\leq n$,
\[\mathbb P_{\bs W}(u_n = \nu_i) = \mathbb P_{\bs W}(v_n = \nu_i)  = \frac{W_i}{S_n}.\]
%Let us denote by $\kappa_n:= u_n\wedge v_n$ their last common ancestor.
%If $u_n$ is in $\mathcal T_n^{\ell}(\kappa_n)$, then, 
%we denote $\mathcal T(u_n) = \mathcal T_n^{\ell}(\kappa_n)$ and $\mathcal T(v_n) = \mathcal T_n^{r}(\kappa_n)$, the converse otherwise. 
%(These notations are illustrated in Figure~\ref{fig:LCA_Notations}.)
%Finally, let $\tau(u_n)$ (resp.\ $\tau(v_n)$) be the number of nodes of $\mathcal T(u_n)$ (resp.\ $\mathcal T(v_n)$).
Then, conditionally on $\bs W$, we have
$u_n\wedge v_n
\xrightarrow{\sss d} \kappa$ when $n\to\infty$,
where $\kappa$ is a (finite) random element of $\{1, 2, 3, \ldots \}^*$.
\end{proposition}

\begin{figure}
\includegraphics[width=.25\textwidth]{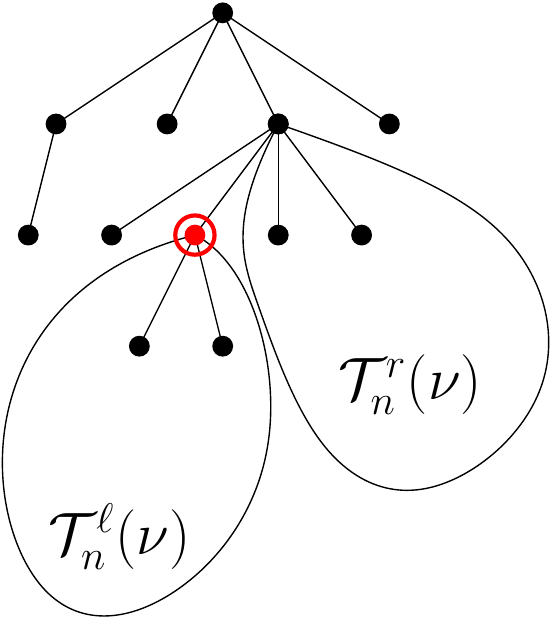}
\caption{%{\it Left}: 
A node $\nu$ (in red and marked by a circle) and its ``left'' and ``right'' subtrees, $\mathcal T_n^\ell(\nu)$ and $\mathcal T_n^r(\nu)$.}
\label{fig:LCA_Notations}
\end{figure}

The first step in proving Proposition~\ref{lem:LCA} is to prove that $|u_n\wedge v_n|$ converges in distribution to an almost surely finite random variable.
\begin{lemma}\label{lem:K}
Under the assumptions of Theorem~\ref{th:sums_along_branches}, assuming that $\delta\in(0,\nicefrac12)$ if $\mu=\mu_2$, we have that, asymptotically when $n\to\infty$,
\begin{enumerate}[(a)]
\item $|u_n\wedge v_n| \xrightarrow{\sss d} \theta$, where $\theta$ is almost surely finite, and
\item $\Phi(u_n\wedge v_n)\xrightarrow{\sss d} \Theta$, where $\Theta$ is almost surely finite.
\end{enumerate}
\end{lemma}

\begin{proof}
We treat the proofs of {\it (a)} and {\it (b)} at once; to do so, we introduce the notation $\hat\Phi$ to be understood as $\hat\Phi = |\,\cdot\,|$ in the proof of {\it (a)} and $\hat\Phi = \Phi$ in the proof of {\it (b)}.

Recall that we denote by $\nu_1, \ldots, \nu_n$ the nodes of the $n$-node $\bs W$-{\sc wrrt} in their order of addition in the tree. Also, for all~$i$, $\xi(i)$ is the index of the parent of node $\nu_i$. We reason conditionally on $\bs L$, and thus $\bs W$, and recall that $S_n = \sum_{i=1}^n W_i$ for all $n\geq 1$.

Recall that, for all $n\geq 1$, $(\mathcal T_n, u_n, v_n)$ is equal in distribution to $(\tilde {\mathcal T}_n, \tilde u_n, \tilde v_n)$ (defined just before Proposition~\ref{prop:Dobrow}). 
%We obtain that $\hat\Phi(\tilde u_n \wedge \tilde v_n) = \hat\Phi(u_n \wedge v_n)$ in distribution for all $n\geq 1$.
%Conditionally on $u_n\neq \nu_n$ and $v_{n}\neq \nu_n$, $(u_n, v_n)$ is distributed like $(u_{n-1},v_{n-1})$ in $\mathcal T_{n-1}$.
%Conditionally on $u_n = \nu_n$ and $v_n \neq \nu_n$, 
%\begin{itemize} 
%\item the last common ancestor of $u_n=\nu_n$ and $v_n$ is also the last common ancestor of $\nu_{\xi(n)}$ and~$v_n$,
%\item $\nu_{\xi(n)}$ is distributed as $u_{n-1}$ in $\mathcal T_{n-1}$
%\item $v_n$ (conditioned on being different from $\nu_n$) is distributed as $v_{n-1}$ in $\mathcal T_{n-1}$, and is independent of $\nu_{\xi(n)}$.
%\end{itemize}
%Let us prove the preceding assertions by induction. 
By the construction of $(\tilde {\mathcal T}_n, \tilde u_n, \tilde v_n)_{n\geq 1}$, 
we have
\begin{linenomath}
\begin{align*}
\hat\Phi(\tilde u_n \wedge  \tilde v_n) 
=& \bs 1_{B_n=B'_n=1} \hat\Phi(\nu_n) 
+ \bs 1_{B_n=1\neq B'_n} \hat\Phi( \tilde u_{n-1} \wedge  \tilde v_{n-1})\\
&
+ \bs 1_{B_n \neq 1 = B'_n} \hat\Phi( \tilde u_{n-1} \wedge  \tilde v_{n-1}) 
+ \bs 1_{B_n, B'_n \neq 1} \hat\Phi( \tilde u_{n-1}\wedge  \tilde v_{n-1})\\
=& 
\big(1-\bs 1_{B_n = B'_n = 1}\big) 
\hat\Phi( \tilde u_{n-1}\wedge \tilde v_{n-1}) 
+ \bs 1_{B_n = B'_n = 1}\hat\Phi(\nu_n).
\end{align*}\end{linenomath}Thus, if we let $K_n := \hat\Phi(\tilde u_n \wedge \tilde v_n)$ 
for all integers~$n$, we have
\begin{equation}\label{eq:rec_Kn}
K_n \overset{\sss d}{=} K_{n-1} + \bs 1_{B_n = B'_n = 1}
\big(\hat\Phi(\nu_n)-K_{n-1}\big).
\end{equation}
Note that, for all $i\geq 1$,
\[\mathbb P_{\bs L}\Big(\bs 1_{B_i = B'_i = 1}\big(\hat\Phi(\nu_i)-K_{i-1}\big) \neq 0\Big)
\leq \mathbb P_{\bs L}\Big(\bs 1_{B_i = B'_i = 1} \neq 0\Big)
= \Big(\frac{W_i}{S_i}\Big)^2.\]
Lemma~\ref{lem:sumsF_iW_i} (with $a=1$) implies that, if $\mu=\mu_1$ or $\mu=\mu_2$ and $\delta\in(0,\nicefrac12)$, then, $\bs L$-almost surely, we have
\begin{equation}\label{eq:delta_crucial}
\sum_{i\geq 1} \frac{W_i^2}{S_i^2} = \mathcal O(1).
\end{equation}
Therefore, by the Borel-Cantelli lemma, we can infer that, almost surely,
\begin{linenomath}
\begin{align*}
\Theta&:=\sum_{i=2}^{\infty} \bs 1_{B_i = B'_i = 1}\big(\Phi(\nu_i)-K_{i-1}\big)<+\infty,\\
\theta&:=\sum_{i=2}^{\infty} \bs 1_{B_i = B'_i = 1}\big(|\nu_i|-K_{i-1}\big)<+\infty,\
\end{align*}\end{linenomath}and that $|\tilde u_n\wedge \tilde v_n|\xrightarrow{\sss d}\theta$ and $\Phi(\tilde u_n\wedge \tilde v_n)\xrightarrow{\sss d}\Theta$ when $n$ goes to infinity. This implies the claim since $(\mathcal T_n, u_n, v_n)$ is equal in distribution to $(\tilde {\mathcal T}_n, \tilde u_n, \tilde v_n)$ .
%Also note that, for all $n\geq 0$,
%\[\mathrm{Var}(K_n) =
%\mathrm{Var}\bigg(\sum_{i=1}^n \bs 1_{B_i = B'_i = 1}
%\big(\hat\Phi(\nu_i)-K_{i-1}\bigg)
%\leq \mathbb E\bigg[\bigg(\sum_{i=1}^n \bs 1_{B_i = B'_i = 1}
%\big(\hat\Phi(\nu_i)-K_{i-1}\big)\bigg)^{\!\!2}\bigg].\qedhere\]
\end{proof}

Note that Equation~\eqref{eq:delta_crucial} is no longer true if $\mu=\mu_2$ 
and $\delta = \nicefrac12$, which is why Proposition~\ref{prop:marginals} holds for $\delta=\nicefrac12$ while the joint convergence of Theorem~\ref{th:sums_along_branches} is open in that case (and maybe does not holds).

\begin{proof}[Proof of Proposition~\ref{lem:LCA}]
Let us denote by $k=k(\nu)$ the random integer such 
that $\nu$ is an internal node of $\mathcal T_k$ but not of $\mathcal T_{k-1}$.
%Note that, almost surely, $k(\nu)<+\infty$. 
%Let us prove this result by induction on~$|\nu|$.
%The result obviously holds for the root since $k(\varnothing) = 1$ almost surely. 
%Now assume that the statement holds for all nodes of height at most $|\nu|-1$. It holds in particular for the parent of $\nu$, which we denote by $\bar\nu$. We denote by $\mathtt{deg}_n(\bar\nu)$ the out-degree (number of children) of node $\bar\nu$ in $\mathcal T_n$.
%Conditionally on $k(\bar\nu)$, using Equation~\eqref{eq:W/S}, 
%Lemma~\ref{lem:sums} and the fact that $k(\bar\nu)<\infty$ almost surely, 
%we have
%\[\mathbb E\mathtt{deg}_n(\bar\nu)
%= W_{k(\bar\nu)}\sum_{i=k(\bar\nu)}^n \frac{1}{S_i} 
%= \frac{W_{k(\bar\nu)}}{\mathbb E W} \, \log n + \mathcal O(1),\]
%and
%\[\mathtt{Var}\big(\mathtt{deg}_n(\bar\nu)\big)
%= \frac{W_{k(\bar\nu)}}{\mathbb E W} \, \log n + \mathcal O(1),\]
%implying that $\mathtt{deg}_n(\bar\nu)$ goes to infinity in probability by Markov's inequality. Since $\mathtt{deg}_n(\bar\nu)$ is almost surely non-decreasing, we can infer that it converges to infinity almost surely, and that, therefore,
%$k(\nu)$ is almost surely finite.
Let us denote by $W^*_n(\nu)$ 
the sum of the weights of the internal 
nodes of $\mathcal T_n$ descending from~$\nu$ 
(including $\nu$ itself). 

{\cec $\bullet$ We show that $W^*_n(\nu)$ is asymptotically almost surely proportional to $S_n$, the total weight of the $n$-node tree.}
Conditionally on $k(\nu)<+\infty$, 
we have $W^*_{k(\nu)}(\nu) = W_{k(\nu)}$, 
and, for all $n\geq k(\nu)$,
\[\mathbb E_{\bs L}\left[W^*_{n+1}(\nu)\Big|\mc F_n\right]
= W^*_n(\nu) + \frac{W^*_n(\nu)}{S_n}\, W_{n+1},\]
because at time $n+1$, the weight of the subtree rooted at $\nu$ increases by $W_{n+1}$
with probability $W^*_n(\nu)/S_n$, and stays unchanged otherwise.
Therefore, conditionally on $k(\nu)<+\infty$ and on $\bs L$,
\[\prod_{i=1}^{n-1}\left(1+\frac{W_{i+1}}{S_i}\right)^{-1} W^*_n(\nu) \quad (n\geq k(\nu))\]
is a positive martingale, 
and thus converges almost surely to a random variable $W^*(\nu)$.
Note that
\[ \prod_{i=1}^{n-1}\left(1+\frac{W_{i+1}}{S_i}\right)
= \prod_{i=1}^{n-1}\frac{S_i + W_{i+1}}{S_i}
= \prod_{i=1}^{n-1}\frac{S_{i+1}}{S_i} = \frac{S_n}{W_1},\]
implying that, almost surely when $n\to\infty$,
\begin{equation}\label{eq:cv_W*}
\frac{W^*_n(\nu)}{S_n} \to \frac{W^*(\nu)}{W_1} =: \hat W^*(\nu).
\end{equation}
%Let us denote by $\tau_n^{\ell}(\nu)$ and $\tau_n^{r}(\nu)$ 
%the respective sizes of $\mathcal T_n^{\ell}(\nu)$ and $\mathcal T_n^r(\nu)$, 
%the ``left'' and ``right'' subtree of node $\nu$ in $\mathcal T_n$.
%Fix $x, y\geq 0$, and $\nu\in\{1, 2, \ldots\}^*$.
%We denote by $\bar\nu$ the parent of $\nu$, 
%and by $\mathrm{last}(\nu)\geq 1$ the rank of $\nu$ among its siblings; 
%in other words, $\nu = \bar\nu \mathrm{last}(\nu)$.
{\cec $\bullet$} Conditionally on $\bs L$ (and thus $\bs W$), we have
\[
\mathbb P_{\bs L}
\left(%\frac{\tau(u_n)}{n}\geq x, \frac{\tau(v_n)}{n}\geq y, 
u_n\star v_n = \nu\right)
= \frac{W_n^*(\nu)\big(W_n^*(\bar\nu) - \sum_{j=1}^{\mathrm{last}(\nu)} W_n^*(\bar\nu j)\big)}{S_n^2},
%\left(\bs 1_{\tau_n^{\ell}(\nu)\geq xn}\bs 1_{\tau_n^{r}(\nu)\geq yn}
%+ \bs 1_{\tau_n^{r}(\nu)\geq xn}\bs 1_{\tau_n^{\ell}(\nu)\geq yn}\right).
\]
where $\bar\nu$ is the parent of node $\nu$, and $\mathrm{last}(\nu)\geq 1$ the rank of $\nu$ among its siblings.
%Note that, given $\tau_n(\nu)$ and $W^*_n(\nu)$, $\tau_{n+1}(\nu) = \tau_{n}(\nu)+1$ with probability $W^*_n(\nu)/S_n$ and $\tau_{n+1}(\nu) = \tau_{n}(\nu)$ otherwise. In other words, given $k(\nu)<+\infty$,
%\[Q_n:=\tau_n(\nu) - \sum_{i=1}^{n-1}\frac{W^*_i(\nu)}{S_i}\quad (n\geq k(\nu))\]
%is a martingale for the filtration $\mc H_n:= \sag{\tau_n(\nu), W^*_n(\nu)}$.
%Note that
%\[\langle Q\rangle_n 
%= \sum_{i=1}^n \left(\tau_{n+1}(\nu)-\tau_n(\nu)+\frac{W^*_n(\nu)}{S_n}\right)^{\!\!2}
%\leq 4n,\]
%almost surely since $|\tau_{n+1}(\nu)-\tau_n(\nu)|\leq 1$ and $|W^*_n(\nu)/S_n|\leq 1$ both almost surely.
%Thus, $(\langle Q\rangle_n/n)_{n\geq 1}$ is bounded almost surely, implying that, by the law of large number for martingales (cf. second law of large numbers in \cite{MR1485774}), that $Q_n/n \to 0$ almost surely when $n\to\infty$.
%Since $W_n^*(\nu)/S_n \to \hat W_n(\nu)$, we get that, almost surely when $n\to\infty$,
%\[\tau_n(\nu) = n\hat W^*(\nu) + o(n).\]
Therefore, 
we have that, almost surely, for all $\nu\in\mathcal T_n$,
\[
\mathbb P_{\bs L}
\left(u_n\star v_n = \nu\right)
\to
\hat W^*(\nu) \big(\hat W^*(\bar \nu)-\sum_{j=1}^{\mathrm{last}(\nu)} \hat W^*(\bar\nu j)\big)
\]
implying that, conditionally on $\bs L$, we have,
\begin{equation}\label{eq:cv_sizes}
u_n\star v_n
\xrightarrow{\sss d} \eta,
\end{equation}
in distribution when $n$ goes to infinity,
where $\eta$ is a random word of $\{1,2, 3, \ldots\}^*$.
%and
%\[\{\alpha,\beta\} = \big\{\hat W^*(\eta),\hat W^*(\bar\eta)-\sum_{j=1}^{\mathrm{last}(\eta)}\hat W^*(\bar\eta j)\big\},\]
%where $\mathrm{last}(\eta)$ is the last element of the word $\eta$.
%Note that, 
%if $\hat W^*(\nu)=0$ or  $\hat W^*(\bar \nu)-\sum_{j=1}^{\mathrm{last}(\nu)} \hat W^*(\bar\nu j)=0$, then 
%\[\mathbb P_{\bs L}
%\left(\frac{\tau(u_n)}{n}\geq x, \frac{\tau(v_n)}{n}\geq y, u_n\star v_n = \nu\right) \to 0,\]
%implying that $\alpha$ and $\beta$ are positive.
The result follows from the fact that $u_n\wedge v_n$ is the parent of $u_n\star v_n$; we define $\kappa$ as the parent of the random node $\eta$, it is almost surely finite by Lemma~\ref{lem:K}{\it (a)}.
\end{proof}

\begin{proof}[Proof of Theorem~\ref{th:sums_along_branches}]
In the proof, we condition on the sequences $\bs L$  (and thus on $\bs W$) and $\bs F$. 
Recall that, $(\mathcal T_n, u_n, v_n)$ is equal in distribution to the triple $(\tilde{\mathcal T}_n, \tilde u_n, \tilde v_n)$ defined just before Proposition~\ref{prop:2Dobrow}. We denote by
\[K=
\max\set{k\geq 1: B_i = B_i'};
\]
we know from Lemma~\ref{lem:LCA} that $K<+\infty$ almost surely.
Also, by definition (see Proposition~\ref{prop:2Dobrow}), 
we have that, for all $i\geq K$,
\[\big(\bs 1_{\nu_i\prec u_n}, \bs 1_{\nu_i\prec v_n}\big)
= (B_i, B'_i).\]
%By definition of 
%; in particular, $\imf{{\bf 1}}{\nu_i\prec u_n}$ and $\imf{{\bf 1}}{\nu_i\prec v_n}$ are eventually equal to the Bernoulli random variables $B_i$ and $B'_i$, say from the random (but finite) index
%\[
%K=\min\set{k\geq 1: \imf{{\bf 1}}{\nu_i\prec u_n}=B_i\text{ and }\imf{{\bf 1}}{\nu_i\prec v_n}=B'_i\text{ for all }i\geq k}. 
%\]
%Indeed, $K$ also coincides with the first index after which $B$ and $B'$ are not non-zero simultaneously, which we proved to be finite in Lemma \ref{lem:K}. 
Consider $C_n=\sum_{i=1}^n F_iB_i$ and $C'_n=\sum_{i=1}^n F_iB'_i$;
conditionally on $\bs L$ and $\bs F$, $C$ and $C'$ are independent (inhomogeneous) random walks. 
We have
\begin{equation}\label{eq:C_n}
\left(\frac{\Phi(\tilde u_n)-b_{\hat \kappa_2s(n)}}{\sqrt{\hat\kappa_3s(n)}},
\frac{\Phi(\tilde v_n)-b_{\hat \kappa_2s(n)}}{\sqrt{\hat\kappa_3s(n)}}\right)
= \left(\frac{C_n-b_{\hat \kappa_2s(n)}}{\sqrt{\hat\kappa_3s(n)}} 
+ \frac{\Phi(\tilde u_n)-C_n}{\sqrt{\hat\kappa_3 s(n)}},
\frac{C'_n-b_{\hat \kappa_2s(n)}}{\sqrt{\hat\kappa_3s(n)}}
+ \frac{\Phi(\tilde v_n)-C'_n}{\sqrt{\hat\kappa_3 s(n)}}\right).
\end{equation}
Note that,
\begin{equation}\label{eq:approx_C_n}
\max\left\{\frac{\abs{C_n-\imf{\Phi}{\tilde u_n}}}{\sqrt{s(n)}},
\frac{\abs{C'_n-\imf{\Phi}{\tilde v_n}}}{\sqrt{s(n)}}\right\}
\leq \frac{2\sum_{i=1}^{K\wedge n}F_i}{\sqrt{s(n)}}\to 0
\end{equation}
almost surely when $n\to\infty$. %, and similarly for $\abs{C'_n-\imf{\Phi}{\tilde v_n}}$.
Furthermore, for all $x,y\in\mathbb R$, by independence of $C_n$ and $C'_n$ conditionally on~$\bs F$, we get
\begin{align*}
\mathbb P_{\bs L,\bs F}
\left(\frac{C_n-b_{\hat \kappa_2s(n)}}{\sqrt{\hat\kappa_3s(n)}} \geq x,
\frac{C'_n-b_{\hat \kappa_2s(n)}}{\sqrt{\hat\kappa_3s(n)}}\geq y\right)
= \mathbb P_{\bs L,\bs F}
\left(\frac{C_n-b_{\hat \kappa_2s(n)}}{\sqrt{\hat\kappa_3s(n)}} \geq x\right)
 \mathbb P_{\bs L,\bs F}\left(\frac{C'_n-b_{\hat \kappa_2s(n)}}{\sqrt{\hat\kappa_3s(n)}}\geq y\right)&\\
= 
\mathbb P_{\bs L,\bs F}
\left(\frac{\Phi(\tilde u_n)-b_{\hat \kappa_2s(n)}}{\sqrt{\hat\kappa_3s(n)}} 
+\frac{\Phi(\tilde u_n)-C_n}{\sqrt{\hat\kappa_3 s(n)}}\geq x\right)
\mathbb P_{\bs L,\bs F}\left(\frac{\Phi(\tilde v_n)-b_{\hat \kappa_2s(n)}}{\sqrt{\hat\kappa_3s(n)}}
+\frac{\Phi(\tilde v_n)-C'_n}{\sqrt{\hat\kappa_3 s(n)}}\geq y\right),&
 \end{align*}
 where we have used Equation~\eqref{eq:C_n}.
 Thus, using Proposition~\ref{prop:marginals}, Equation~\eqref{eq:approx_C_n} and Slutzky's lemma, we get that
\[\mathbb P_{\bs L,\bs F}
\left(\frac{C_n-b_{\hat \kappa_2s(n)}}{\sqrt{\hat\kappa_3s(n)}} \geq x,
\frac{C'_n-b_{\hat \kappa_2s(n)}}{\sqrt{\hat\kappa_3s(n)}}\geq y\right)
\to \mathbb P(\Lambda_1\geq x)\mathbb P(\Lambda_2\geq y),
\]
where $\Lambda_1$ and $\Lambda_2$ are two standard Gaussian random variables;
in other words,
\[\left(\frac{C_n-b_{\hat \kappa_2s(n)}}{\sqrt{\hat\kappa_3s(n)}},
\frac{C'_n-b_{\hat \kappa_2s(n)}}{\sqrt{\hat\kappa_3s(n)}}\right)
\xrightarrow{\sss d} (\Lambda_1,\Lambda_2),\]
where $\Lambda_1$ and $\Lambda_2$ are independent.
Using Slutzky's lemma again, together with Equations~\eqref{eq:C_n} and~\eqref{eq:approx_C_n}, we get that
\[\left(\frac{\Phi(\tilde u_n)-b_{\hat \kappa_2s(n)}}{\sqrt{\hat\kappa_3s(n)}},
\frac{\Phi(\tilde v_n)-b_{\hat \kappa_2s(n)}}{\sqrt{\hat\kappa_3s(n)}}\right)
\xrightarrow{\sss d} (\Lambda_1,\Lambda_2),\]
which concludes the proof since $(\mathcal T_n, u_n, v_n) \overset{\sss d}= (\tilde{\mathcal T}_n, \tilde u_n, \tilde v_n)$.
%almost surely as $n\to\infty$, and similarly for $C'_n$ and $\Phi(\tilde v_n)$. 
%$\imf{\Phi}{u_n}$ satisfies the central limit theorem  of Proposition \ref{prop:marginals}. 
%Therefore, $(C_n, C_n')$ admits a CLT with independent Gaussian random variables in the limit and, by the above display, the same holds for $(\imf{\Phi}{u_n},\imf{\Phi}{v_n})$. 
\end{proof}

\subsection{Sketch proof of Theorem \ref{th:profile}}
Following exactly the proof of Theorem~\ref{th:sums_along_branches}, replacing $\Phi(\nu)$ by 
\[\Psi(\nu) = |\nu| + 1 = \sum_{i=1}^n \bs 1_{\nu_i \preccurlyeq \nu},\]
one can prove that:
If $u_n$ and $v_n$ be two nodes in the $n$-node {\sc wrrt},
chosen independently at random with probability proportional to the weights,
then, conditionally on $\bs L$,
asymptotically when $n\to\infty$,
\begin{equation}\label{eq:Psi}
\left(\frac{\Psi(u_n) - s(n)}{\sqrt{s(n)}}, \frac{\Psi(v_n) - s(n)}{\sqrt{s(n)}}\right) 
\xrightarrow{\sss d} (\Lambda_1, \Lambda_2),
\end{equation}
where $\Lambda_1$ and $\Lambda_2$ are two independent standard Gaussian random variables.
The result follows by applying the standard \cite[Lemma~3.1]{MR3629870}
%~\cite[Lemma~??]{MM17}
 since the profile is the (random) 
distribution of $|u_n|$ and $|v_n|$ given $\mathcal T_n$.

\bibliography{memory.bib}

\newcommand{\etalchar}[1]{$^{#1}$}
\providecommand{\bysame}{\leavevmode\hbox to3em{\hrulefill}\thinspace}
\providecommand{\MR}{\relax\ifhmode\unskip\space\fi MR }
\providecommand{\MRhref}[2]{%
  \href{http://www.ams.org/mathscinet-getitem?mr=#1}{#2}
}
\providecommand{\href}[2]{#2}
\begin{thebibliography}{FCBGM17}

\bibitem[AdAC{\etalchar{+}}14]{Alves}
G.~A. Alves, J.~M. de~Ara{\'u}jo, J.~C. Cressoni, L~R da~Silva, M.~A.~A.
  da~Silva, and G.~M. Viswanathan, \emph{Superdiffusion driven by exponentially
  decaying memory}, Journal of Statistical Mechanics: Theory and Experiment
  \textbf{2014} (2014), no.~4, P04026.

\bibitem[BB01]{BB01}
G.~Bianconi and A.-L. Barab{\'a}si, \emph{Bose-{E}instein condensation in
  complex networks}, Physical review letters \textbf{86} (2001), no.~24, 5632.

\bibitem[BB16]{PhysRevE.94.052134}
E.~Baur and J.~Bertoin, \emph{Elephant random walks and their connection to
  {P}\'olya-type urns}, Phys. Rev. E \textbf{94} (2016), 052134.

\bibitem[BEM17]{doi:10.1088/1742-5468/aa58b6}
D.~Boyer, M.~R. Evans, and S.~N. Majumdar, \emph{Long time scaling behaviour
  for diffusion with resetting and memory}, Journal of Statistical Mechanics:
  Theory and Experiment \textbf{2017} (2017), no.~2, 023208.

\bibitem[BFS92]{MR1251994}
F.~Bergeron, Ph. Flajolet, and B.~Salvy, \emph{Varieties of increasing trees},
  C{AAP} '92 ({R}ennes, 1992), Lecture Notes in Comput. Sci., vol. 581,
  Springer, Berlin, 1992, pp.~24--48. \MR{1251994}

\bibitem[BK16]{MR3543905}
J.~Bertoin and I.~Kortchemski, \emph{Self-similar scaling limits of {M}arkov
  chains on the positive integers}, Ann. Appl. Probab. \textbf{26} (2016),
  no.~4, 2556--2595. \MR{3543905}

\bibitem[BL17]{2017arXiv170907345B}
B.~{Bercu} and L.~{Laulin}, \emph{{On the multi-dimensional elephant random
  walk}}, ArXiv e-prints (2017).

\bibitem[BP16]{PhysRevE.93.022103}
D.~Boyer and I.~Pineda, \emph{Slow l\'evy flights}, Phys. Rev. E \textbf{93}
  (2016), 022103.

\bibitem[BSS14]{PhysRevLett.112.240601}
D.~Boyer and C.~Solis-Salas, \emph{Random walks with preferential relocations
  to places visited in the past and their application to biology}, Phys. Rev.
  Lett. \textbf{112} (2014), 240601.

\bibitem[{Bus}17]{2017arXiv171005671B}
S.~{Businger}, \emph{{The shark random walk (L\'evy flight with memory)}},
  ArXiv e-prints (2017), 1710.05671.

\bibitem[BV05]{BV05}
K.~A. Borovkov and V.~A. Vatutin, \emph{Trees with product-form random
  weights}, Discrete Mathematics and Theoretical Computer Science, Discrete
  Mathematics and Theoretical Computer Science, 2005, pp.~423--426.

\bibitem[BV06]{BV06}
\bysame, \emph{On the asymptotic behaviour of random recursive trees in random
  environments}, Advances in applied probability \textbf{38} (2006), no.~4,
  1047--1070.

\bibitem[BVHS99]{bertoin1999renewal}
J.~Bertoin, K.~Van~Harn, and F.~W. Steutel, \emph{Renewal theory and level
  passage by subordinators}, Statistics \& Probability Letters \textbf{45}
  (1999), no.~1, 65--69.

\bibitem[CDJH01]{CDJH01}
B.~Chauvin, M.~Drmota, and J.~Jabbour-Hattab, \emph{The profile of binary
  search trees}, Annals of Applied Probability (2001), 1042--1062.

\bibitem[CGS17]{MR3652225}
C.~F. Coletti, R.~Gava, and G.~M. Sch\"utz, \emph{Central limit theorem and
  related results for the elephant random walk}, J. Math. Phys. \textbf{58}
  (2017), no.~5, 053303, 8. \MR{3652225}

\bibitem[CH14]{CH14}
Nicolas Curien and B{\'e}n{\'e}dicte Haas, \emph{Random trees constructed by
  aggregation}, arXiv preprint arXiv:1411.4255 (2014).

\bibitem[CKMR05]{CKMR05}
B.~Chauvin, T.~Klein, J.-F. Marckert, and A.~Rouault, \emph{Martingales and
  profile of binary search trees}, Electronic Journal of Probability
  \textbf{10} (2005), 420--435.

\bibitem[Dev88]{MR969872}
L.~Devroye, \emph{Applications of the theory of records in the study of random
  trees}, Acta Inform. \textbf{26} (1988), no.~1-2, 123--130. \MR{969872}

\bibitem[DG97]{DrmotaGittenberger97}
M.~Drmota and B.~Gittenberger, \emph{On the profile of random trees}, Random
  Structures and Algorithms \textbf{10} (1997), no.~4, 421--451.

\bibitem[Dob96]{MR1401472}
R.~P. Dobrow, \emph{On the distribution of distances in recursive trees}, J.
  Appl. Probab. \textbf{33} (1996), no.~3, 749--757. \MR{1401472}

\bibitem[Duf13]{Duflo}
Marie Duflo, \emph{Random iterative models}, vol.~34, Springer Science \&
  Business Media, 2013.

\bibitem[EM11]{EvansMajumdar11}
M.~R. Evans and S.~N. Majumdar, \emph{Diffusion with stochastic resetting},
  Physical review letters \textbf{106} (2011), no.~16, 160601.

\bibitem[EM14]{EvansMajumdar14}
\bysame, \emph{Diffusion with resetting in arbitrary spatial dimension},
  Journal of Physics A: Mathematical and Theoretical \textbf{47} (2014),
  no.~28, 285001.

\bibitem[FCBGM17]{FalconEtAl17}
A.~Falc{\'o}n-Cort{\'e}s, D.~Boyer, L.~Giuggioli, and S.~N. Majumdar,
  \emph{Localization transition induced by learning in random searches},
  Physical review letters \textbf{119} (2017), no.~14, 140603.

\bibitem[Haa17]{Haas}
B{\'e}n{\'e}dicte Haas, \emph{Asymptotics of heights in random trees
  constructed by aggregation}, Electronic Journal of Probability \textbf{22}
  (2017).

\bibitem[HI17]{HI17}
Ella Hiesmayr and {\"U}mit I{\c{s}}lak, \emph{Asymptotic results on hoppe trees
  and its variations}, arXiv preprint arXiv:1712.03572 (2017).

\bibitem[HM11]{MR2854770}
B.~Haas and G.~Miermont, \emph{Self-similar scaling limits of non-increasing
  {M}arkov chains}, Bernoulli \textbf{17} (2011), no.~4, 1217--1247.
  \MR{2854770}

\bibitem[Jan04]{MR2040966}
S.~Janson, \emph{Functional limit theorems for multitype branching processes
  and generalized {P}\'olya urns}, Stochastic Process. Appl. \textbf{110}
  (2004), no.~2, 177--245. \MR{2040966 (2005a:60134)}

\bibitem[Kat05]{Katona05}
Z.~Katona, \emph{Width of a scale-free tree}, Journal of Applied Probability
  \textbf{42} (2005), no.~3, 839--850.

\bibitem[KMS17]{KMS17}
Z.~Kabluchko, A.~Marynych, and H.~Sulzbach, \emph{General edgeworth expansions
  with applications to profiles of random trees}, The Annals of Applied
  Probability \textbf{27} (2017), no.~6, 3478--3524.

\bibitem[KW10]{MR2729386}
M.~Kuba and S.~Wagner, \emph{On the distribution of depths in increasing
  trees}, Electron. J. Combin. \textbf{17} (2010), no.~1, Research Paper 137,
  9. \MR{2729386}

\bibitem[Lam62]{MR0138128}
J.~Lamperti, \emph{Semi-stable stochastic processes}, Trans. Amer. Math. Soc.
  \textbf{104} (1962), 62--78. \MR{0138128}

\bibitem[MM17]{MR3629870}
C.~Mailler and J.-F. Marckert, \emph{Measure-valued {P}\'olya urn processes},
  Electron. J. Probab. \textbf{22} (2017), Paper No. 26, 33. \MR{3629870}

\bibitem[MVdS{\etalchar{+}}16]{Moura}
T.~R.~S. Moura, G.~M. Viswanathan, M.~A.~A. da~Silva, J.~C. Cressoni, and L.~R.
  da~Silva, \emph{Transient superdiffusion in random walks with a
  q-exponentially decaying memory profile}, Physica A: Statistical Mechanics
  and its Applications \textbf{453} (2016), 259--263.

\bibitem[Nev75]{MR0402915}
J.~Neveu, \emph{Discrete-parameter martingales}, revised ed., North-Holland
  Publishing Co., Amsterdam, 1975, Translated from the French by T. P. Speed,
  North-Holland Mathematical Library, Vol. 10. \MR{0402915 (53 \#6729)}

\bibitem[Pem07]{Pemantle07}
R.~Pemantle, \emph{A survey of random processes with reinforcement},
  Probability surveys \textbf{4} (2007), 1--79.

\bibitem[P{\'o}l30]{Polya}
G.~P{\'o}lya, \emph{Sur quelques points de la th{\'e}orie des
  probabilit{\'e}s}, Annales de l'Institut Henri Poincar{\'e} \textbf{1}
  (1930), no.~2, 117--161.

\bibitem[Roy88]{MR1013117}
H.~L. Royden, \emph{Real analysis}, third ed., Macmillan Publishing Company,
  New York, 1988. \MR{1013117}

\bibitem[Sch10]{Schopp10}
E.-M. Schopp, \emph{A functional limit theorem for the profile of $b$-ary
  trees}, The Annals of Applied Probability \textbf{20} (2010), no.~3,
  907--950.

\bibitem[SSH{\etalchar{+}}08]{SimsEtAl08}
D.~W. Sims, E.~J. Southall, N.~E. Humphries, G.~C. Hays, C.~J.~A. Bradshaw,
  J.~W. Pitchford, A.~James, M.~Z. Ahmed, A.~S. Brierley, and M.~A. Hindell,
  \emph{Scaling laws of marine predator search behaviour}, Nature \textbf{451}
  (2008), no.~7182, 1098.

\bibitem[ST04]{PhysRevE.70.045101}
G.~M. Sch\"utz and S.~Trimper, \emph{Elephants can always remember: Exact
  long-range memory effects in a non-markovian random walk}, Phys. Rev. E
  \textbf{70} (2004), 045101.

\bibitem[Sul08]{Sulzbach08}
H.~Sulzbach, \emph{A functional limit law for the profile of plane-oriented
  recursive trees.}, DMTCS Proceedings \textbf{0} (2008), no.~1.

\bibitem[vdH01]{vdH01}
R.~van~der Hofstad, \emph{The lace expansion approach to ballistic behaviour
  for one-dimensional weakly self-avoiding walks}, Probability theory and
  related fields \textbf{119} (2001), no.~3, 311--349.

\bibitem[Wil91]{MR1155402}
D.~Williams, \emph{Probability with martingales}, Cambridge Mathematical
  Textbooks, Cambridge University Press, Cambridge, 1991. \MR{1155402
  (93d:60002)}

\end{thebibliography}
\bibliographystyle{amsalpha}

\end{document}